\documentclass[preprint, 12pt]{elsarticle}
\usepackage{color,amsthm,amsmath,amsfonts,amssymb,mathtools,bm,framed,fixltx2e,hyperref}
\usepackage{bbm}
\usepackage{lineno,hyperref}
\usepackage{mathrsfs}
\usepackage{amscd}
\usepackage{enumerate}
\usepackage{graphicx}
\usepackage{url}

-\marginparwidth \oddsidemargin -9mm \evensidemargin \oddsidemargin
\advance\hoffset by 5mm

\newtheorem{thm}{Theorem}
\newtheorem*{thm*}{Theorem}
\newtheorem{cor}{Corollary}
\newtheorem{lem}{Lemma}
\newtheorem{prop}{Proposition}
\theoremstyle{definition}
\newtheorem{defn}{Definition}
\theoremstyle{remark}
\newtheorem{rem}{Remark}

\newtheorem{assum}{Assumption}

\biboptions{sort&compress}

\numberwithin{equation}{section} \numberwithin{lem}{section}
\numberwithin{thm}{section} \numberwithin{prop}{section}
\numberwithin{cor}{section} \numberwithin{rem}{section}\numberwithin{hyp}{section}
\allowdisplaybreaks

\modulolinenumbers[5]

\begin{document}

\begin{frontmatter}

\title{Mild solutions to the Cauchy problem for time-space fractional Keller-Segel-Navier-Stokes system}

\author[mymainaddress]{Ziwen Jiang}
\ead{jiangziwensss@163.com}
\author[mymainaddress]{Lizhen Wang\corref{mycorrespondingauthor}}

\cortext[mycorrespondingauthor]{Corresponding author}
\ead{wanglizhen@nwu.edu.cn}

\address[mymainaddress]{Center for Nonlinear Studies, School of Mathematics, Northwest University, Xi'an, Shaanxi Province, 710127, China}

\begin{abstract}
This paper investigates the Cauchy problem of the time-space fractional Keller-Segel-Navier-Stokes model in $ \mathbb{R}^d~( d\geq 2)$ which can describe both memory effect and L\'{e}vy process of the system. The local existence and global existence in Lebesgue space are obtained by means of Banach fixed point theorem and Banach implicit function theorem, respectively. In addition, the regularities of local and global mild solutions are improved in fractional homogeneous Sobolev spaces. Furthermore, some properties of mild solutions including mass conservation, decay estimates, stability and self-similarity are established.
\end{abstract}

\end{frontmatter}

\section{Introduction}

In this paper, for $0<\beta<1$, $1<\alpha\leq2$, $\gamma\geq0$ and $d\geq 2$, we consider the following Cauchy problem for time-space fractional Keller-Segel-Navier-Stokes system:
\begin{equation}\label{chemo-fluid-eq}
\left\{
  \begin{aligned}
  &^c_0D_t^\beta n+(-\Delta)^{\alpha/2} n+u\cdot \nabla n+\nabla\cdot(n \nabla v)=0  && \mbox{in }  \mathbb{R}^d\times(0,\infty),\\
  &^c_0D_t^\beta v+(-\Delta)^{\alpha/2} v+u\cdot \nabla v+\gamma v-n=0  && \mbox{in }  \mathbb{R}^d\times(0,\infty),\\
  &^c_0D_t^\beta u+(-\Delta)^{\alpha/2} u+(u\cdot \nabla) u+\nabla\rho+n \nabla \phi =0  && \mbox{in }  \mathbb{R}^d\times(0,\infty),\\
  &\nabla\cdot u=0  && \mbox{in }  \mathbb{R}^d\times(0,\infty),\\
  &n|_{t=0}=n_0,~v|_{t=0}=v_0,~u|_{t=0}=u_0 && \mbox{in } \mathbb{R}^d,\\
  \end{aligned}
     \right.
\end{equation}
where $n=n(x,t)$ denotes the density of cells that diffuses in the manner of L\'{e}vy flight and is transported along the velocity field of the incompressible fluid, $v=v(x,t)$ indicates the concentration of chemical attractant, the vector function $u=u(x,t)=(u_1(x,t),u_2(x,t),\cdots,u_d(x,t))$ stands for the fluid velocity field, the scalar function $\rho=\rho(x,t)$ represents the pressure of the fluid, $\phi=\phi(x)$ is a given potential function considering the effects of external forces such as the gravitational potential generated by aggregation of cells onto the fluid and $n_0=n_0(x),v_0=v_0(x),u_0=(u_{1,0}(x),u_{2,0}(x),\cdots,u_{d,0}(x))$ are prescribed initial data. It can be seen from \eqref{chemo-fluid-eq} that chemotaxis and fluid are coupled through both the transport of the cells and the chemical $u\cdot\nabla n$, $u\cdot\nabla v$ by the fluid and the external force $-n \nabla \phi$ exerted on the fluid by the cells. $ ^c_0D_t^\beta$ is weak Caputo fractional derivative operator of order $\beta$ introduced in \cite{LJLiu1,LJLiu2} and $^c_0D_t^\beta w$ models memory effects in time. When function $w(t)$ is absolutely continuous in time, the definition of weak Caputo derivative is reduced to the following traditional form
\[
 ^c_0D_t^\beta w(t)=\frac{1}{\Gamma(1-\beta)}\int_0^t(t-s)^{-\beta}\partial_sw(s)\,ds,
\]
where $\Gamma$ is the Gamma function. According to Chapter V in \cite{Stein}, the nonlocal operator $(-\Delta)^{\alpha/2}$, known as the Laplacian of order $\frac{\alpha}{2}$, is given by the Fourier multiplier
$$(-\Delta)^{\alpha/2}\rho(x):=\mathcal{F}^{-1}\big(|\xi|^{\alpha}\hat{\rho}(\xi)\big)(x).$$

Fractional derivatives have been employed to describe the nonlocal effects of the considered equations in references \cite{SGSAAKOIM,Fracdiffeq2006,KDiethelm}. The motivation of using fractional diffusion instead of classical diffusion comes from the fact that organisms adopt a L\'{e}vy process instead of Brownian motion \cite{NBVCal,JBRGBelin,DCL,V1,V2} and  it is believed that the feeding behavior of some organisms is based on a L\'{e}vy process generated by the spatial-fractional diffusion operator $(-\Delta)^{\frac{\alpha}{2}}~(0<\alpha<2)$. There are some interesting works investigating partial differential equations with fractional Laplacian, see references \cite{CHENLIOU,Caffarelli2007,Caffarelli2008,Caffarelli2010,BW,Stanvazqez, LMZH}. Chen, Li and Ou \cite{CHENLIOU} solved an open problem posed by Lieb with the method of moving planes in an integral form and classified all the solutions of one type of integral equation equivalent to a family of semi-linear partial differnetial equations with fractional Laplacian. Caffarelli and Silvestre \cite{Caffarelli2007} obtained characterizations for general fractional powers of the Laplacian and other integro-differential operators. Lai, Miao and Zhang \cite{LMZH} constructed a global-time forward self-similar solution to the Navier-Stokes equations with the fractional diffusion $(-\Delta)^\alpha~(5/6<\alpha<1)$ for arbitrarily large self-similar initial data. At the macroscopic level, time fractional operators are associated with anomalous subdiffusion \cite{Fengliliuxu2018,LJLiu2,TALBIH,W,MBSB,MALCAV,SakaYama}. For example, in \cite{TALBIH}, Riemann-Liouville fractional derivative can model anomalous diffusion of continuous time random walks with power law waiting time. Allen, Caffarelli and Vasseur \cite{MALCAV} discussed the existence and regularity for weak solutions to porous medium equation with fractional time derivative of Caputo-type and  inverse fractional Laplacian operator. Sakamoto and Yamamoto \cite{SakaYama} established the unique existence of the weak solution and the asymptotic behavior of fractional diffusion-wave equations.

In the 1970s, Keller and Segel introduced the first mathematical model of chemotaxis describing the aggregation of slime mold amoebae due to attractive chemicals in \cite{EFKellerSe1,EFKellerSe2}. There are numerous references investigating the Keller-Segel models such as \cite{CalvezLC2008,LCBP2006,LucasCFFJCP,WinklerM2010,Biler1998,KOZONOYS2009,TNRSMU}. Recently, some results related to the fractional Keller-Segel equations were presented in \cite{BW,E,HL,LRZ,MBSB,Azevedo,ZLZ,LLW,ZWJiangLW,LJLiu1,ZMSD,MarioB}. Escudero \cite{E} constructed the global in time solutions for the fractional diffusion $(-\Delta)^\frac{\alpha}{2}~(1<\alpha\leq2)$. Azevedo, Cuevas and Henr\'{\i}quez \cite{Azevedo} managed to verify the existence and asymptotic behaviour for the time-fractional Keller-Segel model. Huang and Liu explored the uniqueness and stability of nonlocal Keller-Segel equations with fractional Laplacian in \cite{HL}. Li, Liu and Wang \cite{LLW} considered the existence, integrability, nonnegativity and blow up behaviors of the Cauchy problem of the fractional time-space generalized Keller-Segel equations. Jiang and Wang \cite{ZWJiangLW} investigated the global existence and mass conservation of weak solutions to the time-space fractional Keller-Segel equation.

In what follows, we present a brief development of the chemotaxis-fluid system and introduce some related consequences on it. First, the chemotaxis system with the effect of fluid, proposed by Tuval et al. to model the dynamics of populations of aerobic bacteria within the flow of an incompressible Newtonian fluid \cite{TuvalCDW}, can be written as
\begin{equation}\label{noche-flu}
\left\{
  \begin{aligned}
  &\partial_t n+u\cdot\nabla n=\Delta n-\nabla\cdot(n\chi(c) \nabla c)  && \mbox{in }  \mathbb{R}^d\times(0,\infty),\\
  &\partial_t c+u\cdot\nabla c=\Delta c-n \kappa(c)  && \mbox{in }  \mathbb{R}^d\times(0,\infty),\\
  &\partial_t u+(u\cdot\nabla) u=\Delta u-\nabla \rho -n\nabla \phi  && \mbox{in }  \mathbb{R}^d\times(0,\infty),\\
  &\nabla\cdot u=0 && \mbox{in }  \mathbb{R}^d\times(0,\infty).
  \end{aligned}
     \right.
\end{equation}
Here, the unknowns $n,c,u,\rho $ denote the cell density, oxygen concentration, velocity field and pressure of the fluid, respectively. $\chi(c)$ is the chemotactic sensitivity and $\kappa(c)$ is the consumption rate of the chemical by the cells. There are a large quantity of literature discussing the chemotaxis-fluid system \eqref{noche-flu}, see \cite{LYLY,DLMarkowich,LIULorz,Lorz2010,ChaeKangLee,MCHAEKLEE,Winkler2012,Winkler2014,ZhangLI} for instance. Duan, Lorz and Markowich \cite{DLMarkowich} obtained the global existence and rates of convergence on classical solutions of \eqref{noche-flu} near constant states. In \cite{LIULorz}, for the system \eqref{noche-flu} in two dimensional space, Liu and Lorz constructed the global existence of weak solutions for large data and for the chemotaxis-Stokes system with nonlinear diffusion for the cell density in three dimensional space, they proved the global existence of weak solutions. It was shown in \cite{ChaeKangLee} that in two dimensional space \eqref{noche-flu} admits a global classical solutions under some suitable assumptions. The global existence, stabilization and convergence rate of solutions to \eqref{noche-flu} on bounded domain $\Omega\in \mathbb{R}^2$ have been established in references \cite{Winkler2012,Winkler2014,ZhangLI}.

Next, \cite{TANZhou2018} discussed the following double chemotaxis model under the effect of the Navier-Stokes fluid:
\begin{equation}\label{che-flucvTAN}
\left\{
  \begin{aligned}
  &\partial_t n+u\cdot\nabla n=\Delta n-\nabla\cdot(n\chi(c) \nabla c)-\nabla\cdot(n\iota(v) \nabla v)+h(n)  && \mbox{in }\mathbb{R}^d\times(0,\infty),\\
  &\partial_t c+u\cdot\nabla c=\Delta c-n \kappa(c)  && \mbox{in }  \mathbb{R}^d\times(0,\infty),\\
  &\partial_t v+u\cdot\nabla v=\Delta v+g(v)n-\gamma v   && \mbox{in } \mathbb{R}^d\times(0,\infty),\\
  &\partial_t u+(u\cdot\nabla) u=\Delta u-\nabla \rho -n\nabla \phi  && \mbox{in } \mathbb{R}^d\times(0,\infty),\\
  &\nabla\cdot u=0 && \mbox{in }  \mathbb{R}^d\times(0,\infty),
  \end{aligned}
     \right.
\end{equation}
where the meaning of unknowns $(n,c,u,\rho)$ is same as the one expressed in \eqref{noche-flu}, while $v$ denotes the chemical concentration; $\chi(c)$ and $\iota(v)$ are the chemotactic sensitivities, $h(n)$ is the external source term usually taken as logical source $\rho n-\mu n^2$, $\kappa(c)$ stands for the consumption rate of oxygen, $g(v)$ is the consumption rate of chemical attractant by the cells, and the parameter $\gamma\geq0$ presents the decay rate of attractant. Tan and Zhou \cite{TANZhou2018,TanZhou} obtained the global existence and decay estimates of solutions to the double-type chemotaxis system. Taking $\chi(c)=\iota(v)=g(v)=1$, $h(n)=0$ and $\kappa(c)=c$ in \eqref{che-flucvTAN}, Kozono, Miura and Sugiyama presented the existence of global mild solutions with small initial data in the scaling invariant space using Banach implicit function theorem in \cite{Kozono2016}. Jiang et.al \cite{JLLZhou} established the existence and uniqueness of global mild solution to fractional chemotaxis-Navier-Stokes system. Recently, Azevedo et.al \cite{Azevedo2022} proved the existence of global mild solutions of time fractional Keller-Segel model coupled with the Navier-Stokes fluid with small critical initial data in Besov-Morrey spaces.

For some bioconvection processes, in which the signal substance is produced by the cells themselves, the Keller-Segel-Navier-Stokes model is considered as follows \cite{MWinkler2019}:
\begin{equation}\label{KS-NSwinkler}
\left\{
  \begin{aligned}
  &\partial_t n+u\cdot\nabla n=\Delta n-\nabla\cdot(n\nabla c)-\nabla\cdot(n \nabla v)+\rho n-\mu n^2  && \mbox{in }\mathbb{R}^d\times(0,\infty),\\
  &\partial_t v+u\cdot\nabla v=\Delta v+n-\gamma v   && \mbox{in } \mathbb{R}^d\times(0,\infty),\\
  &\partial_t u+(u\cdot\nabla) u=\Delta u-\nabla p -n\nabla \phi +f(x,t) && \mbox{in } \mathbb{R}^d\times(0,\infty),\\
  &\nabla\cdot u=0 && \mbox{in }  \mathbb{R}^d\times(0,\infty),
  \end{aligned}
     \right.
\end{equation}
where $(n,v,u,\rho)$ is shown as above. Some research related to equation \eqref{KS-NSwinkler} have been done in references \cite{MWinkler2019,KangKim,HuaZhang,BnBaCn}. Winkler \cite{MWinkler2019} constructed the global weak solutions for arbitrarily
large initial data to \eqref{KS-NSwinkler} and the large time behavior of these solutions under appropriate assumptions. Kang and Kim \cite{KangKim} obtained the generalized solutions for the Keller-Segel system with a degradation source coupled to Navier-Stokes equations in three dimensional space. Hua and Zhang \cite{HuaZhang} established the global well-posedness for the incompressible Keller-Segel-Navier-Stokes equations in $\mathbb{R}^3$.

The work of this present paper is inspired by Kozono, Miura and Sugiyama \cite{Kozono2016}. We consider the time-space fractional single type of chemotaxis model that the signal substance is produced by the cells themselves under Navier-Stokes fluids as shown in \eqref{chemo-fluid-eq}, covering numerous biologically relevant situations when cells actively use chemotaxis as a means of communication, rather than the double chemotaxis model in \cite{Kozono2016}. The main mathematical difficulty comes from the fractional diffusion operator and the weak Caputo fractional derivative operator which are nonlocal operators, so the argument of heat semigroup $e^{t\Delta}$ is infeasible for system \eqref{chemo-fluid-eq}. To overcome this difficulty, we replace the heat semigroup $e^{t\Delta}$ applied in \cite{Kozono2016} with the Mittag-Leffler operators $E_\beta(-t^\beta(-\Delta)^{\frac{\alpha}{2}})$ and $E_{\beta,\beta}(-t^\beta(-\Delta)^{\frac{\alpha}{2}})$. Our goal in this paper is to study the existence and some properties of mild solutions to \eqref{chemo-fluid-eq}.

The rest of this paper is organized as follows. We begin in Section 2 with some basic definitions and properties which will be used in the subsequent proof and the $L^p-L^q$ estimates of Mittag-Leffler operators are established. Section \ref{sec3} is dedicated to verifying the existence and uniqueness of mild solution in Lebesgue spaces $L^p(\mathbb{R}^d)$. On one hand, local existence and uniqueness to \eqref{chemo-fluid-eq} is verified by Banach fixed point theorem. On the other hand,  global existence and uniqueness of mild solution is obtained by Banach implicit function theorem utilized in \cite{Kozono2016}. Section \ref{sec4} is devoted to the existence of mild solution in fractional homogeneous Sobolev spaces $\dot{H}^{\mu,q}(\mathbb{R}^d)$. The mass conservation of local solution and decay estimates, stability of global solution are provided in Section \ref{sec5} and in addition, the self-similar solution is constructed whenever taking $\gamma=0$ in \eqref{chemo-fluid-eq}.

\begin{assum}\label{assum1}
For $d\geq 2$, assume that $\alpha,\beta$ and exponents $p,q,r$ satisfy one of the following conditions (I), (II), (III), (IV) and (V):
\begin{enumerate}[(I)]
\item
for any $1<\alpha\leq 2$, $0<\beta<1$,
\[
 ~~\frac{d}{2\alpha-2}<q\leq\frac{d}{\alpha-1},~~ \frac{d}{\alpha-1}<p<\frac{qd}{d-(\alpha-1)q},~~ \frac{d}{\alpha-1}<r<\frac{qd}{d-(\alpha-1)q};
\]
\item for any  $1<\alpha\leq\frac{3}{2}$, $0<\beta<1$,
\[
~~\frac{d}{\alpha-1}<q<\infty,~~~~\frac{d}{\alpha-1}<p<\infty,~~~~q\leq r<\infty;
\]
\item for any  $\frac{3}{2}<\alpha\leq2$, $0<\beta\leq\frac{\alpha}{3\alpha-3}$,
\[
~~\frac{d}{\alpha-1}<q<\infty,~~~~ \frac{d}{\alpha-1}<p<\infty,~~~~q\leq r<\infty;
\]
\item for any  $\frac{3}{2}<\alpha\leq2$, $\frac{\alpha}{3\alpha-3}<\beta<1$,
\[\frac{d}{\alpha-1}<q\leq\frac{d\beta}{(3\alpha-3)\beta-\alpha},~ ~~\frac{d}{\alpha-1}<p<\infty,~~~ q\leq r<\infty;\]
\item for any  $\frac{3}{2}<\alpha\leq2$, $\frac{\alpha}{3\alpha-3}<\beta<1$,
\[\resizebox{.9\hsize}{!}{$\frac{d\beta}{(3\alpha-3)\beta-\alpha}<q<\frac{2d\beta}{(3\alpha-3)\beta-\alpha},~\frac{d}{\alpha-1}<p<\frac{d\beta q}{\big((3\alpha-3)\beta-\alpha\big)q-d\beta},~q\leq r<\frac{d\beta q}{\big((3\alpha-3)\beta-\alpha\big)q-d\beta}.$}\]
\end{enumerate}
\end{assum}

\begin{assum}\label{assum2}
For $d\geq 2$, $\alpha,\beta,\mu$ and the exponents $p,q,r$ satisfy either one of the following conditions:
\begin{enumerate}[(I)]
\item Assume that $\alpha,\beta,\mu$ satisfy either one of the following conditions
\begin{enumerate}[(i)]
\item
$1<\alpha\leq \frac{5}{4}$, ~~~~$0<\beta<1$, ~~~~$0<\mu< \alpha-1$;
\item
$\frac{5}{4}<\alpha\leq2$, ~~~~$0<\beta\leq\frac{\alpha}{5\alpha-5}$,~~~~ $0<\mu<\alpha-1$;
\item
$\frac{5}{4}<\alpha\leq\frac{4}{3}$,~~~~ $\frac{\alpha}{5\alpha-5}<\beta<1$, ~~~~$0<\mu<\alpha-1$;
\item
$\frac{4}{3}<\alpha\leq2$, ~~~~$\frac{\alpha}{5\alpha-5}<\beta\leq\frac{\alpha}{4\alpha-4}$, ~~~~$0<\mu<\alpha-1$;
\item
$\frac{4}{3}<\alpha\leq\frac{3}{2}$, ~~~~$\frac{\alpha}{4\alpha-4}<\beta<1$,~~~~ $0<\mu\leq\frac{\alpha}{\beta}-(3\alpha-3)$;
\item
$\frac{3}{2}<\alpha\leq2$, ~~~~$\frac{\alpha}{4\alpha-4}<\beta\leq\frac{\alpha}{3\alpha-3}$, ~~~~$0<\mu\leq\frac{\alpha}{\beta}-(3\alpha-3)$;
\end{enumerate}
and $p,q,r$ satisfy either one of the following inequalities
\begin{enumerate}[(1)]
\item$\frac{d}{2\alpha-2-\mu}<q\leq\frac{d}{\alpha-1}$,~~~$\frac{d}{\alpha-1}<p<\frac{qd}{d-(\alpha-1-\mu)q}$,~~$\frac{d}{\alpha-1}< r<\frac{qd}{d-(\alpha-1)q}$;
\item$\frac{d}{\alpha-1}<q<\frac{d}{\alpha-1-\mu}$,~~~$\frac{d}{\alpha-1}<p<\frac{qd}{d-(\alpha-1-\mu)q}$, ~~~$q\leq r<\infty$;
\item$\frac{d}{\alpha-1-\mu}\leq q<\infty$,~~~$\frac{d}{\alpha-1}<p<\infty$, ~~~$q\leq r<\infty$.
\end{enumerate}
\item Assume that $\alpha,\beta,\mu$ satisfy either one of the following conditions
\begin{enumerate}[(i)]
\item
$1<\alpha\leq \frac{5}{4}$, ~~~~$0<\beta<1$,~~~~ $\alpha-1\leq\mu< 2\alpha-2$;
\item
$\frac{5}{4}<\alpha\leq2$,~~~~ $0<\beta\leq\frac{\alpha}{5\alpha-5}$,~~~~ $\alpha-1\leq\mu<2\alpha-2$;
\item
$\frac{5}{4}<\alpha\leq\frac{4}{3}$,~~~~ $\frac{\alpha}{5\alpha-5}<\beta<1$,~~~~ $\alpha-1\leq\mu\leq\frac{\alpha}{\beta}-(3\alpha-3)$;
\item
$\frac{4}{3}<\alpha\leq2$, ~~~~$\frac{\alpha}{5\alpha-5}<\beta\leq\frac{\alpha}{4\alpha-4}$,~~~~ $\alpha-1\leq\mu\leq\frac{\alpha}{\beta}-(3\alpha-3)$;
\end{enumerate}
and $p,q,r$ satisfy the following inequality
\begin{enumerate}[(1)]
\item
$\frac{d}{2\alpha-2-\mu}<q<\infty,~~~\frac{d}{\alpha-1}<p<\frac{qd}{d-(\alpha-1-\mu)q}, ~~~q\leq r<\infty$.
\end{enumerate}
\item Assume that $\alpha,\beta,\mu$ satisfy either one of the following conditions
\begin{enumerate}[(i)]
\item
$\frac{5}{4}<\alpha\leq\frac{4}{3}$,~~~~ $\frac{\alpha}{5\alpha-5}<\beta<1$, ~~~$\frac{\alpha}{\beta}-(3\alpha-3)<\mu<\frac{\alpha}{2\beta}-\frac{\alpha-1}{2}$;
\item
$\frac{4}{3}<\alpha\leq2$, ~~~~$\frac{\alpha}{5\alpha-5}<\beta\leq\frac{\alpha}{4\alpha-4}$, ~~~$\frac{\alpha}{\beta}-(3\alpha-3)<\mu<\frac{\alpha}{2\beta}-\frac{\alpha-1}{2}$;
\item
$\frac{4}{3}<\alpha\leq\frac{3}{2}$,~~~~ $\frac{\alpha}{4\alpha-4}<\beta<1$, ~~~$\alpha-1\leq\mu<\frac{\alpha}{2\beta}-\frac{\alpha-1}{2}$;
\item
$\frac{3}{2}<\alpha\leq2$, ~~~~$\frac{\alpha}{4\alpha-4}<\beta\leq\frac{\alpha}{3\alpha-3}$, ~~~$\alpha-1\leq\mu<\frac{\alpha}{2\beta}-\frac{\alpha-1}{2}$;
\end{enumerate}
and $p,q,r$ satisfy either one of the following inequalities
\begin{enumerate}[(1)]
\item
$\frac{d}{2\alpha-2-\mu}<q\leq\frac{d\beta}{(3\alpha-3+\mu)\beta-\alpha}$,~~$\frac{d}{\alpha-1}<p<\frac{qd}{d-(\alpha-1-\mu)q}$,~~ $q\leq r<\infty$;
\item$\frac{d\beta}{(3\alpha-3+\mu)\beta-\alpha}<q<\frac{2d\beta}{(3\alpha-3+\mu)\beta-\alpha}$,~~$\frac{d}{\alpha-1}<p<\frac{qd}{d-(\alpha-1-\mu)q}$,~~
$q\leq r<\frac{d\beta q}{(3\alpha-3+\mu)\beta q-\alpha q-d\beta}$.
\end{enumerate}
\item Assume that $\alpha,\beta,\mu$ satisfy either one of the following conditions
\begin{enumerate}[(i)]
\item
$\frac{4}{3}<\alpha\leq\frac{3}{2}$,~~~~ $\frac{\alpha}{4\alpha-4}<\beta<1$, ~~~$\frac{\alpha}{\beta}-(3\alpha-3)<\mu\leq\frac{\alpha}{2\beta}-(\alpha-1)$;
\item
$\frac{3}{2}<\alpha\leq2$, ~~~~$\frac{\alpha}{4\alpha-4}<\beta\leq\frac{\alpha}{3\alpha-3}$, ~~~$\frac{\alpha}{\beta}-(3\alpha-3)<\mu\leq\frac{\alpha}{2\beta}-(\alpha-1)$;
\item
$\frac{3}{2}<\alpha\leq2$, ~~~~$\frac{\alpha}{3\alpha-3}<\beta<1$, ~~~$0<\mu\leq\frac{\alpha}{2\beta}-(\alpha-1);$
\end{enumerate}
and $p,q,r$ satisfy either one of the following inequalities
\begin{enumerate}[(1)]
\item$\frac{d}{2\alpha-2-\mu}<q\leq\frac{d}{\alpha-1}$,~~~$\frac{d}{\alpha-1}<p<\frac{qd}{d-(\alpha-1-\mu)q}$,~~$\frac{d}{\alpha-1}< r<\frac{qd}{d-(\alpha-1)q}$;
\item$\frac{d}{\alpha-1}<q\leq\frac{d}{\alpha-1-\mu}$,~~~$\frac{d}{\alpha-1}<p<\frac{qd}{d-(\alpha-1-\mu)q}$, ~~~$q\leq r<\infty$;
\item$\frac{d}{\alpha-1-\mu}< q\leq\frac{d\beta}{(3\alpha-3+\mu)\beta-\alpha}$,~~$\frac{d}{\alpha-1}<p<\infty$,~~$q\leq r<\infty$;
\item$\frac{d\beta}{(3\alpha-3+\mu)\beta-\alpha}<q<\frac{2d\beta}{(3\alpha-3+\mu)\beta-\alpha}$,$\frac{d}{\alpha-1}<p<\frac{d\beta q}{(3\alpha-3+\mu)\beta q-\alpha q-d\beta}$, $q\leq r<\frac{d\beta q}{(3\alpha-3+\mu)\beta q-\alpha q-d\beta}$.
\end{enumerate}
\item Assume that $\alpha,\beta,\mu$ satisfy either one of the following conditions
\begin{enumerate}[(i)]
\item
$\frac{4}{3}<\alpha\leq\frac{3}{2}$,~~~~ $\frac{\alpha}{4\alpha-4}<\beta<1$,~~~~ $\frac{\alpha}{2\beta}-(\alpha-1)<\mu<\alpha-1$;
\item
$\frac{3}{2}<\alpha\leq2$, ~~~~$\frac{\alpha}{4\alpha-4}<\beta\leq\frac{\alpha}{3\alpha-3}$,~~~~ $\frac{\alpha}{2\beta}-(\alpha-1)<\mu<\alpha-1$;
\item
$\frac{3}{2}<\alpha\leq2$, ~~~~$\frac{\alpha}{3\alpha-3}<\beta<1$,~~~~ $\frac{\alpha}{2\beta}-(\alpha-1)<\mu\leq\frac{\alpha}{\beta}-(2\alpha-2)$;
\end{enumerate}
and $p,q,r$ satisfy either one of the following inequalities
\begin{enumerate}[(1)]
\item$\frac{d}{2\alpha-2-\mu}<q\leq\frac{d}{\alpha-1}$,~~~$\frac{d}{\alpha-1}<p<\frac{qd}{d-(\alpha-1-\mu)q}$,~~$\frac{d}{\alpha-1}< r<\frac{qd}{d-(\alpha-1)q}$;
\item$\frac{d}{\alpha-1}< q\leq\frac{d\beta}{(3\alpha-3+\mu)\beta-\alpha}$,~~$\frac{d}{\alpha-1}<p<\frac{qd}{d-(\alpha-1-\mu)q}$,~~$q\leq r<\infty$;
\item$\frac{d\beta}{(3\alpha-3+\mu)\beta-\alpha}<q\leq\frac{2d\beta}{(4\alpha-4)\beta-\alpha}$,~~$\frac{d}{\alpha-1}<p<\frac{qd}{d-(\alpha-1-\mu)q}$,~~$q\leq r<\frac{d\beta q}{(3\alpha-3+\mu)\beta q-\alpha q-d\beta}$;
\item$\frac{2d\beta}{(4\alpha-4)\beta-\alpha}\leq q<\frac{2d\beta}{(3\alpha-3+\mu)\beta-\alpha}$,~~$\frac{d}{\alpha-1}<p<\frac{d\beta q}{(3\alpha-3+\mu)\beta q-\alpha q-d\beta}$,~~$q\leq r<\frac{d\beta q}{(3\alpha-3+\mu)\beta q-\alpha q-d\beta}$.
\end{enumerate}

\item Assume that $\alpha,\beta,\mu$ satisfy either one of the following conditions
\begin{enumerate}[(i)]
\item
$\frac{5}{4}<\alpha\leq\frac{3}{2}$,~~~~$\frac{\alpha}{5\alpha-5}<\beta<1$,~~~~$\frac{\alpha}{2\beta}-\frac{\alpha-1}{2}\leq\mu<\frac{\alpha}{3\beta}+\frac{\alpha-1}{3}$;
\item
$\frac{3}{2}<\alpha\leq2$, ~~~~$\frac{\alpha}{5\alpha-5}<\beta\leq\frac{\alpha}{3\alpha-3}$,~~~~$\frac{\alpha}{2\beta}-\frac{\alpha-1}{2}\leq\mu<\frac{\alpha}{3\beta}+\frac{\alpha-1}{3}$;
\item
$\frac{3}{2}<\alpha\leq2$, ~~~~$\frac{\alpha}{3\alpha-3}<\beta<1$,~~~~$\alpha-1\leq\mu<\frac{\alpha}{3\beta}+\frac{\alpha-1}{3}$;
\end{enumerate}
and $p,q,r$ satisfy the following inequality
\begin{enumerate}[(1)]
\item
$\frac{d}{2\alpha-2-\mu}<q<\frac{2d\beta}{(3\alpha-3+\mu)\beta-\alpha}$,~~$\frac{d}{\alpha-1}<p<\frac{qd}{d-(\alpha-1-\mu)q}$,~~ $q\leq r<\frac{d\beta q}{(3\alpha-3+\mu)\beta q-\alpha q-d\beta}$.
\end{enumerate}
\item Assume that $\alpha,\beta,\mu$ satisfy the following condition
\begin{enumerate}[(i)]
\item
$\frac{3}{2}<\alpha\leq2$, ~~~~$\frac{\alpha}{3\alpha-3}<\beta<1$,~~~$\frac{\alpha}{\beta}-(2\alpha-2)<\mu<\frac{\alpha}{2\beta}-\frac{\alpha-1}{2}$;
\end{enumerate}
and $p,q,r$ satisfy either one of the following inequalities
\begin{enumerate}[(1)]
\item
$\frac{d}{2\alpha-2-\mu}<q\leq\frac{2d\beta}{(4\alpha-4+\mu)\beta-\alpha}$,~~$\frac{d}{\alpha-1}<p<\frac{qd}{d-(\alpha-1-\mu)q}$,~~ $\frac{d}{\alpha-1}< r<\frac{qd}{d-(\alpha-1)q}$;
\item $\frac{2d\beta}{(4\alpha-4+\mu)\beta-\alpha}< q\leq\frac{d}{\alpha-1}$,~~$\frac{d}{\alpha-1}<p<\frac{qd}{d-(\alpha-1-\mu)q}$,~~$\frac{d}{\alpha-1}< r<\frac{d\beta q}{(3\alpha-3+\mu)\beta q-\alpha q-d\beta}$;
\item $\frac{d}{\alpha-1}<q\leq\frac{2d\beta}{(4\alpha-4)\beta-\alpha}$,~~$\frac{d}{\alpha-1}<p<\frac{qd}{d-(\alpha-1-\mu)q}$,~~$q\leq r<\frac{d\beta q}{(3\alpha-3+\mu)\beta q-\alpha q-d\beta}$;
\item $\frac{2d\beta}{(4\alpha-4)\beta-\alpha}<q< \frac{2d\beta}{(3\alpha-3+\mu)\beta-\alpha}$,~~$\frac{d}{\alpha-1}<p<\frac{d\beta q}{(3\alpha-3+\mu)\beta q-\alpha q-d\beta}$,~~$q\leq r<\frac{d\beta q}{(3\alpha-3+\mu)\beta q-\alpha q-d\beta}$.
\end{enumerate}
\item Assume that $\alpha,\beta,\mu$ satisfy the following condition
\begin{enumerate}[(i)]
\item
$\frac{3}{2}<\alpha\leq2$, ~~~~$\frac{\alpha}{3\alpha-3}<\beta<1$,~~~$\frac{\alpha}{2\beta}-\frac{\alpha-1}{2}\leq\mu<\frac{\alpha}{3\beta}$;
\end{enumerate}
and $p,q,r$ satisfy either one of the following inequalities
\begin{enumerate}[(1)]
\item
$\frac{d}{2\alpha-2-\mu}<q\leq\frac{2d\beta}{(4\alpha-4+\mu)\beta-\alpha}$,~~$\frac{d}{\alpha-1}<p<\frac{qd}{d-(\alpha-1-\mu)q}$,~~ $\frac{d}{\alpha-1}< r<\frac{qd}{d-(\alpha-1)q}$;
\item $\frac{2d\beta}{(4\alpha-4+\mu)\beta-\alpha}<q\leq\frac{d}{\alpha-1}$,~~$\frac{d}{\alpha-1}<p<\frac{qd}{d-(\alpha-1-\mu)q}$,~~$\frac{d}{\alpha-1}<r<\frac{d\beta q}{(3\alpha-3+\mu)\beta q-\alpha q-d\beta}$.
\end{enumerate}
\item Assume that $\alpha,\beta,\mu$ satisfy the following condition
\begin{enumerate}[(i)]
\item
$\frac{3}{2}<\alpha\leq2$, ~~~~$\frac{\alpha}{3\alpha-3}<\beta<1$,~~~$\frac{\alpha}{3\beta}\leq\mu<\alpha-1$;
\end{enumerate}
and $p,q,r$ satisfy the following inequality
\begin{enumerate}[(1)]
\item
$\frac{d}{2\alpha-2-\mu}<q\leq\frac{d}{\alpha-1}$,~~$\frac{d}{\alpha-1}<p<\frac{qd}{d-(\alpha-1-\mu)q}$,~~ $\frac{d}{\alpha-1}<r<\frac{d\beta q}{(3\alpha-3+\mu)\beta q-\alpha q-d\beta}$.
\end{enumerate}
\item Assume that $\alpha,\beta,\mu$ satisfy the following condition
\begin{enumerate}[(i)]
\item
$\frac{3}{2}<\alpha\leq2$, ~~~~$\frac{\alpha}{3\alpha-3}<\beta<1$,~~~$\frac{\alpha}{2\beta}-\frac{\alpha-1}{2}\leq\mu<\alpha-1$;
\end{enumerate}
and $p,q,r$ satisfy either one of the following inequalities
\begin{enumerate}[(1)]
\item $\frac{d}{\alpha-1}<q\leq\frac{2d\beta}{(4\alpha-4)\beta-\alpha}$,~~$\frac{d}{\alpha-1}<p<\frac{qd}{d-(\alpha-1-\mu)q}$,~~$q\leq r<\frac{d\beta q}{(3\alpha-3+\mu)\beta q-\alpha q-d\beta}$;
\item $\frac{2d\beta}{(4\alpha-4)\beta-\alpha}< q<\frac{2d\beta}{(3\alpha-3+\mu)\beta-\alpha}$,~~$\frac{d}{\alpha-1}<p<\frac{d\beta q}{(3\alpha-3+\mu)\beta q-\alpha q-d\beta}$,~~$q\leq r<\frac{d\beta q}{(3\alpha-3+\mu)\beta q-\alpha q-d\beta}$.
\end{enumerate}
\end{enumerate}
\end{assum}

The main results of this paper are presented as follows.
\begin{thm}\label{thm:localmildexis}
Assume that $d\geq 2$, $0<\beta<1$, $1<\alpha\leq 2$, $\gamma\geq 0$, the gravitational potential $\phi$ satisfies $\nabla \phi\in L^d(\mathbb{R}^d)$ and the exponents $p,q$ and $r$ satisfy either one of the following conditions
\begin{enumerate}[(1)]
\item
$\frac{d}{2\alpha-2}<q\leq\frac{d}{\alpha-1}, ~~~~\frac{d}{\alpha-1}<p<\frac{qd}{d-(\alpha-1)q},~~~~\frac{d}{\alpha-1}<r<\frac{qd}{d-(\alpha-1)q}$;
\item
$\frac{d}{\alpha-1}< q<\infty,~~~~\frac{d}{\alpha-1}<p<\infty,~~~~q\leq r<\infty$.
\end{enumerate}
Then for any $n_0\in L^{q}(\mathbb{R}^d), \nabla v_0\in L^{r}(\mathbb{R}^d)$ and $u_0\in L^{p}(\mathbb{R}^d)$, there exists $T>0$ such that system \eqref{chemo-fluid-eq} admits a unique mild solution $(n,v,u)$ with
\[
n\in C((0,T];L^q(\mathbb{R}^d)),~\nabla v\in C((0,T];L^r(\mathbb{R}^d)),~
u\in C((0,T];L^p(\mathbb{R}^d)).
\]
Furthermore, define the largest time of existence
\[
\begin{split}
T_{max}=\sup\{T>0: &~\eqref{chemo-fluid-eq} \text{ has a unique solution } (n,v,u) \mbox{ with } n\in C((0,T];L^q(\mathbb{R}^d)),\\
&~\nabla v\in C((0,T];L^r(\mathbb{R}^d)),~
u\in C((0,T];L^p(\mathbb{R}^d))\}.
\end{split}
\]
If $T_{max}<\infty$, then we have
\[
\limsup_{t\to T_{max}^-}\|n(\cdot, t)\|_{q}=\infty,~~~\limsup_{t\to T_{max}^-}\|\nabla v(\cdot, t)\|_{r}=\infty,~~~\limsup_{t\to T_{max}^-}\|u(\cdot, t)\|_{p}=\infty.
\]
\end{thm}

\begin{thm}\label{thm:mildexis}
Assume that $d\geq 2$, $\gamma\geq 0$, $\alpha,\beta$ and the exponents $p,q,r$ satisfy Assumption \ref{assum1},
the gravitational potential $\phi$ satisfies $\nabla \phi\in L^d(\mathbb{R}^d)$, and there is a constant $M>0$ such that for all $n_0\in L^{\frac{d}{2\alpha-2}}(\mathbb{R}^d), \nabla v_0\in L^{\frac{d}{\alpha-1}}(\mathbb{R}^d)$ and $u_0\in L^{\frac{d}{\alpha-1}}(\mathbb{R}^d)$ with
\begin{equation}\label{initialdata}
\|n_0\|_{\frac{d}{2\alpha-2}}+\|\nabla v_0\|_{\frac{d}{\alpha-1}}+\|u_0\|_{\frac{d}{\alpha-1}}+\|\nabla\phi\|_{d}<M,
\end{equation}
then there exists a unique mild solution $(n,v,u)$ of \eqref{chemo-fluid-eq} with the property that
\begin{gather*}
\begin{split}
&t^{\frac{d\beta}{\alpha}(\frac{2\alpha-2}{d}-\frac{1}{q})}n\in C((0,\infty);L^q(\mathbb{R}^d)),\\
&t^{\frac{d\beta}{\alpha}(\frac{\alpha-1}{d}-\frac{1}{r})}\nabla v\in C((0,\infty);L^r(\mathbb{R}^d)),\\
&t^{\frac{d\beta}{\alpha}(\frac{\alpha-1}{d}-\frac{1}{p})}u\in C((0,\infty);L^p(\mathbb{R}^d)).
\end{split}
\end{gather*}
\end{thm}

\begin{thm} \label{thm:localhighregu}
Assume that $d\geq 2$, $0<\beta<1$, $1<\alpha\leq 2$, $\gamma\geq 0$, the gravitational potential $\phi$ satisfies $\nabla \phi\in L^d(\mathbb{R}^d)$ and the exponents $p,q,r$ and $\mu$ satisfy either one of the following conditions
\begin{enumerate}[(1)]
\item $0<\mu<\alpha-1,~\frac{d}{2\alpha-2-\mu}<q\leq\frac{d}{\alpha-1},~\frac{d}{\alpha-1}<p<\frac{qd}{d-(\alpha-1-\mu)q},~
    \frac{d}{\alpha-1}<r<\frac{qd}{d-(\alpha-1)q}$;
\item $0<\mu<\alpha-1,~\frac{d}{\alpha-1}< q\leq\frac{d}{\alpha-1-\mu},~\frac{d}{\alpha-1}<p<\frac{qd}{d-(\alpha-1-\mu)q},~q\leq r<\infty$;
\item $0<\mu<\alpha-1,~\frac{d}{\alpha-1-\mu}<q<\infty, ~~\frac{d}{\alpha-1}<p<\infty,~q\leq r<\infty$;
\item $\alpha-1\leq\mu<2\alpha-2,~ \frac{d}{2\alpha-2-\mu}<q<\infty, ~\frac{d}{\alpha-1}<p<\frac{qd}{d-(\alpha-1-\mu)q},~q\leq r<\infty$.
\end{enumerate}
Then for any $n_0\in H^{\mu,q}(\mathbb{R}^d), \nabla v_0\in H^{\mu,r}(\mathbb{R}^d)$ and $u_0\in H^{\mu,p}(\mathbb{R}^d)$, there exists $T>0$ such that \eqref{chemo-fluid-eq} admits a unique mild solution $(n,v,u)$ with
\[
n\in C((0,T];H^{\mu,q}(\mathbb{R}^d)),~\nabla v\in C((0,T];H^{\mu,r}(\mathbb{R}^d)),~
u\in C((0,T];H^{\mu,p}(\mathbb{R}^d)).
\]
Furthermore, define the largest time of existence
\[
\begin{split}
T_{max}=\sup\{T>0: &~\eqref{chemo-fluid-eq} \text{ has a unique solution } (n,v,u) \mbox{ with } n\in C((0,T];H^{\mu,q}(\mathbb{R}^d)),\\
&~\nabla v\in C((0,T];H^{\mu,r}(\mathbb{R}^d)),~
u\in C((0,T];H^{\mu,p}(\mathbb{R}^d))\}.
\end{split}
\]
If $T_{max}<\infty$, then we have
\[
\limsup_{t\to T_{max}^-}\|n(\cdot, t)\|_{H^{\mu,q}}=\infty,~~~\limsup_{t\to T_{max}^-}\|\nabla v(\cdot, t)\|_{H^{\mu,r}}=\infty,~~~\limsup_{t\to T_{max}^-}\|u(\cdot, t)\|_{H^{\mu,p}}=\infty.
\]
\end{thm}

\begin{thm} \label{thm:globalhire}
Assume that $d\geq 2$, $\gamma\geq 0$, $\alpha,\beta,\mu$ and the exponents $p,q,r$ satisfy Assumption \ref{assum2},
the gravitational potential $\phi$ satisfies $\nabla \phi\in L^d(\mathbb{R}^d)$, and there is a constant $\tilde{M}>0$ such that for all $n_0\in L^{\frac{d}{2\alpha-2}}(\mathbb{R}^d), \nabla v_0\in L^{\frac{d}{\alpha-1}}(\mathbb{R}^d)$ and $u_0\in L^{\frac{d}{\alpha-1}}(\mathbb{R}^d)$ with
\begin{equation}\label{hiinitialdata}
\|n_0\|_{\frac{d}{2\alpha-2}}+\|\nabla v_0\|_{\frac{d}{\alpha-1}}+\|u_0\|_{\frac{d}{\alpha-1}}+\|\nabla\phi\|_{d}<\tilde{M},
\end{equation}
then there exists a unique mild solution $(n,v,u)$ of \eqref{chemo-fluid-eq} with the property that
\begin{gather}\label{higlosolu}
\begin{split}
&t^{\frac{d\beta}{\alpha}(\frac{2\alpha-2}{d}-\frac{1}{q})}n\in C((0,\infty);L^q(\mathbb{R}^d)),
~~t^{\frac{d\beta}{\alpha}(\frac{2\alpha-2+\mu}{d}-\frac{1}{q})}n\in C((0,\infty);\dot{H}^{\mu,q}(\mathbb{R}^d)),\\
&t^{\frac{d\beta}{\alpha}(\frac{\alpha-1}{d}-\frac{1}{r})}\nabla v\in C((0,\infty);L^r(\mathbb{R}^d)),
~~t^{\frac{d\beta}{\alpha}(\frac{\alpha-1+\mu}{d}-\frac{1}{r})}\nabla v\in C((0,\infty);\dot{H}^{\mu,r}(\mathbb{R}^d)),\\
&t^{\frac{d\beta}{\alpha}(\frac{\alpha-1}{d}-\frac{1}{p})}u\in C((0,\infty);L^p(\mathbb{R}^d)),
~~t^{\frac{d\beta}{\alpha}(\frac{\alpha-1+\mu}{d}-\frac{1}{p})}u\in C((0,\infty);\dot{H}^{\mu,p}(\mathbb{R}^d)).
\end{split}
\end{gather}
\end{thm}

\begin{thm}\label{thm:localL1mass}
Assume that $d\geq \max\{2,3\alpha-3\}$ and the other conditions in Theorem \ref{thm:localmildexis} hold. Suppose $n_0\in L^1(\mathbb{R}^d)\cap L^{q}(\mathbb{R}^d),~v_0\in L^{1}(\mathbb{R}^d),~\nabla v_0\in L^{r}(\mathbb{R}^d)$ and $u_0\in L^{p}(\mathbb{R}^d)$, then the solution $(n,v,u)$ of \eqref{chemo-fluid-eq} given by Theorem \ref{thm:localmildexis} fulfills
\[
n\in C((0,T];L^1(\mathbb{R}^d)),~v\in C((0,T];L^1(\mathbb{R}^d)).
\]
In addition, the following mass conservation is held to be true
\[
\int_{\mathbb{R}^d}n(x,t)\,dx=\int_{\mathbb{R}^d}n_0(x)\,dx,
\]
\[
\int_{\mathbb{R}^d}v(x,t)\,dx=
\left\{
  \begin{aligned}
  &\int_{\mathbb{R}^d}v_0(x)\,dx+\frac{t^\beta}{\beta\Gamma(\beta)}\int_{\mathbb{R}^d}n_0(x)\,dx, &&\gamma=0,\\
  &E_\beta(-\gamma t^\beta)\int_{\mathbb{R}^d}v_0(x)\,dx+\frac{1-E_\beta(-\gamma t^\beta)}{\gamma\Gamma(\beta)}\int_{\mathbb{R}^d}n_0(x)\,dx, &&\gamma>0.
  \end{aligned}
\right.
\]
\end{thm}

\begin{thm}\label{self-similarsolu}
Let $d\geq 3$ and $\gamma=0$. Assume that $( n_0,v_0,u_0)$ and gravitational potential $\phi$ are as in Theorem \ref{thm:mildexis}. Suppose that $[ n_0,v_0,u_0,\phi]$ are homogeneous functions with degree $-(2\alpha-2),0,-(\alpha-1),0$, respectively, i.e. for any $x\in \mathbb{R}^d, \lambda>0$,
\begin{gather}\label{inida}
\lambda^{2\alpha-2}n_0(\lambda x)=n_0(x),~v_0(\lambda x)=v_0(x),~\lambda^{\alpha-1}u_0(\lambda x)=u_0(x),~\phi(\lambda x)=\phi (x).
\end{gather}
If $[ n_0,v_0,u_0,\phi]$ satisfies \eqref{initialdata}, then the solution $[n,v,u]$ given in Theorem \ref{thm:mildexis} is a self-similar one, that is, for all $x\in \mathbb{R}^d$, $t>0$ and $\lambda>0$, the following holds,
\begin{gather}\label{selfsolu}
\lambda^{2\alpha-2}n(\lambda x,\lambda^{\frac{\alpha}{\beta}}t)=n(x,t),~v(\lambda x,\lambda^{\frac{\alpha}{\beta}}t)=v(x,t),~\lambda^{\alpha-1}u(\lambda x,\lambda^{\frac{\alpha}{\beta}}t)=u(x,t).
\end{gather}
\end{thm}

\section{Notations and Preliminaries}\label{sec2}
The purpose of this section is to state some notions and functional spaces, to introduce briefly the definition of Caputo derivative (provided in \cite{LJLiu1,LJLiu2}), then to define the mild solution to system \eqref{chemo-fluid-eq} and to recall some results which will be used later.

For $1<p<\infty$, $\|u\|_p$ denotes the $L^p$-norm of the function $u$ in $L^p(\mathbb{R}^d)$ space. The space $C((0,\infty);X)$ stands for the set of continuous functions on $(0,\infty)$ with values in Banach space $X$.
For $\mu\in \mathbb{R}, $ $1<p<\infty$, the fractional nonhomogeneous Sobolev spaces and fractional homogeneous Sobolev spaces are defined as those introduce in \cite{GuoBL}:
\[
\begin{split}
H^{\mu,p}(\mathbb{R}^n)&=\{u\in L^p(\mathbb{R}^n): \mathcal{F}^{-1}((1+|\xi|^2)^{\frac{\mu}{2}}\hat{u}) \in L^p(\mathbb{R}^n)\},\\
\dot{H}^{\mu,p}(\mathbb{R}^n)&=\{u\in L^p(\mathbb{R}^n): \mathcal{F}^{-1}(|\xi|^{\mu}\hat{u}) \in L^p(\mathbb{R}^n)\},\\
\end{split}
\]
where $\mathcal{F}$ is the Fourier transform and the corresponding norms are given by
\[
\begin{split}
\|u\|_{H^{\mu,p}}:&=\|\mathcal{F}^{-1}((1+|\xi|^2)^{\frac{\mu}{2}}\hat{u})\|_p=\|(I-\Delta)^{\frac{\mu}{2}}u\|_p,\\
\|u\|_{\dot{H}^{\mu,p}}:&=\|\mathcal{F}^{-1}(|\xi|^{\mu}\hat{u})\|_p=\|(-\Delta)^{\frac{\mu}{2}}u\|_p.
\end{split}
\]
\begin{lem}\label{equiregu}\cite{GuoBL}
Let $\mu\geq0$, $1< p< \infty$. Then $f\in H^{\mu,p}(\mathbb{R}^d)$  if and only if $f\in L^p(\mathbb{R}^d)$ and $(-\Delta)^{\frac{\mu}{2}}f\in L^p(\mathbb{R}^d)$, that is the norms $\|f\|_{H^{\mu,p}}$ and $\|f\|_p+\|f\|_{\dot{H}^{\mu,p}}$ are equivalent.
\end{lem}

In order to introduce the generalized definition of Caputo derivative, let us first recall the following definition of limit.
\begin{defn}\label{def:limit}\cite{LJLiu1}
Let $B$ be a Banach space. For a function $u\in L^{1}_{loc}((0,T);B)$, if there exists $u_0\in B$ such that
\begin{gather}
\lim_{t\to 0^+}\frac{1}{t}\int_0^t\|u(s)-u_0\|_Bds=0,
\end{gather}
we call $u_0$ the right limit of $u$ at $t=0$, denoted by $u(0+)=u_0$. Similarly, we define $u(T-)$, the left limit of $u$ at $t=T$, to be the constant $u_T\in B$ such that
\begin{gather}
\lim_{t\to T^-}\frac{1}{T-t}\int_{t}^T\|u(s)-u_T\|_B\,ds=0.
\end{gather}
\end{defn}
For $\beta>-1$, as discussed in \cite{LJLiu1}, we use the following distributions $\{g_\beta\}$ as the convolution kernels
\begin{gather}\label{kernelg}
g_{\beta}(t) :=\displaystyle
\begin{cases}
\frac{\theta(t)}{\Gamma(\beta)}t^{\beta-1},& \beta>0,\\
\delta(t), &\beta=0,\\
\frac{1}{\Gamma(1+\beta)}D\left(\theta(t)t^{\beta}\right), & \beta\in (-1, 0),
\end{cases}
\end{gather}
here $\theta(t)$ is the standard Heaviside step function and $D$ represents the distributional derivative. $g_{\beta}$ can also be defined for $\beta\le -1$ (see \cite{LJLiu1}) 
and consequently
\begin{gather}\label{eq:group}
g_{\beta_1}*g_{\beta_2}=g_{\beta_1+\beta_2},~~\forall \beta_1,\beta_2\in\mathbb{R}.
\end{gather}
Correspondingly, we can introduce the following time-reflected group
\begin{gather*}
\tilde{\mathscr{C}} :=\{\tilde{g}_{\alpha}: \tilde{g}_{\alpha}(t)=g_{\alpha}(-t), \alpha\in \mathbb{R}\}.
\end{gather*}
Clearly, supp $\tilde{g}\subset(-\infty, 0]$ and for $\gamma\in (0, 1)$, the following equality is held to be true
\begin{gather}
\tilde{g}_{-\gamma}(t)=-\frac{1}{\Gamma(1-\gamma)}D(\theta(-t)(-t)^{-\gamma})=-D\tilde{g}_{1-\gamma}(t).
\end{gather}

\begin{defn}\label{def:rightcaputo}\cite{LJLiu1}
Let $0<\beta<1$. Consider $u\in L^1_{loc} ([0,T);\mathbb{R})$ such that $u$ has a right limit $u(0+)$ at $t=0$ in the sense of Definition \ref{def:limit}. The $\beta$-th order Caputo derivative of $u$, a distribution in $\mathscr{D}'(\mathbb{R})$ with support in $[0,T)$, is defined by
\[
D^\beta_c u:=J_{-\beta}u-u_0g_{1-\beta}=g_{-\beta}\ast(\theta(t)(u-u_0)),
\]
where $J_\alpha$ denotes the fractional integral operator
\begin{equation}\label{frac-inte-opera}
J_\alpha u(t)=\frac{1}{\Gamma(\alpha)}\int_0^t(t-s)^{\alpha-1}u(s)ds.
\end{equation}
Similarly, the $\beta$-th order right Caputo derivative of $u$ is a distribution in $\mathscr{D}'(\mathbb{R})$ with support in $(-\infty, T]$, given by
\[
\tilde{D}_{c;T}^{\beta}u :=\tilde{g}_{-\beta}*\big(\theta(T-t)(u(t)-u(T-))\big).
\]
\end{defn}

Now we state the definition of Caputo derivatives for mappings into Banach spaces.
\begin{defn}\label{def:weakcap}\cite{LJLiu1}
Let $B$ be a Banach space and $u\in L^{1}_{loc}((0,T);B)$. Let $u_0\in B$.
Define the weak Caputo derivative of $u$ associated with initial data $u_0$ to be $_0^cD_t^{\beta}u\in \mathscr{D}'$ such that for any test function $\varphi \in
C_c^{\infty}((-\infty, T); \mathbb{R}) $,
\begin{gather}
\langle _0^cD_t^{\beta}u, \varphi \rangle :=\int_{-\infty}^T (u-u_0)\theta(t) (\tilde{D}_{c; T}^{\beta}\varphi)\, dt
=\int_0^T (u-u_0)\tilde{D}_{c; T}^{\beta}\varphi \, dt,
\end{gather}
where $\mathscr{D}'=\{v|v:C_c^\infty\big((-\infty,T);\mathbb{R}\big)\rightarrow B$ is a bounded linear operator\}.
We call the weak Caputo derivative $_0^cD_t^{\beta}u$ associated with initial value $u_0$ the Caputo derivative of $u$ if $u(0+)=u_0$ in the sense of Definition \ref{def:limit} under the norm of the underlying Banach space $B$.
\end{defn}

In the sequel, we introduce three special functions, Wright function $W_{\kappa,\lambda}$, Marnardi function $M_\beta$ and Mittag-Leffler function $E_{\beta,\gamma}$, respectively. Let $\kappa>-1$ and $\lambda\in \mathbb{C}$. The Wright function $W_{\kappa,\lambda}$ is defined by the following complex series representation and convergent in the whole complex plane
\[
W_{\kappa,\lambda}(z):=\sum_{j=0}^{\infty}\frac{z^j}{j!\Gamma(\kappa j+\lambda)}.
\]
Mainardi function $M_\beta:\mathbb{C}\rightarrow \mathbb{C} $ is a particular case of the Wright function and is given by
\[
M_\beta(z):=W_{-\beta,1-\beta}(-z), ~~~z\in \mathbb{C}.
\]
For real numbers $\beta$ and $\gamma$, the Mittag-Leffler function $E_{\beta,\gamma}: \mathbb{C}\to \mathbb{C}$ is defined as follows
\begin{equation*}
E_{\beta,\gamma}(z):=\sum_{j=0}^{\infty} \frac{z^j}{\Gamma(j\beta+\gamma)},~~~~E_{\beta}(z):=E_{\beta,1}(z).
\end{equation*}

The following classical result plays an important role to obtain the structural estimates for the Mittag-Leffler operators $\{E_\beta(t^\beta\cdot):t>0\}$ and $\{E_{\beta,\beta}(t^\beta\cdot):t>0 \}$.
\begin{lem}\label{lemMbeta(s)}\cite{de}
For $0<\beta<1$ and $-1<r<\infty$, when we restrict $M_\beta$ to the positive real line, it holds that
\begin{align*}
M_\beta(t)\geq0 ~\mbox{for all } t\geq0, ~~\mbox{and}~~ \int_0^\infty t^rM_\beta(t)dt=\frac{\Gamma(1+r)}{\Gamma(1+\beta r)}.
\end{align*}
\end{lem}

In what follows, we list some results about the fractional heat semigroup $\{e^{-t(-\Delta)^{\frac{\alpha}{2}}}\}_{t>0}$,
which is the convolution operator with the fractional heat kernel $K_t(x)$ and denoted by
\begin{equation}\label{Kalphat}
e^{-t(-\Delta)^{\frac{\alpha}{2}}}f:=K^\alpha(t)f=K_t\ast f,
\end{equation}
where $K_t(x)$ is defined via the Fourier transform
\begin{equation}\label{ktx}
K_t(x)=\mathcal{F}^{-1}(e^{-t|\xi|^\alpha})=t^{-\frac{d}{\alpha}}K(t^{-\frac{1}{\alpha}}x)~~~\mbox{with}~~ K(x)=(2\pi)^{-d}\int_{\mathbb{R}^d}e^{ix\cdot\xi}e^{-|\xi|^\alpha}\,d\xi.
\end{equation}
In \cite{bonforte2016}, since $e^{-t|\xi|^\alpha}$ is a tempered distribution, it holds that $K_t\in C^\infty((0,\infty)\times \mathbb{R}^d)$. In addition, we define $K^\alpha_\gamma(t)=e^{-t\gamma}K^\alpha(t)$ as the damped fractional heat semigroup.

\begin{lem}\label{estiKtx}\cite{MiaoYuanZhang}
For any $x\in \mathbb{R}^d$, $0<t<\infty$, $\alpha>0$, $\mu>0$ and  $1\leq p\leq \infty$, the kernel function $K(x)$ has the following properties:
\begin{enumerate}[(i)]
\item
$|K(x)|\leq C(1+|x|)^{-d-\alpha}$,  $K(x)\in L^p(\mathbb{R}^d)$ and $K_t(x)\in L^p(\mathbb{R}^d)$,
\item
$|(-\Delta)^{\frac{\mu}{2}}K(x)|\leq C(1+|x|)^{-d-\mu}$,
$(-\Delta)^{\frac{\mu}{2}}K(x)\in L^p(\mathbb{R}^d)$ and $(-\Delta)^{\frac{\mu}{2}}K_t(x)\in L^p(\mathbb{R}^d)$,
\item
$|\nabla K(x)|\leq C(1+|x|)^{-d-1}$,
$\nabla K(x)\in L^p(\mathbb{R}^d)$ and $\nabla K_t(x)\in L^p(\mathbb{R}^d)$.
\end{enumerate}
\end{lem}

As showed in \cite{Taylor1}, one has the following connections between Mittag-Leffler operators and the Mainardi function:
\begin{align}\label{MLM}
\begin{split}
E_\beta(-t^\beta(-\Delta)^{\frac{\alpha}{2}})&=\int_0^\infty M_\beta(s)K^\alpha(st^\beta)\,ds,\\
E_{\beta,\beta}(-t^\beta(-\Delta)^{\frac{\alpha}{2}})&=\int_0^\infty \beta s M_\beta(s)K^\alpha(st^\beta)\,ds,\\
E_\beta(t^\beta(-(-\Delta)^{\frac{\alpha}{2}}-\gamma))&=\int_0^\infty M_\beta(s)K_\gamma^\alpha(st^\beta)\,ds,\\
E_{\beta,\beta}(t^\beta(-(-\Delta)^{\frac{\alpha}{2}}-\gamma))&=\int_0^\infty \beta s M_\beta(s)K_\gamma^\alpha(st^\beta)\,ds.
\end{split}
\end{align}

In \cite{Taylor1}, the solution to the following abstract initial value problem
\begin{equation*}
\left\{
  \begin{aligned}
  &^c_0D_t^\beta w=A w+f(t),\\
  &w|_{t=0}=w_0,\\
  \end{aligned}
     \right.
\end{equation*}
is given by
\begin{align*}
w(t)=E_\beta( t^\beta A)w_0+\int_0^t(t-\tau)^{\beta-1}E_{\beta,\beta}( (t-\tau)^\beta A)f(\tau)\,d\tau.
\end{align*}
Specialize $A=-(-\Delta)^{\frac{\alpha}{2}}$ to find that $(n,v,u)$ satisfies the following Duhamel type integral equations:
\begin{equation}\label{mildsolu0}
\left\{
  \begin{aligned}
  &n=E_\beta(-t^\beta (-\Delta)^{\frac{\alpha}{2}})n_0-\int_0^t(t-\tau)^{\beta-1}E_{\beta,\beta}\big(-(t-\tau)^\beta (-\Delta)^{\frac{\alpha}{2}}\big)(u\cdot\nabla n+\nabla\cdot (n\nabla v))\,d\tau,\\
  &v=E_\beta(t^\beta (-(-\Delta)^{\frac{\alpha}{2}}-\gamma)v_0-\int_0^t(t-\tau)^{\beta-1}E_{\beta,\beta}\big((t-\tau)^\beta (-(-\Delta)^{\frac{\alpha}{2}}-\gamma)\big)(u\cdot\nabla v-n)\,d\tau,\\
  &u=E_\beta(-t^\beta (-\Delta)^{\frac{\alpha}{2}})u_0-\int_0^t(t-\tau)^{\beta-1}E_{\beta,\beta}\big(-(t-\tau)^\beta (-\Delta)^{\frac{\alpha}{2}}\big)[P((u\cdot\nabla) u+n\nabla\phi)]\,d\tau,
  \end{aligned}
     \right.
\end{equation}
where $P=\{P_{jk}\}_{j,k=1,2,\cdots,d}$ is the projection operator onto the solenoidal vector fields with the expression
\[
P_{jk}=\delta_{jk}+R_j R_k
\]
with Riesz operator $R_j:=\frac{\partial}{\partial x_j}(-\Delta)^{-\frac{1}{2}} $. It is well-known that for all $1<q<\infty$, there holds  $PL^q(\mathbb{R}^d)=L^q_\sigma(\mathbb{R}^d):=\{g\in L^q(\mathbb{R}^d)|\nabla\cdot g=0\}$.

Now we provide the definition of mild solution as follows.
\begin{defn}\label{def:mild}
Let $X$ be a Banach space over space and time. We call that $(n,v,u)\in X$ is a mild solution to system \eqref{chemo-fluid-eq} if $(n,v,u)$ satisfies the integral equation \eqref{mildsolu0} in $X$.
\end{defn}

The following lemma regarding the estimates of the gradient of $K^\alpha(t)$ plays an indispensable role in the proof process of $L^p-L^q$ estimates to the Mittag-Leffler operators.
\begin{lem}\label{Kesti}
Suppose that $1\leq q\leq p\leq \infty$, $\mu\geq0$ and $f\in L^q({\mathbb{R}^d})$. Then the following estimate holds
\[
\begin{split}
\|(-\Delta)^{\frac{\mu}{2}} K^\alpha(t)f\|_{p}&\leq Ct^{-\frac{1}{\alpha}(\mu+\frac{d}{q}-\frac{d}{p})}\|f\|_{q},\\
\|(-\Delta)^{\frac{\mu}{2}} \nabla K^\alpha(t)f\|_{p}&\leq Ct^{-\frac{1}{\alpha}(\mu+1+\frac{d}{q}-\frac{d}{p})}\|f\|_{q}.
\end{split}
\]
here $C$ is a positive constant and independent of $t$.
\end{lem}

\begin{proof}
Combining \eqref{ktx} and Lemma \ref{estiKtx}, we get
\[
\begin{split}
\|(-\Delta)^{\frac{\mu}{2}}K_t(x)\|_r&=t^{-\frac{d}{\alpha}}\big(\int_{\mathbb{R}^d}|(-\Delta)^{\frac{\mu}{2}}K(t^{-\frac{1}{\alpha}}x)|^r\,dx|\big)^{\frac{1}{r}}\\
&=t^{-\frac{d}{\alpha}-\frac{1}{\mu}+\frac{d}{\alpha r}}\big(\int_{\mathbb{R}^d}|(-\Delta)^{\frac{\mu}{2}}K(\xi)|^r\,dx|\big)^{\frac{1}{r}}\\
&\leq Ct^{-\frac{d}{\alpha}(\frac{\mu}{d}+1-\frac{1}{r})}.
\end{split}
\]
By Young's inequality, it holds that
\[
\begin{split}
\|(-\Delta)^{\frac{\mu}{2}} K^\alpha(t)f\|_{p}&=\|(-\Delta)^{\frac{\mu}{2}}K_t(x)\ast f\|_p\\
&\leq \|(-\Delta)^{\frac{\mu}{2}}K_t(x)\|_r\|f\|_q\\
&\leq Ct^{-\frac{d}{\alpha}(\frac{\mu}{d}+1-\frac{1}{r})}\|f\|_q,
\end{split}
\]
where $1+\frac{1}{p}=\frac{1}{r}+\frac{1}{q}$. And the estimates of $\|(-\Delta)^{\frac{\mu}{2}} \nabla K^\alpha(t)f\|_{p}$ can be obtained similarly.
\end{proof}

The following proposition gives the $L^p-L^q$ estimates of the Mittag-Leffler operators.

\begin{prop}\label{pro:MLopLpqesti}
Let $d\geq2$, $0<\beta<1$, $1<\alpha\leq 2$, $\gamma\geq 0$ and $\mu\geq 0$. Assume that $1\leq q\leq p\leq\infty$ and $f\in L^q(\mathbb{R}^d)$, then there exists $C>0$ such that the following estimates hold
\begin{enumerate}[(i)]
\item If $~0\leq\frac{1}{q}-\frac{1}{p}<\frac{\alpha-\mu}{d}$, we have
\begin{equation}\label{MLLpq1}
\|(-\Delta)^{\frac{\mu}{2}}E_\beta(-t^\beta(-\Delta)^{\frac{\alpha}{2}})f\|_{p}\leq Ct^{-\frac{d\beta}{\alpha}(\frac{1}{q}-\frac{1}{p})-\frac{\mu\beta}{\alpha}}\|f\|_{q},
\end{equation}
\begin{equation}\label{MLLpqgamma1}
\|(-\Delta)^{\frac{\mu}{2}}E_\beta(t^\beta(-(-\Delta)^{\frac{\alpha}{2}}-\gamma))f\|_{p}\leq Ct^{-\frac{d\beta}{\alpha}(\frac{1}{q}-\frac{1}{p})-\frac{\mu\beta}{\alpha}}A_{-\frac{d}{\alpha}(\frac{1}{q}-\frac{1}{p})-\frac{\mu}{\alpha}}(t)\|f\|_{q}.
\end{equation}
Particularly, if $p=q$, the constant $C$ can be chosen to be $1$.
\item If $~0\leq\frac{1}{q}-\frac{1}{p}<\frac{2\alpha-\mu}{d}$, we have
\begin{equation}\label{MLLpq2}
\|(-\Delta)^{\frac{\mu}{2}}E_{\beta,\beta}(-t^\beta(-\Delta)^{\frac{\alpha}{2}})f\|_{p}\leq Ct^{-\frac{d\beta}{\alpha}(\frac{1}{q}-\frac{1}{p})-\frac{\mu\beta}{\alpha}}\|f\|_{q},
\end{equation}
\begin{equation}\label{MLLpqgamma2}
\|(-\Delta)^{\frac{\mu}{2}}E_{\beta,\beta}(t^\beta(-(-\Delta)^{\frac{\alpha}{2}}-\gamma))f\|_{p}\leq Ct^{-\frac{d\beta}{\alpha}(\frac{1}{q}-\frac{1}{p})-\frac{\mu\beta}{\alpha}}A_{1-\frac{d}{\alpha}(\frac{1}{q}-\frac{1}{p})-\frac{\mu}{\alpha}}(t)\|f\|_{q}.
\end{equation}
\item If $~0\leq\frac{1}{q}-\frac{1}{p}<\frac{\alpha-1-\mu}{d}$, we have
\begin{equation}\label{MLLpq3}
\|(-\Delta)^{\frac{\mu}{2}}\nabla E_\beta(-t^\beta(-\Delta)^{\frac{\alpha}{2}})f\|_{p}\leq Ct^{-\frac{d\beta}{\alpha}(\frac{1}{q}-\frac{1}{p})-\frac{\beta}{\alpha}-\frac{\mu\beta}{\alpha}}\|f\|_{q},
\end{equation}
\begin{equation}\label{MLLpqgamma3}
\|(-\Delta)^{\frac{\mu}{2}}\nabla E_\beta(t^\beta(-(-\Delta)^{\frac{\alpha}{2}}-\gamma))f\|_{p}\leq Ct^{-\frac{d\beta}{\alpha}(\frac{1}{q}-\frac{1}{p})-\frac{\beta}{\alpha}-\frac{\mu\beta}{\alpha}}A_{-\frac{d}{\alpha}(\frac{1}{q}-\frac{1}{p})-\frac{1+\mu}{\alpha}}(t)\|f\|_{q}.
\end{equation}
\item If $~0\leq\frac{1}{q}-\frac{1}{p}<\frac{2\alpha-1-\mu}{d}$, we have
\begin{equation}\label{MLLpq4}
\|(-\Delta)^{\frac{\mu}{2}}\nabla E_{\beta,\beta}(-t^\beta(-\Delta)^{\frac{\alpha}{2}})f\|_{p}\leq Ct^{-\frac{d\beta}{\alpha}(\frac{1}{q}-\frac{1}{p})-\frac{\beta}{\alpha}-\frac{\mu\beta}{\alpha}}\|f\|_{q},
\end{equation}
\begin{equation}\label{MLLpqgamma4}
\|(-\Delta)^{\frac{\mu}{2}}\nabla E_{\beta,\beta}(t^\beta(-(-\Delta)^{\frac{\alpha}{2}}-\gamma))f\|_{p}\leq Ct^{-\frac{d\beta}{\alpha}(\frac{1}{q}-\frac{1}{p})-\frac{\beta}{\alpha}-\frac{\mu\beta}{\alpha}}A_{1-\frac{d}{\alpha}(\frac{1}{q}-\frac{1}{p})-\frac{1+\mu}{\alpha}}(t)\|f\|_{q},
\end{equation}
\end{enumerate}
where $A_{\sigma}(z)=\int_0^\infty s^{\sigma}M_\beta (s)e^{-sz^\beta\gamma}\,ds$.
\end{prop}

\begin{rem}
For $-1<\sigma<\infty$, the boundedness of function $A_{\sigma}$ has been pointed out in \cite{MarioB}. In fact, by Lemma \ref{lemMbeta(s)}, one can easily find that $$\sup_{t>0}A_\sigma(t)\leq\frac{\Gamma(1+\sigma)}{\Gamma(1+\beta\sigma)}.$$
\end{rem}

\begin{proof}
By virtue of the first formula in \eqref{MLM} and Lemma \ref{Kesti}, it follows that
\begin{gather*}
\begin{split}
\|(-\Delta)^{\frac{\mu}{2}}E_\beta(-t^\beta(-\Delta)^{\frac{\alpha}{2}})f\|_{p}&\leq\int_0^\infty M_\beta (s)\|(-\Delta)^{\frac{\mu}{2}}K^\alpha(st^\beta)f\|_{p}\,ds\\
&\leq Ct^{-\frac{d\beta}{\alpha}(\frac{1}{q}-\frac{1}{p})-\frac{\mu\beta}{\alpha}}\|f\|_{q}\int_0^\infty M_\beta (s)s^{-\frac{d}{\alpha}(\frac{1}{q}-\frac{1}{p})-\frac{\mu}{\alpha}}\,ds\\
&=C\frac{\Gamma(1-\frac{d}{\alpha}(\frac{1}{q}-\frac{1}{p})-\frac{\mu}{\alpha})}{\Gamma(1-\frac{d\beta}{\alpha}(\frac{1}{q}-\frac{1}{p})-\frac{\mu\beta}{\alpha})}
t^{-\frac{d\beta}{\alpha}(\frac{1}{q}-\frac{1}{p})-\frac{\mu\beta}{\alpha}}\|f\|_{q},
\end{split}
\end{gather*}
and the last equality holds due to Lemma \ref{lemMbeta(s)}. Thus \eqref{MLLpq1} is verified.

We now come to estimate \eqref{MLLpqgamma1}. Utilizing the third formula in \eqref{MLM}, the definition of $K^\alpha_\gamma(t)$ and Lemma \ref{Kesti}, we see that
\begin{gather*}
\begin{split}
\|(-\Delta)^{\frac{\mu}{2}}E_\beta(t^\beta(-(-\Delta)^{\frac{\alpha}{2}})-\gamma)f\|_{p}\leq&\int_0^\infty M_\beta (s)\|(-\Delta)^{\frac{\mu}{2}}K_\gamma^\alpha(st^\beta)f\|_{p}\,ds\\
\leq &Ct^{-\frac{d\beta}{\alpha}(\frac{1}{q}-\frac{1}{p})-\frac{\mu\beta}{\alpha}}\int_0^\infty M_\beta (s)e^{-st^\beta\gamma}s^{-\frac{d}{\alpha}(\frac{1}{q}-\frac{1}{p})-\frac{\mu}{\alpha}}\|f\|_{q}\,ds\\
\leq& Ct^{-\frac{d\beta}{\alpha}(\frac{1}{q}-\frac{1}{p})-\frac{\mu\beta}{\alpha}}A_{-\frac{d}{\alpha}(\frac{1}{q}-\frac{1}{p})-\frac{\mu}{\alpha}}(t)\|f\|_{q}.
\end{split}
\end{gather*}

Employing \eqref{MLM} and Lemma \ref{Kesti}, we obtain
\begin{gather*}
\begin{split}
\|(-\Delta)^{\frac{\mu}{2}}E_{\beta,\beta}(-t^\beta(-\Delta)^{\frac{\alpha}{2}})f\|_{p}&\leq\int_0^\infty \beta s M_\beta (s)\|(-\Delta)^{\frac{\mu}{2}}K^\alpha(st^\beta)f\|_{p}\,ds\\
&\leq Ct^{-\frac{d\beta}{\alpha}(\frac{1}{q}-\frac{1}{p})-\frac{\mu\beta}{\alpha}}\int_0^\infty M_\beta (s)s^{1-\frac{d}{\alpha}(\frac{1}{q}-\frac{1}{p})-\frac{\mu}{\alpha}}\|f\|_{q}\, ds\\
&=C\frac{\Gamma(2-\frac{d}{\alpha}(\frac{1}{q}-\frac{1}{p})-\frac{\mu}{\alpha})}{\Gamma(1+\beta-\frac{d\beta}{\alpha}(\frac{1}{q}-\frac{1}{p})-\frac{\mu\beta}{\alpha})}
t^{-\frac{d\beta}{\alpha}(\frac{1}{q}-\frac{1}{p})-\frac{\mu\beta}{\alpha}}\|f\|_{q},
\end{split}
\end{gather*}
and
\begin{gather*}
\begin{split}
\|(-\Delta)^{\frac{\mu}{2}}\nabla E_{\beta,\beta}(-t^\beta(-\Delta)^{\frac{\alpha}{2}})f\|_{p}\leq&\int_0^\infty \beta s M_\beta (s)\|(-\Delta)^{\frac{\mu}{2}}\nabla K^\alpha(st^\beta)f\|_{p}\,ds\\
\leq& Ct^{-\frac{d\beta}{\alpha}(\frac{1}{q}-\frac{1}{p})-\frac{(1+\mu)\beta}{\alpha}}\int_0^\infty M_\beta (s)s^{1-\frac{1+\mu}{\alpha}-\frac{d}{\alpha}(\frac{1}{q}-\frac{1}{p})}\|f\|_{q}\,ds\\
=&C\frac{\Gamma(2-\frac{1+\mu}{\alpha}-\frac{d}{\alpha}(\frac{1}{q}-\frac{1}{p}))}{\Gamma(1+\beta-\frac{(1+\mu)\beta}{\alpha}-\frac{d\beta}{\alpha}(\frac{1}{q}-\frac{1}{p}))}
t^{-\frac{d\beta}{\alpha}(\frac{1}{q}-\frac{1}{p})-\frac{(1+\mu)\beta}{\alpha}}\|f\|_{q},
\end{split}
\end{gather*}
which yields \eqref{MLLpq2} and \eqref{MLLpq4}. The proofs of the other estimates \eqref{MLLpqgamma2}-\eqref{MLLpqgamma3} and \eqref{MLLpqgamma4} can be verified with the same arguments as the above.
\end{proof}

\begin{rem}
Reference \cite{LLW} listed the $L^p-L^q$ estimates regarding the operators $S_\alpha^\beta(t)$ and $T_\alpha^\beta(t)$, which are defined by
\begin{align}\label{PQ}
\begin{split}
S_{\alpha}^{\beta}(t)f(x):&=E_{\beta}(-t^{\beta}A)f(x)=P(\cdot, t)*f(x),\\
T_{\alpha}^{\beta}(t)f(x):&=t^{\beta-1}E_{\beta,\beta}(-t^{\beta}A)f(x)=Y(\cdot, t)*f(x),
\end{split}
\end{align}
and the $L^p-L^q$ estimates in \cite{LLW} follow from the asymptotic behaviours of $P(x,t)$ and $~Y(x,t)$. As a matter of fact, the $L^p-L^q$ estimates of the operators $S_\alpha^\beta(t)$ and $T_\alpha^\beta(t)$ in \cite{LLW} are equivalent to the estimates listed in Proposition \ref{pro:MLopLpqesti} whenever $\mu=0$.
\end{rem}

Based on the above estimates, we then establish the time continuity of the Mittag-Leffler operators.
\begin{prop}\label{pro:MLopconti}
For $d\geq2$, $0<T\leq \infty$, $0<\beta<1$, $1<\alpha\leq2$, $\gamma\geq 0$, $\mu\geq0$ and $p,q$ satisfy the conditions in Proposition \ref{pro:MLopLpqesti}. If $f\in L^q(\mathbb{R}^d)$, then we have
\begin{equation}\label{conMLLpq1}
t^{\frac{d\beta}{\alpha}(\frac{1}{q}-\frac{1}{p})+\frac{\mu\beta}{\alpha}}(-\Delta)^{\frac{\mu}{2}}E_\beta(-t^\beta(-\Delta)^{\frac{\alpha}{2}})f\in C((0,T);L^p(\mathbb{R}^d)),
\end{equation}
\begin{equation}\label{conMLLpqgamma1}
t^{\frac{d\beta}{\alpha}(\frac{1}{q}-\frac{1}{p})+\frac{\mu\beta}{\alpha}}(-\Delta)^{\frac{\mu}{2}}E_\beta(t^\beta(-(-\Delta)^{\frac{\alpha}{2}}-\gamma)f\in C((0,T);L^p(\mathbb{R}^d)),
\end{equation}
\begin{equation}\label{conMLLpq2}
t^{\frac{d\beta}{\alpha}(\frac{1}{q}-\frac{1}{p})+\frac{(1+\mu)\beta}{\alpha}}(-\Delta)^{\frac{\mu}{2}}\nabla E_\beta(-t^\beta(-\Delta)^{\frac{\alpha}{2}})f\in C((0,T);L^p(\mathbb{R}^d)),
\end{equation}
\begin{equation}\label{conMLLpqgamma2}
t^{\frac{d\beta}{\alpha}(\frac{1}{q}-\frac{1}{p})+\frac{(1+\mu)\beta}{\alpha}}(-\Delta)^{\frac{\mu}{2}}\nabla E_\beta(t^\beta(-(-\Delta)^{\frac{\alpha}{2}}-\gamma)f\in C((0,T);L^p(\mathbb{R}^d)),
\end{equation}
\begin{equation}\label{conMLLpq3}
t^{\frac{d\beta}{\alpha}(\frac{1}{q}-\frac{1}{p})+\frac{\mu\beta}{\alpha}}(-\Delta)^{\frac{\mu}{2}}E_{\beta,\beta}(-t^\beta(-\Delta)^{\frac{\alpha}{2}})f\in C((0,T);L^p(\mathbb{R}^d)),
\end{equation}
\begin{equation}\label{conMLLpqgamma3}
t^{\frac{d\beta}{\alpha}(\frac{1}{q}-\frac{1}{p})+\frac{\mu\beta}{\alpha}}(-\Delta)^{\frac{\mu}{2}}E_{\beta,\beta}(t^\beta(-(-\Delta)^{\frac{\alpha}{2}}-\gamma)f\in C((0,T);L^p(\mathbb{R}^d)),
\end{equation}
\begin{equation}\label{conMLLpq4}
t^{\frac{d\beta}{\alpha}(\frac{1}{q}-\frac{1}{p})+\frac{(1+\mu)\beta}{\alpha}}(-\Delta)^{\frac{\mu}{2}}\nabla E_{\beta,\beta}(-t^\beta(-\Delta)^{\frac{\alpha}{2}})f\in C((0,T);L^p(\mathbb{R}^d)),
\end{equation}
\begin{equation}\label{conMLLpqgamma4}
t^{\frac{d\beta}{\alpha}(\frac{1}{q}-\frac{1}{p})+\frac{(1+\mu)\beta}{\alpha}}(-\Delta)^{\frac{\mu}{2}}\nabla E_{\beta,\beta}(t^\beta(-(-\Delta)^{\frac{\alpha}{2}}-\gamma)f\in C((0,T);L^p(\mathbb{R}^d)).
\end{equation}
\end{prop}

\begin{proof}
Fix $t_0>0$ and consider $t>t_0$, the case $t<t_0$ follows analogously. Using the definition of Mittag-Leffler operators \eqref{MLM}, we first write
\begin{align*}
\begin{split}
&\|t^{\frac{d\beta}{\alpha}(\frac{1}{q}-\frac{1}{p})+\frac{\mu\beta}{\alpha}}(-\Delta)^{\frac{\mu}{2}}E_\beta(-t^\beta(-\Delta)^{\frac{\alpha}{2}})f
-t_0^{\frac{d\beta}{\alpha}(\frac{1}{q}-\frac{1}{p})+\frac{\mu\beta}{\alpha}}(-\Delta)^{\frac{\mu}{2}}E_\beta(-t_0^\beta(-\Delta)^{\frac{\alpha}{2}})f\|_{p}\\
\leq&\int_0^\infty s^{-\frac{d}{\alpha}(\frac{1}{q}-\frac{1}{p})-\frac{\mu}{\alpha}}M_\beta(s)\|(st^\beta)^{\frac{d}{\alpha}(\frac{1}{q}-\frac{1}{p})
+\frac{\mu}{\alpha}}(-\Delta)^{\frac{\mu}{2}}K^\alpha(st^\beta)f
-(st_0^\beta)^{\frac{d}{\alpha}(\frac{1}{q}-\frac{1}{p})+\frac{\mu}{\alpha}}(-\Delta)^{\frac{\mu}{2}}K^\alpha(st_0^\beta)f\|_{p}ds\\
\leq&\int_0^\infty s^{-\frac{d}{\alpha}(\frac{1}{q}-\frac{1}{p})-\frac{\mu}{\alpha}}M_\beta(s)\big((st^\beta)^{\frac{d}{\alpha}(\frac{1}{q}-\frac{1}{p})+\frac{\mu}{\alpha}}
-(st_0^\beta)^{\frac{d}{\alpha}(\frac{1}{q}-\frac{1}{p})+\frac{\mu}{\alpha}}\big)\|(-\Delta)^{\frac{\mu}{2}}K^\alpha(st^\beta)f\|_{p}\,ds\\
&+\int_0^\infty s^{-\frac{d}{\alpha}(\frac{1}{q}-\frac{1}{p})-\frac{\mu}{\alpha}}M_\beta(s)(st_0^\beta)^{\frac{d}{\alpha}(\frac{1}{q}-\frac{1}{p})+\frac{\mu}{\alpha}}\|(-\Delta)^{\frac{\mu}{2}}K^\alpha(st^\beta)f
-(-\Delta)^{\frac{\mu}{2}}K^\alpha(st_0^\beta)f\|_{p}\,ds\\
:=&I_1(t,t_0)+I_2(t,t_0).
\end{split}
\end{align*}
Regarding $I_1(t,t_0)$, in terms of Lemma \ref{lemMbeta(s)} and \eqref{MLLpq1} in Proposition \ref{pro:MLopLpqesti}, one obtains that
\begin{align*}
\begin{split}
I_1(t,t_0)&\leq \int_0^\infty s^{-\frac{d}{\alpha}(\frac{1}{q}-\frac{1}{p})-\frac{\mu}{\alpha}}M_\beta(s)\big(t^{\frac{d\beta}{\alpha}(\frac{1}{q}-\frac{1}{p})+\frac{\mu\beta}{\alpha}}
-t_0^{\frac{d\beta}{\alpha}(\frac{1}{q}-\frac{1}{p})+\frac{\mu\beta}{\alpha}}\big)t^{-\frac{d\beta}{\alpha}(\frac{1}{q}-\frac{1}{p})-\frac{\mu\beta}{\alpha}}\|f\|_{q}\,ds\\
&\leq Ct^{-\frac{d\beta}{\alpha}(\frac{1}{q}-\frac{1}{p})-\frac{\mu\beta}{\alpha}}|t^{\frac{d\beta}{\alpha}(\frac{1}{q}-\frac{1}{p})+\frac{\mu\beta}{\alpha}}
-t_0^{\frac{d\beta}{\alpha}(\frac{1}{q}-\frac{1}{p})+\frac{\mu\beta}{\alpha}}|\|f\|_{q}.
\end{split}
\end{align*}
Therefore, $I_1(t,t_0)$ goes to $0$ as $t\rightarrow t_0$.

As for $I_2(t,t_0)$, since $K_t\in C^\infty((0,\infty)\times \mathbb{R}^d)$, we have
\begin{align*}
I_2(t,t_0)&\leq \int_0^\infty M_\beta(s)t_0^{\frac{d\beta}{\alpha}(\frac{1}{q}-\frac{1}{p})+\frac{\mu\beta}{\alpha}}\|(-\Delta)^{\frac{\mu}{2}}K_{st^\beta}
-(-\Delta)^{\frac{\mu}{2}}K_{st_0^\beta}\|_{\frac{pq}{pq+q-p}}\|f\|_{q}\,ds,
\end{align*}
$I_2(t,t_0)$ also goes to $0$ as $t\rightarrow t_0$. Hence this implies that $$t^{\frac{d\beta}{\alpha}(\frac{1}{q}-\frac{1}{p})+\frac{\mu\beta}{\alpha}}(-\Delta)^{\frac{\mu}{2}}E_\beta(-t^\beta(-\Delta)^{\frac{\alpha}{2}})f\in C((0,T);L^p(\mathbb{R}^d)).$$
The other results can be verified in the same way as what we have done above.
\end{proof}

\begin{lem}\label{highregu}\cite{GuoBL}
Let $\mu>0$, $1< p< \infty$. If $f,g\in \mathcal{S}(\mathbb{R}^d)$ and $1<p_2,p_3<\infty$ satisfies
\[
\frac{1}{p}=\frac{1}{p_1}+\frac{1}{p_2}=\frac{1}{p_3}+\frac{1}{p_4},
\]
then the following estimate holds
\[
\|(-\Delta)^{\frac{\mu}{2}}(fg)\|_{p}\leq C\big(\|f\|_{{p_1}}\|g\|_{\dot{H}^{\mu,p_2}}+\|f\|_{\dot{H}^{\mu,p_3}}\|g\|_{{p_4}}\big),
\]
where $\mathcal{S}(\mathbb{R}^d)$ denotes the Schwartz space.
\end{lem}

\section{Existence of mild solutions in \texorpdfstring{$L^p$}{m } spaces}\label{sec3}
In this section, we concentrate on the local existence and global existence of mild solutions to the Cauchy problem of the time-space fractional Keller-Segel-Navier-Stokes system \eqref{chemo-fluid-eq}. To this end, we take advantage of Banach fixed point theorem and Banach implicit function theorem to obtain the local existence and global existence, respectively.

\subsection{Proof of Theorem \ref{thm:localmildexis}}\label{3.1}

\begin{proof}
According to the assumptions in Theorem \ref{thm:localmildexis}, the initial data $n_0\in L^q(\mathbb{R}^d),\nabla v_0\in L^r(\mathbb{R}^d)$ and $u_0\in L^p(\mathbb{R}^d)$, thus there exists a constant $M_0>0$ such that the initial data satisfy
\[
\|n_0\|_q+\|\nabla v_0\|_r+\|u_0\|_q\leq M_0.
\]
For fixed $T>0$, we define the Banach space as follows
\[
X_{T}:=\{(n,v,u):n\in C((0,T];L^q(\mathbb{R}^d)),\nabla v \in C((0,T];L^r(\mathbb{R}^d)),u\in C((0,T];L^p(\mathbb{R}^d))\}
\]
endowed with the norm
\[
\|(n,v,u)\|_{X_T}:=\sup_{0< t\leq T}\|n(t)\|_{q}+\sup_{0< t\leq T}\|\nabla v(t)\|_{r}+\sup_{0< t\leq T}\|u(t)\|_{p},
\]
and the closed subset $S$, which is denoted by
\[
S:=\{(n,v,u)\in X_T:\|(n,v,u)\|_{X_T}\leq 2M_0\}.
\]
Then for $0< t\leq T$, we introduce a mapping $\mathcal{H}=(\mathcal{H}_1,\mathcal{H}_2,\mathcal{H}_3)(n,v,u)$ on $S$ by defining
\begin{equation}\label{mildsoluH}
\left\{
  \begin{aligned}
  \mathcal{H}_1(n,v,u)=&E_\beta(-t^\beta(-\Delta)^{\frac{\alpha}{2}})n_0-\int_0^t(t-\tau)^{\beta-1}E_{\beta,\beta}(-(t-\tau)^\beta(-\Delta)^{\frac{\alpha}{2}})(u\cdot\nabla n)\,d\tau\\
  &-\int_0^t(t-\tau)^{\beta-1} E_{\beta,\beta}(-(t-\tau)^\beta(-\Delta)^{\frac{\alpha}{2}})\nabla\cdot(n\nabla v)\,d\tau,\\
  \mathcal{H}_2(n,v,u)=&E_\beta\big(t^\beta(-(-\Delta)^{\frac{\alpha}{2}}-\gamma)\big)v_0
  +\int_0^t(t-\tau)^{\beta-1} E_{\beta,\beta}\big((t-\tau)^\beta(-(-\Delta)^{\frac{\alpha}{2}}-\gamma)\big)n\,d\tau\\
  &-\int_0^t(t-\tau)^{\beta-1}E_{\beta,\beta}\big((t-\tau)^\beta(-(-\Delta)^{\frac{\alpha}{2}}-\gamma)\big)(u\cdot\nabla v)\,d\tau,\\
  \mathcal{H}_3(n,v,u)=&E_\beta(-t^\beta(-\Delta)^{\frac{\alpha}{2}})u_0-\int_0^t(t-\tau)^{\beta-1}E_{\beta,\beta}(-(t-\tau)^\beta(-\Delta)^{\frac{\alpha}{2}})[P((u\cdot\nabla) u)]\,d\tau\\
  &-\int_0^t(t-\tau)^{\beta-1}E_{\beta,\beta}(-(t-\tau)^\beta(-\Delta)^{\frac{\alpha}{2}})[P(n\nabla\phi)]\,d\tau.\\
  \end{aligned}
     \right.
\end{equation}

Next, we intend to show that $\mathcal{H}=(\mathcal{H}_1,\mathcal{H}_2,\mathcal{H}_3)$ defined in \eqref{mildsoluH} is a mapping from $S$ to itself.
For any fixed $(n,v,u)\in S$, by \eqref{MLLpq4} in Proposition \ref{pro:MLopLpqesti} and H\"{o}lder's inequality, one has
\begin{equation}\label{preun1}
\begin{split}
\|\nabla \cdot E_{\beta,\beta}(-(t-\tau)^\beta(-\Delta)^{\frac{\alpha}{2}})(u n)\|_{q}&\leq C (t-\tau)^{-\frac{d\beta}{\alpha}(\frac{1}{p}+\frac{1}{q}-\frac{1}{q})-\frac{\beta}{\alpha}}\|un\|_{\frac{pq}{p+q}}\\
&\leq C (t-\tau)^{-\frac{d\beta}{\alpha p}-\frac{\beta}{\alpha}}\|u\|_p\|n\|_q,
\end{split}
\end{equation}
and
\begin{equation}\label{prenv1}
\begin{split}
\|\nabla\cdot E_{\beta,\beta}(-(t-\tau)^\beta(-\Delta)^{\frac{\alpha}{2}})(n\nabla v)\|_{q}&\leq  C (t-\tau)^{-\frac{d\beta}{\alpha}(\frac{1}{r}+\frac{1}{q}-\frac{1}{q})-\frac{\beta}{\alpha}}\|n\nabla v\|_{\frac{rq}{r+q}}\\
&\leq C (t-\tau)^{-\frac{d\beta}{\alpha r}-\frac{\beta}{\alpha}}\|\nabla v\|_r\|n\|_q,
\end{split}
\end{equation}
provided $0<\frac{1}{p}<\frac{2\alpha-1}{d}$ and $0<\frac{1}{r}<\frac{2\alpha-1}{d}$.

For any $0< t\leq T$, we deal with the first formula in \eqref{mildsoluH} with the help of  \eqref{MLLpq1} in Proposition \ref{pro:MLopLpqesti}, \eqref{preun1} and \eqref{prenv1} as follows
\begin{equation*}
\begin{split}
\|\mathcal{H}_1(n,v,u)\|_{q}\leq& \|E_\beta(-t^\beta(-\Delta)^{\frac{\alpha}{2}})n_0\|_{q}+\int_0^t(t-\tau)^{\beta-1}\|\nabla\cdot E_{\beta,\beta}(-(t-\tau)^\beta(-\Delta)^{\frac{\alpha}{2}})(u n)\|_{q}\,d\tau\\
&~+\int_0^t(t-\tau)^{\beta-1}\|\nabla\cdot E_{\beta,\beta}(-(t-\tau)^\beta(-\Delta)^{\frac{\alpha}{2}})(n\nabla v)\|_{q}\,d\tau\\
\leq&\|n_0\|_{q}+ C\int_0^t(t-\tau)^{-\frac{d\beta}{\alpha p}-\frac{\beta}{\alpha}+\beta-1}\|u\|_{p}\|n\|_{q}\,d\tau\\
&~+ C\int_0^t(t-\tau)^{-\frac{d\beta}{\alpha r}-\frac{\beta}{\alpha}+\beta-1}\|\nabla v\|_{r}\|n\|_{q}\,d\tau.
\end{split}
\end{equation*}
Since $p>\frac{d}{\alpha-1}$ and $r>\frac{d}{\alpha-1}$ imply that $-\frac{d\beta}{\alpha p}-\frac{\beta}{\alpha}+\beta>0$ and $-\frac{d\beta}{\alpha r}-\frac{\beta}{\alpha}+\beta>0$, respectively, we then have
\begin{equation}\label{localmildH1}
\begin{split}
\|\mathcal{H}_1(n,v,u)\|_{q}\leq&\|n_0\|_{q}+CT^{-\frac{d\beta}{\alpha}(\frac{1}{p}-\frac{\alpha-1}{d})}\|u\|_{C((0,T];L^p(\mathbb{R}^d))}\|n\|_{C((0,T];L^q(\mathbb{R}^d))}\\
&+CT^{-\frac{d\beta}{\alpha}(\frac{1}{r}-\frac{\alpha-1}{d})}\|\nabla v\|_{C((0,T];L^r(\mathbb{R}^d))}\|n\|_{C((0,T];L^q(\mathbb{R}^d))}.
\end{split}
\end{equation}

Now, we will verify the time continuity of $\mathcal{H}_1(n,v,u)$. Choose $0< t<t+\delta\leq T$ for some $\delta>0$, then one has
\begin{equation*}
\begin{split}
&\|\mathcal{H}_1(n,v,u)(t+\delta)-\mathcal{H}_1(n,v,u)(t)\|_{q}\\
\leq&\|E_{\beta}(-(t+\delta)^\beta(-\Delta)^{\frac{\alpha}{2}})n_0-E_{\beta}(-t^\beta(-\Delta)^{\frac{\alpha}{2}})n_0\|_{q}\\
&+ \int_{t}^{t+\delta}(t+\delta-\tau)^{\beta-1}\|\nabla\cdot E_{\beta,\beta}(-(t+\delta-\tau)^\beta(-\Delta)^{\frac{\alpha}{2}})(u n+n\nabla v)\|_{q}\,d\tau\\
&+\int_0^{t}\|\big((t+\delta-\tau)^{\beta-1}\nabla\cdot E_{\beta,\beta}(-(t+\delta-\tau)^\beta(-\Delta)^{\frac{\alpha}{2}})\\
&~~~~~~~~~~~-(t-\tau)^{\beta-1}\nabla\cdot E_{\beta,\beta}(-(t-\tau)^\beta(-\Delta)^{\frac{\alpha}{2}})\big)(u n+n\nabla v)\|_{q}\,d\tau\\
:=&I_1+I_2+I_3.
\end{split}
\end{equation*}
In terms of \eqref{conMLLpq1} in Proposition \ref{pro:MLopconti}, the first term $I_1$ goes to zero as $\delta \to 0$.

For the second term $I_2$, in a similar fashion as the estimate of \eqref{localmildH1}, by Proposition \ref{pro:MLopLpqesti} and H\"{o}lder's inequality, we see that
\[
\begin{split}
I_2\leq&C\int_t^{t+\delta}(t+\delta-\tau)^{-\frac{d\beta}{\alpha p}-\frac{\beta}{\alpha}+\beta-1}\|u\|_{p}\|n\|_{q}\,d\tau+C\int_t^{t+\delta}(t+\delta-\tau)^{-\frac{d\beta}{\alpha r}-\frac{\beta}{\alpha}+\beta-1}\|\nabla v\|_{r}\|n\|_{q}\,d\tau\\
\leq& C\|n\|_{C((0,T];L^q(\mathbb{R}^d))}\big(\delta^{\frac{d\beta}{\alpha}(\frac{\alpha-1}{d}-\frac{1}{p})}\|u\|_{C((0,T];L^p(\mathbb{R}^d))}+\delta^{\frac{d\beta}{\alpha}(\frac{\alpha-1}{d}-\frac{1}{r})}\|\nabla v\|_{C((0,T];L^r(\mathbb{R}^d))}\big),
\end{split}
\]
then it is not difficult to observe that $I_2$ goes to zero as $\delta \to 0$.

Regarding the third term $I_3$, we split it into two parts as follows
\begin{equation*}
\begin{split}
I_3\leq&\int_0^{t}\big((t+\delta-\tau)^{\beta-1}-(t-\tau)^{\beta-1})\|\nabla\cdot E_{\beta,\beta}(-(t+\delta-\tau)^\beta(-\Delta)^{\frac{\alpha}{2}})(u n+n\nabla v)\,d\tau\|_{q}\,d\tau\\
&+\int_0^t(t-\tau)^{\beta-1}\|\big(\nabla\cdot E_{\beta,\beta}(-(t+\delta-\tau)^\beta(-\Delta)^{\frac{\alpha}{2}})-\nabla\cdot E_{\beta,\beta}(-(t+\delta-\tau)^\beta(-\Delta)^{\frac{\alpha}{2}})\big)\\
&~~~~~~~\times(u n+n\nabla v)\|_{q}\,d\tau\\
:=& I_{31}+I_{32},
\end{split}
\end{equation*}
For any $0<\tau<t$, the fact $(t+\delta)^{-\frac{d\beta}{\alpha p}-\frac{\beta}{\alpha}}<(t+\delta-\tau)^{-\frac{d\beta}{\alpha p}-\frac{\beta}{\alpha}}<\delta^{-\frac{d\beta}{\alpha p}-\frac{\beta}{\alpha}}$ yields that
\[
-\int_0^t(t+\delta-\tau)^{-\frac{d\beta}{\alpha p}-\frac{\beta}{\alpha}}(t-\tau)^{\beta-1}\,d\tau<-C(t+\delta)^{-\frac{d\beta}{\alpha p}-\frac{\beta}{\alpha}}t^\beta.
\]
Combining the above inequality with \eqref{preun1} and \eqref{prenv1}, we find that
\begin{equation}\label{loH1conI31}
\begin{split}
I_{31}\leq&C\big((t+\delta)^{\beta-\frac{\beta}{\alpha}-\frac{d\beta}{ \alpha p}}-\delta^{\beta-\frac{\beta}{\alpha}-\frac{d\beta}{ \alpha p}}-
(t+\delta)^{-\frac{d\beta}{\alpha p}-\frac{\beta}{\alpha}}t^\beta\big)\|n\|_{C((0,T];L^q(\mathbb{R}^d))}\|u\|_{C((0,T];L^p(\mathbb{R}^d))}\\
+&C\big((t+\delta)^{\beta-\frac{\beta}{\alpha}-\frac{d\beta}{ \alpha r}}-\delta^{\beta-\frac{\beta}{\alpha}-\frac{d\beta}{ \alpha r}}-
(t+\delta)^{-\frac{d\beta}{\alpha r}-\frac{\beta}{\alpha}}t^\beta\big)\|n\|_{C((0,T];L^q(\mathbb{R}^d))}\|\nabla v\|_{C((0,T];L^r(\mathbb{R}^d))}.
\end{split}
\end{equation}
Therefore $I_{31}$ goes to zero as $\delta \to 0$.

As regards $I_{32}$, utilizing the estimates listed in \eqref{preun1} and \eqref{prenv1}, we derive that
\begin{equation}\label{loH1conI32}
\begin{split}
I_{32}\leq&C\int_0^t(t-\tau)^{\beta-1}\|\int_0^\infty\beta sM_\beta(s)\big(\nabla\cdot K^\alpha(s(t+\delta)^\beta)-\nabla\cdot K^\alpha(st^\beta)\big)(u n+n\nabla v)\,ds\|_{q}\,d\tau\\
\leq&C\int_0^t(t-\tau)^{\beta-1}\int_0^\infty\beta sM_\beta(s)\|\nabla\cdot K_{s(t+\delta)^\beta}-\nabla\cdot K_{st^\beta}\|_{\frac{p}{p-1}}\|u\|_p \|n\|_q\,ds\,d\tau\\
&+C\int_0^t(t-\tau)^{\beta-1}\int_0^\infty\beta sM_\beta(s)\|\nabla\cdot K_{s(t+\delta)^\beta}-\nabla\cdot K_{st^\beta}\|_{\frac{r}{r-1}}\|\nabla v\|_r \|n\|_q\,ds\,d\tau.
\end{split}
\end{equation}
Recall that $K_t\in C^\infty((0,\infty)\times \mathbb{R}^d)$, $I_{32}$ goes to zero as $\delta \to 0$. Thus, \eqref{loH1conI31}-\eqref{loH1conI32} imply that $I_3$ goes to zero as $\delta \to 0$. At this point, this verifies that $\mathcal{H}_1(n,v,u)\in C((0,T],L^q(\mathbb{R}^d))$.

We next estimate $\|\nabla\mathcal{H}_2(n,v,u)\|_{r}$ and $\|\mathcal{H}_3(n,v,u)\|_{p}$. Once again we employ \eqref{mildsoluH}, Proposition \ref{pro:MLopLpqesti}, H\"{o}lder's inequality and the boundedness of the projection operator $P$ to obtain
\begin{equation}\label{localmildH2}
\begin{split}
\|\nabla\mathcal{H}_2(n,v,u)\|_{r}\leq& \|\nabla E_\beta(t^\beta(-(-\Delta)^{\frac{\alpha}{2}})-\gamma)v_0\|_{r}\\
&+\int_0^t(t-\tau)^{\beta-1}\|\nabla E_{\beta,\beta}((t-\tau)^\beta(-(-\Delta)^{\frac{\alpha}{2}})-\gamma)(u\cdot\nabla v)\|_{r}\,d\tau\\
&+\int_0^t(t-\tau)^{\beta-1}\|\nabla E_{\beta,\beta}((t-\tau)^\beta(-(-\Delta)^{\frac{\alpha}{2}})-\gamma)n\|_{r}\,d\tau\\
\leq&\|\nabla v_0\|_{r}+CT^{-\frac{d\beta}{\alpha}(\frac{1}{p}-\frac{\alpha-1}{d})}\|u\|_{C((0,T];L^p(\mathbb{R}^d))}\|\nabla v\|_{C((0,T];L^r(\mathbb{R}^d))}\\
&+CT^{-\frac{d\beta}{\alpha}(\frac{1}{q}-\frac{1}{r})-\frac{\beta}{\alpha}+\beta}\|n\|_{C((0,T];L^q(\mathbb{R}^d))},
\end{split}
\end{equation}
and
\begin{equation}\label{localmildH3}
\begin{split}
\|\mathcal{H}_3(n,v,u)\|_{p}\leq& \|E_\beta(-t^\beta(-\Delta)^{\frac{\alpha}{2}})u_0\|_{p}+\int_0^t(t-\tau)^{\beta-1}\|E_{\beta,\beta}(-(t-\tau)^\beta(-\Delta)^{\frac{\alpha}{2}})[P(n\nabla \phi)]\|_{p}\,d\tau\\
&+\int_0^t(t-\tau)^{\beta-1}\|\nabla\cdot E_{\beta,\beta}(-(t-\tau)^\beta(-\Delta)^{\frac{\alpha}{2}})[P(u\otimes u)]\|_{p}\,d\tau\\
\leq&\|u_0\|_{p}+ C\int_0^t(t-\tau)^{-\frac{d\beta}{\alpha p}-\frac{\beta}{\alpha}+\beta-1}\|u\|^2_{p}\,d\tau\\
&+ C\int_0^t(t-\tau)^{-\frac{d\beta}{\alpha }(\frac{1}{q}-\frac{1}{p})-\frac{\beta}{\alpha}+\beta-1}\|\nabla \phi\|_{d}\|n\|_{q}\,d\tau\\
\leq&\|u_0\|_{p}+CT^{-\frac{d\beta}{\alpha}(\frac{1}{p}-\frac{\alpha-1}{d})}\|u\|^2_{C((0,T];L^p(\mathbb{R}^d))}
+CT^{-\frac{d\beta}{\alpha}(\frac{1}{q}-\frac{1}{p})-\frac{\beta}{\alpha}+\beta}\|n\|_{C((0,T];L^q(\mathbb{R}^d))},
\end{split}
\end{equation}
provided $p>\frac{d}{\alpha-1}$, $0\leq\frac{1}{q}-\frac{1}{r}<\frac{\alpha-1}{d}$ and $0\leq\frac{1}{q}-\frac{1}{p}+\frac{1}{d}<\frac{\alpha}{d}$.

On the basis of \eqref{localmildH2} and \eqref{localmildH3}, the time continuity can be proved with the same argument as the proof of $\mathcal{H}_1(n,v,u)\in C((0,T];L^q(\mathbb{R}^d))$, which yields that $\nabla\mathcal{H}_2(n,v,u)\in C((0,T];L^r(\mathbb{R}^d))$ and $\mathcal{H}_3(n,v,u)\in C((0,T];L^p(\mathbb{R}^d))$.

Combining the inequalities \eqref{localmildH1}, \eqref{localmildH2} and \eqref{localmildH3} leads to
\begin{equation*}
\begin{split}
&\|\mathcal{H}_1(n,v,u)\|_{q}+\|\nabla\mathcal{H}_2(n,v,u)\|_{r}+\|\mathcal{H}_3(n,v,u)\|_{p}\\
\leq& \|n_0\|_{p}+\|\nabla v_0\|_{r}+\|u_0\|_{p}+C_1\big(T^{-\frac{d\beta}{\alpha}(\frac{1}{q}-\frac{1}{r}-\frac{\alpha-1}{d})}
+T^{-\frac{d\beta}{\alpha}(\frac{1}{q}-\frac{1}{p}-\frac{\alpha-1}{d})}\big)\| n\|_{C((0,T];L^q(\mathbb{R}^d))}\\
&+C_2T^{-\frac{d\beta}{\alpha}(\frac{1}{r}-\frac{\alpha-1}{d})}\|\nabla v\|_{C((0,T];L^r(\mathbb{R}^d))}\| n\|_{C((0,T];L^q(\mathbb{R}^d))}\\
&+C_3T^{-\frac{d\beta}{\alpha}(\frac{1}{p}-\frac{\alpha-1}{d})}\big(\|u\|_{C((0,T];L^p(\mathbb{R}^d))}\| n\|_{C((0,T];L^q(\mathbb{R}^d))}+\|u\|_{C((0,T];L^p(\mathbb{R}^d))}\|\nabla v\|_{C((0,T];L^r(\mathbb{R}^d))}\\
&~~~~~~~~~~~~~~~~~~~~~~~+\|u\|^2_{C((0,T];L^p(\mathbb{R}^d))}\big)\\
\leq&M_0+2C_1M_0\big(T^{-\frac{d\beta}{\alpha}(\frac{1}{q}-\frac{1}{r}-\frac{\alpha-1}{d})}+T^{-\frac{d\beta}{\alpha}(\frac{1}{q}-\frac{1}{p}-\frac{\alpha-1}{d})}\big)
+4\tilde{C}_2M_0^2\big(T^{-\frac{d\beta}{\alpha}(\frac{1}{r}-\frac{\alpha-1}{d})}+T^{-\frac{d\beta}{\alpha}(\frac{1}{p}-\frac{\alpha-1}{d})}\big).
\end{split}
\end{equation*}
If we choose
\[
T_0:=\min\{(8C_1)^{\frac{\alpha q r}{\beta((r-q)d-(\alpha-1)qr)}},
(8C_1)^{\frac{\alpha q p}{\beta((p-q)d-(\alpha-1)qp)}},(16\tilde{C}_2M_0)^{\frac{\alpha r}{\beta(d-(\alpha-1)r)}},
(16\tilde{C}_2M_0)^{\frac{\alpha p}{\beta(d-(\alpha-1)p)}}\},
\]
it is easily to check that for $0<t\leq T_0$,
\[
\|\mathcal{H}_1(n,v,u)\|_{q}+\|\nabla\mathcal{H}_2(n,v,u)\|_{r}+\|\mathcal{H}_3(n,v,u)\|_{p}\leq 2M_0.
\]
Thus, we conclude that $\mathcal{H}(n,v,u): S\rightarrow S$.

Let $(n_1,v_1,u_1), (n_2,v_2,u_2)\in X_{T}$. Define
\[
\begin{split}
D_T(n_1-n_2,v_1-v_2,u_1-u_2):=&\sup_{0\leq t\leq T}\|(n_1-n_2)(t)\|_{q}\\
&+\sup_{0< t\leq T}\|\nabla (v_1-v_2)(t)\|_{r}+\sup_{0< t\leq T}\|(u_1-u_2)(t)\|_{p}.
\end{split}
\]
We proceed in the same way as the estimates of \eqref{localmildH1}, \eqref{localmildH2} and \eqref{localmildH3} to obtain
\begin{equation}\label{localH11-2}
\begin{split}
&\|\mathcal{H}_1(n_1,v_1,u_1)-\mathcal{H}_1(n_2,v_2,u_2)\|_{q}\\
\leq& CT^{-\frac{d\beta}{\alpha}(\frac{1}{p}-\frac{\alpha-1}{d})}\|u_1-u_2\|_{C((0,T];L^p(\mathbb{R}^d))}\|n_1\|_{C((0,T];L^q(\mathbb{R}^d))}\\
&+CT^{-\frac{d\beta}{\alpha}(\frac{1}{p}-\frac{\alpha-1}{d})}\|n_1-n_2\|_{C((0,T];L^q(\mathbb{R}^d))}\|u_2\|_{C((0,T];L^p(\mathbb{R}^d))}\\
&+CT^{-\frac{d\beta}{\alpha}(\frac{1}{r}-\frac{\alpha-1}{d})}\|n_1-n_2\|_{C((0,T];L^q(\mathbb{R}^d))}\|\nabla v_1\|_{C((0,T];L^r(\mathbb{R}^d))}\\
&+CT^{-\frac{d\beta}{\alpha}(\frac{1}{r}-\frac{\alpha-1}{d})}\|\nabla(v_1-v_2)\|_{C((0,T];L^r(\mathbb{R}^d))}\|n_2\|_{C((0,T];L^q(\mathbb{R}^d))}\\
\leq&2C_4M_0\big(T^{-\frac{d\beta}{\alpha}(\frac{1}{p}-\frac{\alpha-1}{d})}+T^{-\frac{d\beta}{\alpha}(\frac{1}{r}-\frac{\alpha-1}{d})}\big) D_T(n_1-n_2,v_1-v_2,u_1-u_2),
\end{split}
\end{equation}
\begin{equation}\label{localH21-2}
\begin{split}
&\|\nabla(\mathcal{H}_2(n_1,v_1,u_1)-\mathcal{H}_2(n_2,v_2,u_2))\|_{r}\\
\leq& CT^{-\frac{d\beta}{\alpha}(\frac{1}{p}-\frac{\alpha-1}{d})}\|u_1-u_2\|_{C((0,T];L^p(\mathbb{R}^d))}\|\nabla v_1\|_{C((0,T];L^r(\mathbb{R}^d))}\\
&+CT^{-\frac{d\beta}{\alpha}(\frac{1}{p}-\frac{\alpha-1}{d})}\|\nabla(v_1-v_2)\|_{C((0,T];L^r(\mathbb{R}^d))}\|u_2\|_{C((0,T];L^p(\mathbb{R}^d))}\\
&+CT^{-\frac{d\beta}{\alpha}(\frac{1}{q}-\frac{1}{r}-\frac{\alpha-1}{d})}\|n_1-n_2\|_{C((0,T];L^q(\mathbb{R}^d))}\\
\leq&C_5\big(2M_0T^{-\frac{d\beta}{\alpha}(\frac{1}{p}-\frac{\alpha-1}{d})}+T^{-\frac{d\beta}{\alpha}(\frac{1}{q}-\frac{1}{r}-\frac{\alpha-1}{d})}\big) D_T(n_1-n_2,v_1-v_2,u_1-u_2),
\end{split}
\end{equation}
and
\begin{equation}\label{localH31-2}
\begin{split}
&\|\mathcal{H}_3(n_1,v_1,u_1)-\mathcal{H}_3(n_2,v_2,u_2)\|_{p}\\
&+CT^{-\frac{d\beta}{\alpha}(\frac{1}{p}-\frac{\alpha-1}{d})}\|u_1-u_2\|_{C((0,T];L^r(\mathbb{R}^d))}(\|u_1\|_{C((0,T];L^p(\mathbb{R}^d))}+\|u_2\|_{C((0,T];L^p(\mathbb{R}^d))})\\
&+CT^{-\frac{d\beta}{\alpha}(\frac{1}{q}-\frac{1}{p}-\frac{\alpha-1}{d})}\|n_1-n_2\|_{C((0,T];L^q(\mathbb{R}^d))}\\
\leq&C_6\big(2M_0T^{-\frac{d\beta}{\alpha}(\frac{1}{p}-\frac{\alpha-1}{d})}+T^{-\frac{d\beta}{\alpha}(\frac{1}{q}-\frac{1}{p}-\frac{\alpha-1}{d})}\big) D_T(n_1-n_2,v_1-v_2,u_1-u_2).
\end{split}
\end{equation}
Combining the above estimates \eqref{localH11-2}-\eqref{localH31-2}, it holds that
\[
\begin{split}
&D_T(\mathcal{H}(n_1,v_1,u_1)-\mathcal{H}(n_2,v_2,u_2))\leq \big(\tilde{C}_3M_0T^{-\frac{d\beta}{\alpha}(\frac{1}{p}-\frac{\alpha-1}{d})}+2C_4M_0T^{-\frac{d\beta}{\alpha}(\frac{1}{r}-\frac{\alpha-1}{d})}\\
&~~~~~+C_5T^{-\frac{d\beta}{\alpha}(\frac{1}{q}-\frac{1}{r}-\frac{\alpha-1}{d})}
+C_6T^{-\frac{d\beta}{\alpha}(\frac{1}{q}-\frac{1}{p}-\frac{\alpha-1}{d})}\big) D_T(n_1-n_2,v_1-v_2,u_1-u_2).
\end{split}
\]
Then choose $T_1$ to satisfy
\begin{small}
\[
T_1:=\min\{T_0, \big(\frac{4\tilde{C}_3M_0}{\varrho}\big)^{\frac{\alpha p}{\beta(d-(\alpha-1)q)}},
\big(\frac{8C_4M_0}{\varrho}\big)^{\frac{\alpha r}{\beta(d-(\alpha-1)q)}},\big(\frac{4C_5}{\varrho}\big)^{\frac{\alpha q r}{\beta((r-q)d-(\alpha-1)qr)}},
\big(\frac{4C_6}{\varrho}\big)^{\frac{\alpha p q}{\beta((p-q)d-(\alpha-1)qp)}}\},
\]
\end{small}
we conclude that $\mathcal{H}:S\rightarrow S$ is a strict contraction mapping, that is there exists a constant $0<\varrho<1$ such that
\begin{equation*}
D_T(\mathcal{H}(n_1,v_1,u_1)-\mathcal{H}(n_2,v_2,u_2))\leq \varrho D_T(n_1-n_2,v_1-v_2,u_1-u_2).
\end{equation*}

Therefore, it follows from Banach fixed point theorem that there exists $(n,v,u)\in S$ such that $(n,v,u)=\mathcal{H}(n,v,u)$ is the unique local mild solution to system \eqref{chemo-fluid-eq}.

Finally, we prove the claim regarding $T_{max}$ by contradiction. Assume that
\[
\limsup_{t\to T_{max}^-}\|n(\cdot, t)\|_{q}<\infty,~~~\limsup_{t\to T_{max}^-}\|\nabla v(\cdot, t)\|_{r}<\infty,~~~\limsup_{t\to T_{max}^-}\|u(\cdot, t)\|_{p}<\infty.
\]
Following the same approach as we show $\mathcal{H}_1(n,v,u)\in C((0,T];L^q(\mathbb{R}^d))$, $\nabla\mathcal{H}_2(n,v,u)\in C((0,T];L^r(\mathbb{R}^d))$ and $\mathcal{H}_3(n,v,u)\in C((0,T];L^p(\mathbb{R}^d))$, for any $\epsilon>0$, there exists $\delta>0$ such that if $T_{max}-\delta<t_1<t_2<T_{max}$,
\[
\|n(t_1)-n(t_2)\|_q<\epsilon,~~~~\|\nabla v(t_1)-\nabla v(t_2)\|_r<\epsilon,~~~~\|u(t_1)-u(t_2)\|_p<\epsilon.
\]
Hence, we can define $n(T_{max})$, $v(T_{max})$ and $u(T_{max})$ such that
\[
n\in C((0,T_{max}];L^q(\mathbb{R}^d)),\nabla v \in C((0,T_{max}];L^r(\mathbb{R}^d)),u\in C((0,T_{max}];L^p(\mathbb{R}^d)).
\]
Now we consider
\begin{equation}\label{Tmaxmildsolu-n}
\begin{aligned}
\check{n}=&E_\beta(-(T_{max}+t)^\beta(-\Delta)^{\frac{\alpha}{2}})n_0\\
&-\int_0^{T_{max}}(T_{max}+t-\tau)^{\beta-1}\nabla\cdot E_{\beta,\beta}(-(T_{max}+t-\tau)^\beta(-\Delta)^{\frac{\alpha}{2}})(u n+n\nabla v)\,d\tau\\
&-\int_0^t(t-\tau)^{\beta-1}\nabla\cdot E_{\beta,\beta}(-(t-\tau)^\beta(-\Delta)^{\frac{\alpha}{2}})(\check{u} \check{n}+\check{n}\nabla \check{v})\,d\tau\\
:=&\check{n}_1+\check{n}_2+\check{n}_3,
\end{aligned}
\end{equation}
\begin{equation}\label{Tmaxmildsolu-v}
\begin{aligned}
\check{v}=&E_\beta\big((T_{max}+t)^\beta(-(-\Delta)^{\frac{\alpha}{2}}-\gamma)\big)v_0\\
  &-\int_0^{T_{max}}(T_{max}+t-\tau)^{\beta-1} E_{\beta,\beta}\big((T_{max}+t-\tau)^\beta(-(-\Delta)^{\frac{\alpha}{2}}-\gamma)\big)(u\cdot \nabla v-n)\,d\tau\\
  &-\int_0^t(t-\tau)^{\beta-1}E_{\beta,\beta}\big((t-\tau)^\beta(-(-\Delta)^{\frac{\alpha}{2}}-\gamma)\big)(\check{u}\cdot\nabla \check{v}-\check{n})\,d\tau\\
  :=&\check{v}_1+\check{v}_2+\check{v}_3,
\end{aligned}
\end{equation}
\begin{equation}\label{Tmaxmildsolu-u}
\begin{aligned}
\check{u}=&E_\beta(-(T_{max}+t) ^\beta(-\Delta)^{\frac{\alpha}{2}})u_0\\
  &-\int_0^{T_{max}}(T_{max}+t-\tau)^{\beta-1}E_{\beta,\beta}(-(T_{max}+t-\tau)^\beta(-\Delta)^{\frac{\alpha}{2}})[P((u\cdot\nabla) u+n\nabla\phi)]\,d\tau\\
  &-\int_0^t(t-\tau)^{\beta-1}E_{\beta,\beta}(-(t-\tau)^\beta(-\Delta)^{\frac{\alpha}{2}})[P((\check{u}\cdot\nabla) \check{u}+\check{n}\nabla\phi)]\,d\tau\\
  :=&\check{u}_1+\check{u}_2+\check{u}_3.
 \end{aligned}
\end{equation}
For some $\delta_1>0$, employ \eqref{conMLLpq1} and \eqref{conMLLpqgamma1} in Proposition \ref{pro:MLopconti} to deduce that
\[
\check{n}_1\in C((0,\delta_1];L^q(\mathbb{R}^d)),~~\nabla\check{v}_1\in C((0,\delta_1];L^r(\mathbb{R}^d)),~~\check{u}_1\in C((0,\delta_1];L^p(\mathbb{R}^d)).
\]
Following the same approach as we show $\mathcal{H}_1(n,v,u)\in C((0,T];L^q(\mathbb{R}^d))$, $\nabla\mathcal{H}_2(n,v,u)\in C((0,T];L^r(\mathbb{R}^d))$, $\mathcal{H}_3(n,v,u)\in C((0,T];L^p(\mathbb{R}^d))$, we find that
\[
\check{n}_2\in C((0,\delta_1];L^q(\mathbb{R}^d)),~~\nabla \check{v}_2\in C((0,\delta_1];L^r(\mathbb{R}^d)),~~ \check{u}_2\in C((0,\delta_1];L^p(\mathbb{R}^d)).
\]
Repeating what has been just done, it holds that
\[
\check{n}_3\in C((0,\delta_1];L^q(\mathbb{R}^d)),~~\nabla \check{v}_3\in C((0,\delta_1];L^r(\mathbb{R}^d)),~~ \check{u}_3\in C((0,\delta_1];L^p(\mathbb{R}^d)).
\]
Then we see that \eqref{Tmaxmildsolu-n}-\eqref{Tmaxmildsolu-u} has a unique mild solution $(\check{n},\check{v},\check{u})$ with the property
\[
\check{n}\in C((0,\delta_1];L^q(\mathbb{R}^d)),~\nabla \check{v} \in C((0,\delta_1];L^r(\mathbb{R}^d)),~ \check{u}\in C((0,\delta_1];L^p(\mathbb{R}^d)).
\]
For $t\in (0,\delta_1]$, if we define
\[n(T_{max}+t)=\check{n}(t), ~~v(T_{max}+t)=\check{v}(t), ~~ u(T_{max}+t)=\check{u}(t),
\]
then the fact that $(\check{n},\check{v},\check{u})$ is a mild solution on $(0,T_{max}+\delta_1)$ contradicts with the definition of $T_{max}$.
\end{proof}

\subsection{Proof of Theorem \ref{thm:mildexis}}\label{3.2}

\begin{proof}
\textbf{Step 1.}(Two functional spaces $\mathcal{X},\mathcal{Y}$ and the map $F$)
In this proof, we define
\[
\mathcal{X}:=\{[n_0,v_0,u_0,\phi]|n_0\in L^{\frac{d}{2\alpha-2}}(\mathbb{R}^d),\nabla v_0\in L^{\frac{d}{\alpha-1}}(\mathbb{R}^d), u_0\in L^{\frac{d}{\alpha-1}}(\mathbb{R}^d), \nabla \phi\in L^{d}(\mathbb{R}^d)\}
\]
with the norm
\begin{gather*}
\|[n_0,v_0,u_0,\phi]\|_\mathcal{X}:=\|n_0\|_{\frac{d}{2\alpha-2}}+\|\nabla v_0\|_{\frac{d}{\alpha-1}}+\|u_0\|_{\frac{d}{\alpha-1}}+\|\nabla \phi\|_{d},
\end{gather*}
and
\[
\begin{split}
\mathcal{Y}:=\{[n,v,u]|&t^{\frac{d\beta}{\alpha}(\frac{2\alpha-2}{d}-\frac{1}{q})}n\in C((0,\infty);L^q(\mathbb{R}^d)),
t^{\frac{d\beta}{\alpha}(\frac{\alpha-1}{d}-\frac{1}{r})}\nabla v\in C((0,\infty);L^r(\mathbb{R}^d)),\\
&t^{\frac{d\beta}{\alpha}(\frac{\alpha-1}{d}-\frac{1}{p})}u\in C((0,\infty);L^p(\mathbb{R}^d))\}
\end{split}
\]
with the norm
\begin{gather*}
\|[n,v,u]\|_\mathcal{Y}:=\sup_{0<t<\infty}t^{\frac{d\beta}{\alpha}(\frac{2\alpha-2}{d}-\frac{1}{q})}\|n(t)\|_{q}
+\sup_{0<t<\infty}t^{\frac{d\beta}{\alpha}(\frac{\alpha-1}{d}-\frac{1}{r})}\|\nabla v(t)\|_{r}+\sup_{0<t<\infty}t^{\frac{d\beta}{\alpha}(\frac{\alpha-1}{d}-\frac{1}{p})}\|u(t)\|_{p}.
\end{gather*}
It is clear that the space $\mathcal{X}$ and $\mathcal{Y}$ equipped with the norm $\|\cdot\|_\mathcal{X}$ and $\|\cdot\|_\mathcal{Y}$ are Banach spaces.

For $0<t<\infty$, $[n_0,v_0,u_0,\phi]\in \mathcal{X}$ and $[n,v,u]\in \mathcal{Y}$, we define the map $F$ as follows
\[
F(n_0,v_0,u_0,\phi,n,v,u):=[\tilde{n},\tilde{v},\tilde{u}],
\]
\begin{equation}\label{mapfmildsolu}
\left\{
\begin{aligned}
\tilde{n}(t)&=n(t)-E_\beta(-t^\beta(-\Delta)^{\frac{\alpha}{2}})n_0+\int_0^t(t-\tau)^{\beta-1}E_{\beta,\beta}\big(-(t-\tau)^\beta (-\Delta)^{\frac{\alpha}{2}}\big)(u\cdot\nabla n)\,d\tau\\
&~~~~+\int_0^t(t-\tau)^{\beta-1}\nabla \cdot E_{\beta,\beta}\big(-(t-\tau)^\beta (-\Delta)^{\frac{\alpha}{2}}\big)(n\nabla v)\,d\tau,\\
\tilde{v}(t)&=v(t)-E_\beta(t^\beta(-(-\Delta)^{\frac{\alpha}{2}}-\gamma))v_0-\int_0^t (t-\tau)^{\beta-1}E_{\beta,\beta}\big((t-\tau)^\beta (-(-\Delta)^{\frac{\alpha}{2}}-\gamma)\big)n\,d\tau\\
&~~~~+\int_0^t(t-\tau)^{\beta-1}E_{\beta,\beta}\big((t-\tau)^\beta (-(-\Delta)^{\frac{\alpha}{2}}-\gamma)\big)(u\cdot\nabla v)\,d\tau,\\
\tilde{u}(t)&=u(t)-E_\beta(-t^\beta(-\Delta)^{\frac{\alpha}{2}})u_0+\int_0^t(t-\tau)^{\beta-1}E_{\beta,\beta}(-(t-\tau)^\beta (-\Delta)^{\frac{\alpha}{2}})\\
  &~~~~\times[P((u\cdot\nabla) u+n\nabla\phi)]\,d\tau.
\end{aligned}
\right.
\end{equation}

\textbf{Step 2.}($F:\mathcal{X}\times \mathcal{Y}\rightarrow \mathcal{Y}$ is a continuous map)
To begin with, we show that $t^{\frac{d\beta}{\alpha}(\frac{2\alpha-2}{d}-\frac{1}{q})}\tilde{n}\in C((0,\infty);L^q(\mathbb{R}^d))$.
Based on \eqref{mapfmildsolu}, we estimate $\|\tilde{n}\|_{q}$ as follows.
\begin{equation}\label{nLqesti}
\begin{split}
\|\tilde{n}(t)\|_{q}&\leq\|n(t)\|_{q}+\|E_\beta(-t^\beta(-\Delta)^{\frac{\alpha}{2}})n_0\|_{q}\\
&~~~~+\int_0^t(t-\tau)^{\beta-1}
\|\nabla\cdot E_{\beta,\beta}\big(-(t-\tau)^\beta (-\Delta)^{\frac{\alpha}{2}}\big)(u n)\|_{q}\,d\tau\\
&~~~~+\int_0^t(t-\tau)^{\beta-1}
\|\nabla \cdot E_{\beta,\beta}\big(-(t-\tau)^\beta (-\Delta)^{\frac{\alpha}{2}}\big)(n\nabla v)\|_{q}\,d\tau\\
&:=A_0+A_1+A_2+A_3.
\end{split}
\end{equation}
Now, applying the $L^p-L^q$ estimates of Mittag-Leffler operators in Proposition \ref{pro:MLopLpqesti}, we estimate $A_1,A_2$ and $A_3$, respectively.
For the term $A_1$, recalling \eqref{MLLpq1}, if $q>\frac{d}{2\alpha-2}$, we arrive at
\begin{gather}\label{A1esti}
A_1\leq Ct^{-\frac{d\beta}{\alpha}(\frac{2\alpha-2}{d}-\frac{1}{q})}\| n_0\|_{\frac{d}{2\alpha-2}}.
\end{gather}
As for $A_2$, based on \eqref{MLLpq4} and H\"{o}lder's inequality, we obtain
\begin{gather}\label{A2esti}
\begin{split}
A_2
&\leq C\int_0^t(t-\tau)^{-\frac{d\beta}{\alpha}(\frac{1}{p}+\frac{1}{q}-\frac{1}{q})-\frac{\beta}{\alpha}+\beta-1}
\|un(\tau)\|_{\frac{pq}{p+q}}\,d\tau\\
&\leq C\int_0^t(t-\tau)^{-\frac{d\beta}{\alpha p}-\frac{\beta}{\alpha}+\beta-1}\tau^{\frac{d\beta}{\alpha}(\frac{1}{p}
+\frac{1}{q}-\frac{3\alpha-3}{d})}\tau^{\frac{d\beta}{\alpha}(\frac{\alpha-1}{d}-\frac{1}{p})}\|u(\tau)\|_{p}
\tau^{\frac{d\beta}{\alpha}(\frac{2\alpha-2}{d}-\frac{1}{q})}\|n(\tau)\|_{q}\,d\tau\\
&\leq C \sup_{\tau>0}\tau^{\frac{d\beta}{\alpha}(\frac{\alpha-1}{d}-\frac{1}{p})}\|u(\tau)\|_{p}
\sup_{\tau>0}\tau^{\frac{d\beta}{\alpha}(\frac{2\alpha-2}{d}-\frac{1}{q})}\|n(\tau)\|_{q} \\
&~~~~\times t^{-\frac{d\beta}{\alpha}(\frac{2\alpha-2}{d}-\frac{1}{q})}\mathbf{B}(-\frac{d\beta}{\alpha p}-\frac{\beta}{\alpha}+\beta, \frac{d\beta}{\alpha p}+\frac{d\beta}{\alpha q}-\frac{3\beta(\alpha-1)}{\alpha}+1),
\end{split}
\end{gather}
where conditions (I), (II), (III), (IV) and (V) in Assumption \ref{assum1} ensure that
\[
-\frac{d\beta}{\alpha p}-\frac{\beta}{\alpha}+\beta>0,~~~\frac{d\beta}{\alpha p}+\frac{d\beta}{\alpha q}-\frac{3\beta(\alpha-1)}{\alpha}+1>0.
\]
Hence, we deduce that the Beta function
$
\mathbf{B}(-\frac{d\beta}{\alpha p}-\frac{\beta}{\alpha}+\beta, \frac{d\beta}{\alpha p}+\frac{d\beta}{\alpha q}-\frac{3\beta(\alpha-1)}{\alpha}+1)\leq C
$, then \eqref{A2esti} satisfies the following estimates
\begin{equation}\label{0A2esti}
A_2\leq C t^{-\frac{d\beta}{\alpha}(\frac{2\alpha-2}{d}-\frac{1}{q})} \sup_{\tau>0}\tau^{\frac{d\beta}{\alpha}(\frac{\alpha-1}{d}-\frac{1}{p})}\|u(\tau)\|_{p}
\sup_{\tau>0}\tau^{\frac{d\beta}{\alpha}(\frac{2\alpha-2}{d}-\frac{1}{q})}\|n(\tau)\|_{q}.
\end{equation}
As regards $A_3$, utilizing \eqref{MLLpq4} and H\"{o}lder's inequality again, we derive
\begin{gather}\label{A3esti}
\begin{split}
A_3&\leq C\int_0^t(t-\tau)^{-\frac{d\beta}{\alpha}(\frac{1}{r}+\frac{1}{q}-\frac{1}{q})-\frac{\beta}{\alpha}+\beta-1}
\|(n\nabla v)(\tau)\|_{\frac{qr}{q+r}}\,d\tau\\
&\leq C\int_0^t(t-\tau)^{-\frac{d\beta}{\alpha r}-\frac{\beta}{\alpha}+\beta-1}\tau^{\frac{d\beta}{\alpha}(\frac{1}{r}
+\frac{1}{q}-\frac{3\alpha-3}{d})}\tau^{\frac{d\beta}{\alpha}(\frac{2\alpha-2}{d}-\frac{1}{q})}\|n(\tau)\|_{q}
\tau^{\frac{d\beta}{\alpha}(\frac{\alpha-1}{d}-\frac{1}{r})}\|\nabla v(\tau)\|_{r}\,d\tau\\
&\leq C \sup_{\tau>0}\tau^{\frac{d\beta}{\alpha}(\frac{\alpha-1}{d}-\frac{1}{r})}\|\nabla v(\tau)\|_{r}\sup_{\tau>0}\tau^{\frac{d\beta}{\alpha}(\frac{2\alpha-2}{d}-\frac{1}{q})}\|n(\tau)\|_{q}\\
&~~~~~~~\times t^{-\frac{d\beta}{\alpha}(\frac{2\alpha-2}{d}-\frac{1}{q})}\mathbf{B}(-\frac{d\beta}{\alpha r}-\frac{\beta}{\alpha}+\beta, \frac{d\beta}{\alpha q}+\frac{d\beta}{\alpha r}-\frac{3\beta(\alpha-1)}{\alpha}+1).
\end{split}
\end{gather}
Since conditions (I), (II), (III), (IV) and (V) in Assumption \ref{assum1} imply that
\[
-\frac{d\beta}{\alpha r}-\frac{\beta}{\alpha}+\beta>0,~~\frac{d\beta}{\alpha q}+\frac{d\beta}{\alpha r}-\frac{3\beta(\alpha-1)}{\alpha}+1>0,
\]
we have
\begin{equation}\label{0A3esti}
A_3\leq C t^{-\frac{d\beta}{\alpha}(\frac{2\alpha-2}{d}-\frac{1}{q})} \sup_{\tau>0}\tau^{\frac{d\beta}{\alpha}(\frac{\alpha-1}{d}-\frac{1}{r})}\|\nabla v(\tau)\|_{p}
\sup_{\tau>0}\tau^{\frac{d\beta}{\alpha}(\frac{2\alpha-2}{d}-\frac{1}{q})}\|n(\tau)\|_{q}.
\end{equation}
Combining the above estimates \eqref{A1esti}, \eqref{0A2esti} and \eqref{0A3esti}, one has
\begin{equation}\label{nestbdd}
\begin{split}
\sup_{t>0}&t^{\frac{d\beta}{\alpha}(\frac{2\alpha-2}{d}-\frac{1}{q})}\|\tilde{n}(t)\|_{q}\leq C\| n_0\|_{\frac{d}{2\alpha-2}}
+C\sup_{t>0}t^{\frac{d\beta}{\alpha}(\frac{2\alpha-2}{d}-\frac{1}{q})}\|n(t)\|_{q}\\
&\times\big(1+\sup_{t>0}t^{\frac{d\beta}{\alpha}(\frac{\alpha-1}{d}-\frac{1}{p})}\|u(t)\|_{p}+
\sup_{t>0}t^{\frac{d\beta}{\alpha}(\frac{\alpha-1}{d}-\frac{1}{r})}\|\nabla v(t)\|_{r}\big).
\end{split}
\end{equation}
The time continuity can be proved with the same argument as the proof of Theorem \ref{thm:localmildexis}, then we claim that
\[
t^{\frac{d\beta}{\alpha}(\frac{2\alpha-2}{d}-\frac{1}{q})}\tilde{n}\in C((0,\infty);L^q(\mathbb{R}^d)).
\]

In views of \eqref{mapfmildsolu}, $\|\nabla \tilde{v}(t)\|_{r}$ can be written as
\begin{equation}\label{vtildeLresti}
\begin{split}
\|\nabla\tilde{v}(t)\|_{r}&\leq\|\nabla v(t)\|_{r}+\|E_\beta(t^\beta(-(-\Delta)^{\frac{\alpha}{2}}-\gamma))\nabla v_0\|_{r}\\
&~~~~+\int_0^t(t-\tau)^{\beta-1}\|\nabla E_{\beta,\beta}\big((t-\tau)^\beta (-(-\Delta)^{\frac{\alpha}{2}}-\gamma)\big)(u\cdot\nabla v)(\tau)\|_{r}\,d\tau\\
&~~~~+\int_0^t(t-\tau)^{\beta-1}\|\nabla E_{\beta,\beta}\big((t-\tau)^\beta (-(-\Delta)^{\frac{\alpha}{2}}-\gamma)\big)n(\tau)\|_{r}\,d\tau\\
&:=Q_0+Q_1+Q_2+Q_3.
\end{split}
\end{equation}
For the term $Q_1$, by virtue of \eqref{MLLpq1} in Proposition \ref{pro:MLopLpqesti}, it holds that for $r>\frac{d}{\alpha-1}$,
\begin{gather}\label{Q1esti}
Q_1\leq Ct^{-\frac{d\beta}{\alpha}(\frac{\alpha-1}{d}-\frac{1}{r})}\|\nabla v_0\|_{\frac{d}{\alpha-1}}.
\end{gather}
As for $Q_2$, based on \eqref{MLLpqgamma4} and H\"{o}lder's inequality, we calculate that
\begin{gather}\label{Q2esti1}
\begin{split}
Q_2&\leq C\int_0^t(t-\tau)^{-\frac{d\beta}{\alpha}(\frac{1}{p}+\frac{1}{r}-\frac{1}{r})-\frac{\beta}{\alpha}+\beta-1}
\|(u\nabla v)(\tau)\|_{\frac{pr}{p+r}}\,d\tau\\
&\leq C\int_0^t(t-\tau)^{-\frac{d\beta}{\alpha p}-\frac{\beta}{\alpha}+\beta-1}\|u(\tau)\|_{p}\|\nabla v(\tau)\|_{r}\,d\tau\\
&\leq C\int_0^t(t-\tau)^{-\frac{d\beta}{\alpha p}-\frac{\beta}{\alpha}+\beta-1}\tau^{\frac{d\beta}{\alpha}(\frac{1}{p}
+\frac{1}{r}-\frac{2\alpha-2}{d})}\tau^{\frac{d\beta}{\alpha}(\frac{\alpha-1}{d}-\frac{1}{p})}\|u(\tau)\|_{p}
\tau^{\frac{d\beta}{\alpha}(\frac{\alpha-1}{d}-\frac{1}{r})}\|\nabla v(\tau)\|_{r}\,d\tau\\
&\leq C \sup_{\tau>0}\tau^{\frac{d\beta}{\alpha}(\frac{\alpha-1}{d}-\frac{1}{p})}\|u(\tau)\|_{p}
\sup_{\tau>0}\tau^{\frac{d\beta}{\alpha}(\frac{\alpha-1}{d}-\frac{1}{r})}\|\nabla v(\tau)\|_{r} \\
&~~~~\times t^{-\frac{d\beta}{\alpha}(\frac{\alpha-1}{d}-\frac{1}{r})}\mathbf{B}(-\frac{d\beta}{\alpha p}-\frac{\beta}{\alpha}+\beta, \frac{d\beta}{\alpha p}+\frac{d\beta}{\alpha r}-\frac{2\beta(\alpha-1)}{\alpha}+1)\\
&\leq Ct^{-\frac{d\beta}{\alpha}(\frac{\alpha-1}{d}-\frac{1}{r})}\sup_{\tau>0}\tau^{\frac{d\beta}{\alpha}(\frac{\alpha-1}{d}-\frac{1}{p})}\|u\|_{p}
\sup_{\tau>0}\tau^{\frac{d\beta}{\alpha}(\frac{\alpha-1}{d}-\frac{1}{r})}\|\nabla v\|_{r},
\end{split}
\end{gather}
where we have used the fact
$$\mathbf{B}(-\frac{d\beta}{\alpha p}-\frac{\beta}{\alpha}+\beta, \frac{d\beta}{\alpha p}+\frac{d\beta}{\alpha r}-\frac{2\beta(\alpha-1)}{\alpha}+1)\leq C$$
due to conditions (I), (II), (III), (IV) and (V) in Assumption \ref{assum1} to derive the last inequality of \eqref{Q2esti1}.
Regarding $Q_3$, using \eqref{MLLpqgamma4} and H\"{o}lder's inequality, we get
\begin{gather}\label{Q3esti}
\begin{split}
Q_3&\leq C\int_0^t(t-\tau)^{-\frac{d\beta}{\alpha}(\frac{1}{q}-\frac{1}{r})-\frac{\beta}{\alpha}+\beta-1}
\|n(\tau)\|_{q}\,d\tau\\
&\leq C\int_0^t(t-\tau)^{-\frac{d\beta}{\alpha q}+\frac{d\beta}{\alpha r}-\frac{\beta}{\alpha}+\beta-1}\tau^{\frac{d\beta}{\alpha}(\frac{1}{q}-\frac{2\alpha-2}{d})}
\tau^{\frac{d\beta}{\alpha}(\frac{2\alpha-2}{d}-\frac{1}{q})}\|n(\tau)\|_{q}\,d\tau\\
&\leq C t^{-\frac{d\beta}{\alpha}(\frac{\alpha-1}{d}-\frac{1}{r})}\mathbf{B}(-\frac{d\beta}{\alpha q}+\frac{d\beta}{\alpha r}-\frac{\beta}{\alpha}+\beta, \frac{d\beta}{\alpha q}-\frac{2\beta(\alpha-1)}{\alpha}+1)\sup_{\tau>0}\tau^{\frac{d\beta}{\alpha}(\frac{2\alpha-2}{d}-\frac{1}{q})}\|n(\tau)\|_{q}\\
&\leq C t^{-\frac{d\beta}{\alpha}(\frac{\alpha-1}{d}-\frac{1}{r})}\sup_{\tau>0}\tau^{\frac{d\beta}{\alpha}(\frac{2\alpha-2}{d}-\frac{1}{q})}\|n(\tau)\|_{q} ,
\end{split}
\end{gather}
where Assumption \ref{assum1} ensure that $\mathbf{B}(-\frac{d\beta}{\alpha q}+\frac{d\beta}{\alpha r}-\frac{\beta}{\alpha}+\beta, \frac{d\beta}{\alpha q}-\frac{2\beta(\alpha-1)}{\alpha}+1)\leq C$.

Thus, \eqref{Q1esti}, \eqref{Q2esti1} and \eqref{Q3esti} lead to the following estimate,
\begin{equation}\label{vtildeestbdd}
\begin{split}
\sup_{t>0}t^{\frac{d\beta}{\alpha}(\frac{\alpha-1}{d}-\frac{1}{r})}\|\nabla \tilde{v}\|_{r}&\leq C\|\nabla v_0\|_{\frac{d}{\alpha-1}}+\sup_{t>0}\tau^{\frac{d\beta}{\alpha}(\frac{\alpha-1}{d}-\frac{1}{r})}\|\nabla v(\tau)\|_{r}\\
&~~+C\sup_{\tau>0}\tau^{\frac{d\beta}{\alpha}(\frac{\alpha-1}{d}-\frac{1}{p})}\|u(\tau)\|_{p}
\sup_{\tau>0}\tau^{\frac{d\beta}{\alpha}(\frac{\alpha-1}{d}-\frac{1}{r})}\|\nabla v(\tau)\|_{r}\\
&~~+C\sup_{\tau>0}\tau^{\frac{d\beta}{\alpha}(\frac{2\alpha-2}{d}-\frac{1}{q})}\|n(\tau)\|_{q}.
\end{split}
\end{equation}
The time continuity of $t^{\frac{d\beta}{\alpha}(\frac{\alpha-1}{d}-\frac{1}{r})}\nabla\tilde{v}$ can be verified similiarly as the proof of Theorem \ref{thm:localmildexis}. Therefore, we derive that
\[
t^{\frac{d\beta}{\alpha}(\frac{\alpha-1}{d}-\frac{1}{r})}\nabla\tilde{v}\in C((0,\infty);L^r(\mathbb{R}^d)).
\]

Next, we shall estimate $\|\tilde{u}\|_{p}$ with the aid of \eqref{mapfmildsolu},
\begin{equation}\label{uLpesti}
\begin{split}
\|\tilde{u}(t)\|_{p}&\leq\|u(t)\|_{p}+\|E_\beta(-t^\beta(-\Delta)^{\frac{\alpha}{2}}) u_0\|_{p}\\
&~~~+\int_0^t(t-\tau)^{\beta-1}\|\nabla\cdot E_{\beta,\beta}\big(-(t-\tau)^\beta (-\Delta)^{\frac{\alpha}{2}}\big)P(u\otimes u)(\tau)\|_{p}\,d\tau\\
&~~~+\int_0^t(t-\tau)^{\beta-1}\| E_{\beta,\beta}\big(-(t-\tau)^\beta (-\Delta)^{\frac{\alpha}{2}}\big)P(n\nabla\phi)(\tau)\|_{p}\,d\tau\\
&:=B_0+B_1+B_2+B_3.
\end{split}
\end{equation}
To estimate the term $B_1$, for $p>\frac{d}{\alpha-1}$, Proposition \ref{pro:MLopLpqesti} leads to
\begin{gather}\label{B1esti}
B_1\leq Ct^{-\frac{d\beta}{\alpha}(\frac{\alpha-1}{d}-\frac{1}{p})}\| u_0\|_{\frac{d}{\alpha-1}}.
\end{gather}
For the term $B_2$, due to conditions (I), (II), (III), (IV) and (V) in Assumption \ref{assum1}, it holds
\[
\mathbf{B}(\beta-\frac{d\beta}{\alpha p}-\frac{\beta}{\alpha}, \frac{2d\beta}{\alpha p}-\frac{2\beta(\alpha-1)}{\alpha}+1)\leq C.
\]
Combining the boundedness of the projection operator $P$, \eqref{MLLpq4} in Proposition \ref{pro:MLopLpqesti} and H\"{o}lder's inequality, we obtain
\begin{gather}\label{B2esti}
\begin{split}
B_2&\leq C\int_0^t(t-\tau)^{-\frac{d\beta}{\alpha}(\frac{2}{p}-\frac{1}{p})-\frac{\beta}{\alpha}+\beta-1}
\|P(u\otimes u)(\tau)\|_{\frac{p}{2}}\,d\tau\\
&\leq C\int_0^t(t-\tau)^{-\frac{d\beta}{\alpha p}-\frac{\beta}{\alpha}+\beta-1}\tau^{\frac{2d\beta}{\alpha}(\frac{1}{p}
-\frac{\alpha-1}{d})}\big(\tau^{\frac{d\beta}{\alpha}(\frac{\alpha-1}{d}-\frac{1}{p})}\|u(\tau)\|_{p}\big)^2\,d\tau\\
&\leq Ct^{-\frac{d\beta}{\alpha}(\frac{\alpha-1}{d}-\frac{1}{p})}\mathbf{B}(\beta-\frac{d\beta}{\alpha p}-\frac{\beta}{\alpha}, \frac{2d\beta}{\alpha p}-\frac{2\beta(\alpha-1)}{\alpha}+1)\big(\sup_{\tau>0}\tau^{\frac{d\beta}{\alpha}(\frac{\alpha-1}{d}-\frac{1}{p})}\|u(\tau)\|_{p}\big)^2\\
&\leq Ct^{-\frac{d\beta}{\alpha}(\frac{\alpha-1}{d}-\frac{1}{p})}\big(\sup_{\tau>0}\tau^{\frac{d\beta}{\alpha}(\frac{\alpha-1}{d}-\frac{1}{p})}\|u(\tau)\|_{p}\big)^2.
\end{split}
\end{gather}
As regards $B_3$, the boundedness of the projection operator $P$, \eqref{MLLpq2} and H\"{o}lder's inequality indicate that
\begin{gather}\label{B3esti}
\begin{split}
B_3&\leq C\int_0^t(t-\tau)^{-\frac{d\beta}{\alpha}(\frac{1}{q}+\frac{1}{d}-\frac{1}{p})+\beta-1}
\|P(n\nabla\phi)(\tau)\|_{\frac{qd}{q+d}}\,d\tau\\
&\leq C\int_0^t(t-\tau)^{-\frac{d\beta}{\alpha}(\frac{1}{q}+\frac{1}{d}-\frac{1}{p})+\beta-1}\tau^{\frac{d\beta}{\alpha}(\frac{1}{q}
-\frac{2\alpha-2}{d})}\tau^{\frac{d\beta}{\alpha}(\frac{2\alpha-2}{d}-\frac{1}{q})}\|n(\tau)\|_{q}
\|\nabla\phi\|_{d}\,d\tau\\
&\leq C t^{-\frac{d\beta}{\alpha}(\frac{\alpha-1}{d}-\frac{1}{p})}\mathbf{B}(\beta-\frac{d\beta}{\alpha q}+\frac{d\beta}{\alpha p}-\frac{\beta}{\alpha}, \frac{d\beta}{\alpha q}-\frac{2\beta(\alpha-1)}{\alpha}+1)\sup_{\tau>0}\tau^{\frac{d\beta}{\alpha}(\frac{2\alpha-2}{d}-\frac{1}{q})}\|n(\tau)\|_{q}\|\nabla\phi\|_{d}.\\
\end{split}
\end{gather}
Since conditions (I), (II), (III), (IV) and (V) in Assumption \ref{assum1} imply that
\[-\frac{d\beta}{\alpha q}+\frac{d\beta}{\alpha p}-\frac{\beta}{\alpha}+\beta>0,~~\frac{d\beta}{\alpha q}-\frac{2\beta(\alpha-1)}{\alpha}+1>0,\]
then we have
\begin{equation}\label{0B3esti}
B_3\leq C t^{-\frac{d\beta}{\alpha}(\frac{\alpha-1}{d}-\frac{1}{p})}\sup_{\tau>0}\tau^{\frac{d\beta}{\alpha}(\frac{2\alpha-2}{d}-\frac{1}{q})}
\|n(\tau)\|_{q}\|\nabla\phi\|_{d}.
\end{equation}
Combine \eqref{B1esti}, \eqref{B2esti} and \eqref{0B3esti} to yield
\begin{equation}\label{uestbdd}
\begin{split}
\sup_{t>0}t^{\frac{d\beta}{\alpha}(\frac{\alpha-1}{d}-\frac{1}{p})}\|\tilde{u}\|_{p}\leq& C\| u_0\|_{\frac{d}{\alpha-1}}
+C\|\nabla\phi\|_{d}\sup_{\tau>0}\tau^{\frac{d\beta}{\alpha}(\frac{2\alpha-2}{d}-\frac{1}{q})}\|n(\tau)\|_{q}\\
&+C\sup_{\tau>0}\tau^{\frac{d\beta}{\alpha}(\frac{\alpha-1}{d}-\frac{1}{p})}\| u(\tau)\|_{p}\big(1+\sup_{t>0}\tau^{\frac{d\beta}{\alpha}(\frac{\alpha-1}{d}-\frac{1}{p})}\| u(\tau)\|_{p}\big).
\end{split}
\end{equation}
Hence it follows that
\[
t^{\frac{d\beta}{\alpha}(\frac{\alpha-1}{d}-\frac{1}{p})}\tilde{u}\in C((0,\infty);L^p(\mathbb{R}^d)).
\]

Then, from \eqref{nestbdd},\eqref{vtildeestbdd} and \eqref{uestbdd}, we conclude that $F(n_0,v_0,u_0,\phi,n,v,u)\in \mathcal{Y}$ with
\[
\|F(n_0,v_0,u_0,\phi,n,v,u)\|_{\mathcal{Y}}
\leq C\|[n_0,v_0,u_0,\phi]\|_{\mathcal{X}}+C\|[n,v,u]\|_{\mathcal{Y}}(1+\|[n,v,u]\|_{\mathcal{Y}}+\|\nabla\phi\|_{d}).
\]

\textbf{Step 3.}(The map $F(n_0,v_0,u_0,\phi,\cdot,\cdot,\cdot)$ is of class $C^1$ from $\mathcal{Y}$ into itself) For each $[n,v,u]\in \mathcal{Y}$, we define a linear map $\mathcal{L}_{[n,v,u]}(\bar{n},\bar{v},\bar{u})=[\bar{N},\bar{V},\bar{U}]$ on $\mathcal{Y}$ by
\begin{equation}\label{maplmildsolu}
\left\{
\begin{aligned}
\bar{N}(t)=&\bar{n}(t)+\int_0^t(t-\tau)^{\beta-1}E_{\beta,\beta}(-(t-\tau)^\beta(-\Delta)^{\frac{\alpha}{2}})(u\cdot\nabla \bar{n}+\bar{u}\cdot\nabla n)\,d\tau\\
&+\int_0^t(t-\tau)^{\beta-1}\nabla\cdot E_{\beta,\beta}(-(t-\tau)^\beta(-\Delta)^{\frac{\alpha}{2}})(\bar{n}\nabla v+n\nabla \bar{v})\,d\tau,\\
  \bar{V}(t)=&\bar{v}(t)+\int_0^t(t-\tau)^{\beta-1}E_{\beta,\beta}((t-\tau)^\beta(-(-\Delta)^{\frac{\alpha}{2}}-\gamma))(\bar{u}\cdot\nabla v+u\cdot\nabla \bar{v})\,d\tau\\
  &-\int_0^t (t-\tau)^{\beta-1}E_{\beta,\beta}((t-\tau)^\beta(-(-\Delta)^{\frac{\alpha}{2}}-\gamma))\bar{n}\,d\tau,\\
  \bar{U}(t)=&\bar{u}(t)+\int_0^t(t-\tau)^{\beta-1}E_{\beta,\beta}(-(t-\tau)^\beta(-\Delta)^{\frac{\alpha}{2}})[P((\bar{u}\cdot\nabla)u+(u\cdot\nabla)\bar{u})]\,d\tau\\
  &+\int_0^t (t-\tau)^{\beta-1}E_{\beta,\beta}(-(t-\tau)^\beta(-\Delta)^{\frac{\alpha}{2}})P(\bar{n}\nabla\phi)\,d\tau.
\end{aligned}
\right.
\end{equation}
We intend to show that for each fixed $[n_0,v_0,u_0,\phi]\in \mathcal{X}$, $\mathcal{L}_{[n,v,u]}$ is the Fr\'{e}chet derivative of $F(n_0,v_0,u_0,\phi,n,v,u)$ at $[n,v,u]\in \mathcal{Y}$. We define $[N,V,U]$ by
\[
[N,V,U]:=F(n_0,v_0,u_0,\phi,n+\bar{n},v+\bar{v},u+\bar{u})-F(n_0,v_0,u_0,\phi,n,v,u)-\mathcal{L}_{[n,v,u]}(\bar{n},\bar{v},\bar{u}).
\]
In view of \eqref{mapfmildsolu} and \eqref{maplmildsolu}, it is not difficult to find that
\begin{gather}\label{Nt}
\begin{split}
N(t)=&n(t)+\bar{n}(t)-E_\beta(-t^\beta(-\Delta)^{\frac{\alpha}{2}})n_0\\
&+\int_0^t(t-\tau)^{\beta-1}E_{\beta,\beta}(-(t-\tau)^\beta(-\Delta)^{\frac{\alpha}{2}})\big((u+\bar{u})\cdot\nabla (n+\bar{n})\big)\,d\tau\\
&+\int_0^t(t-\tau)^{\beta-1}\nabla \cdot E_{\beta,\beta}(-(t-\tau)^\beta(-\Delta)^{\frac{\alpha}{2}})\big((n+\bar{n})\nabla (v+\bar{v})\big)\,d\tau\\
&-\big(n(t)-E_\beta(-t^\beta(-\Delta)^{\frac{\alpha}{2}})n_0+\int_0^t(t-\tau)^{\beta-1}E_{\beta,\beta}(-(t-\tau)^\beta(-\Delta)^{\frac{\alpha}{2}})(u\cdot\nabla n)\,d\tau\\
&+\int_0^t(t-\tau)^{\beta-1}\nabla \cdot E_{\beta,\beta}(-(t-\tau)^\beta(-\Delta)^{\frac{\alpha}{2}})(n\nabla v)\,d\tau\big)\\
&-\big(\bar{n}(t)+\int_0^t(t-\tau)^{\beta-1}E_{\beta,\beta}(-(t-\tau)^\beta(-\Delta)^{\frac{\alpha}{2}})(u\cdot\nabla \bar{n}+\bar{u}\cdot\nabla n)\,d\tau\\
&+\int_0^t(t-\tau)^{\beta-1}\nabla\cdot E_{\beta,\beta}(-(t-\tau)^\beta(-\Delta)^{\frac{\alpha}{2}})(\bar{n}\nabla v+n\nabla \bar{v})\,d\tau\big)\\
=&\int_0^t(t-\tau)^{\beta-1} E_{\beta,\beta}(-(t-\tau)^\beta(-\Delta)^{\frac{\alpha}{2}})(\bar{u}\cdot\nabla \bar{n}+\bar{n}\nabla \bar{v})\,d\tau.
\end{split}
\end{gather}
Due to the similar estimates in \eqref{A2esti} and \eqref{A3esti}, one obtains
\begin{gather}\label{lNmildsolu}
\begin{split}
\|N(t)\|_{q}=&\|\int_0^t(t-\tau)^{\beta-1} E_{\beta,\beta}(-(t-\tau)^\beta(-\Delta)^{\frac{\alpha}{2}})(\bar{u}\cdot\nabla \bar{n})\,d\tau\|_q\\
&+\|\int_0^t(t-\tau)^{\beta-1}\nabla\cdot E_{\beta,\beta}(-(t-\tau)^\beta(-\Delta)^{\frac{\alpha}{2}})(\bar{n}\nabla \bar{v})\,d\tau\|_q\\
\leq& C \sup_{\tau>0}\tau^{\frac{d\beta}{\alpha}(\frac{\alpha-1}{d}-\frac{1}{p})}\|\bar{u}(\tau)\|_{p}
\sup_{\tau>0}\tau^{\frac{d\beta}{\alpha}(\frac{2\alpha-2}{d}-\frac{1}{q})}\|\bar{n}(\tau)\|_{q} \\
&~~~~\times t^{-\frac{d\beta}{\alpha}(\frac{2\alpha-2}{d}-\frac{1}{q})}\mathbf{B}(-\frac{d\beta}{\alpha p}-\frac{\beta}{\alpha}+\beta, \frac{d\beta}{\alpha p}+\frac{d\beta}{\alpha q}-\frac{3\beta(\alpha-1)}{\alpha}+1)\\
&+C \sup_{\tau>0}\tau^{\frac{d\beta}{\alpha}(\frac{\alpha-1}{d}-\frac{1}{r})}\|\nabla \bar{v}(\tau)\|_{r}\sup_{\tau>0}\tau^{\frac{d\beta}{\alpha}(\frac{2\alpha-2}{d}-\frac{1}{q})}\|\bar{n}(\tau)\|_{q}\\
&~~~~\times t^{-\frac{d\beta}{\alpha}(\frac{2\alpha-2}{d}-\frac{1}{q})}\mathbf{B}(-\frac{d\beta}{\alpha r}-\frac{\beta}{\alpha}+\beta, \frac{d\beta}{\alpha q}+\frac{d\beta}{\alpha r}-\frac{3\beta(\alpha-1)}{\alpha}+1).
\end{split}
\end{gather}

By \eqref{mapfmildsolu} and \eqref{maplmildsolu}, it holds that
\begin{gather}\label{Vt}
\nabla V(t)=\int_0^t(t-\tau)^{\beta-1}\nabla E_{\beta,\beta}((t-\tau)^\beta(-(-\Delta)^\frac{\alpha}{2}-\gamma))(\bar{u}\cdot\nabla \bar{v})(\tau)\,d\tau
\end{gather}
with
\begin{gather}\label{lVmildsolu}
\begin{split}
\|\nabla V(t)\|_{r}\leq& C \sup_{\tau>0}\tau^{\frac{d\beta}{\alpha}(\frac{\alpha-1}{d}-\frac{1}{p})}\|\bar{u}(\tau)\|_{p}
\sup_{\tau>0}\tau^{\frac{d\beta}{\alpha}(\frac{\alpha-1}{d}-\frac{1}{r})}\|\nabla \bar{v}(\tau)\|_{r} \\
&~~~\times t^{-\frac{d\beta}{\alpha}(\frac{\alpha-1}{d}-\frac{1}{r})}\mathbf{B}(-\frac{d\beta}{\alpha p}-\frac{\beta}{\alpha}+\beta, \frac{d\beta}{\alpha p}+\frac{d\beta}{\alpha r}-\frac{2\beta(\alpha-1)}{\alpha}+1).
\end{split}
\end{gather}

In the same way as above, we see from \eqref{mapfmildsolu} and \eqref{maplmildsolu} that
\begin{gather}\label{Ut}
U(t)=\int_0^t(t-\tau)^{\beta-1}E_{\beta,\beta}(-(t-\tau)^\beta(-\Delta)^\frac{\alpha}{2})P((\bar{u}\cdot\nabla) \bar{u})(\tau)\,d\tau,
\end{gather}
and it holds
\begin{gather}\label{lUmildsolu}
\begin{split}
\| U(t)\|_{p}\leq& C t^{-\frac{d\beta}{\alpha}(\frac{\alpha-1}{d}-\frac{1}{p})}\mathbf{B}(-\frac{d\beta}{\alpha p}-\frac{\beta}{\alpha}+\beta, \frac{2d\beta}{\alpha p}-\frac{2\beta(\alpha-1)}{\alpha}+1)\\
&~~~\times\big(\sup_{\tau>0}\tau^{\frac{d\beta}{\alpha}(\frac{\alpha-1}{d}-\frac{1}{p})}\|\bar{u}(\tau)\|_{p}\big)^2.
\end{split}
\end{gather}
Therefore, \eqref{lNmildsolu}, \eqref{lVmildsolu}, \eqref{lUmildsolu} yield that for each $[n_0,v_0,u_0,\phi]\in \mathcal{X}$ and each $[n,v,u]\in \mathcal{Y}$
\begin{gather*}
\begin{split}
&\lim_{\|[\bar{n},\bar{v},\bar{u}]\|_\mathcal{Y}\rightarrow 0}\frac{\|[N,V,U]\|_\mathcal{Y}}{\|[\bar{n},\bar{v},\bar{u}]\|_\mathcal{Y}}\\
=&\lim_{\|[\bar{n},\bar{v},\bar{u}]\|_\mathcal{Y}\rightarrow 0}\frac{\big(\|F(n_0,v_0,u_0,\phi,n+\bar{n},v+\bar{v},u+\bar{u})-F(n_0,v_0,u_0,\phi,n,v,u)-\mathcal{L}_{[n,v,u]}(\bar{n},\bar{v},\bar{u})\|_\mathcal{Y}\big)}
{\|[\bar{n},\bar{v},\bar{u}]\|_\mathcal{Y}}\\
=&0.
\end{split}
\end{gather*}
This implies that the Fr\'{e}chet derivative of $F$ at point $[n_0,v_0,u_0,\phi,n+\bar{n},v+\bar{v},u+\bar{u}]\in \mathcal{X}\times \mathcal{Y}$ in the direction to $[n,v,u]$ is equal to $\mathcal{L}_{[n,v,u]}(\bar{n},\bar{v},\bar{u})$.

\textbf{Step 4.}(Fr\'{e}chet derivative $\mathcal{L}_{[n,v,u]}(\bar{n},\bar{v},\bar{u})$ at $[n,v,u]=[0,0,0]$ is a bijective mapping)
From Step 3, for $[\bar{n},\bar{v},\bar{u}]\in \mathcal{Y}$, we have an expression $\mathcal{L}_{[0,0,0]}(\bar{n},\bar{v},\bar{u})=[\bar{N}_0,\bar{V}_0,\bar{U}_0]$ as
\[
\bar{N}_0(t)=\bar{n}(t),~\bar{V}_0(t)=\bar{v}(t)-\int_0^t (t-\tau)^{\beta-1}E_{\beta,\beta}((t-\tau)^\beta(-(-\Delta)^{\frac{\alpha}{2}}-\gamma))\bar{n}(\tau)\,d\tau,~\bar{U}_0(t)=\bar{u}(t).
\]
Hence it is easy to see that $\bar{N}_0=\bar{V}_0=\bar{U}_0=0$ implies that $\bar{n}=\bar{v}=\bar{u}=0$, which suggests that $\mathcal{L}_{[0,0,0]}$ is injective.

On the other hand, for every $[\bar{N}_0,\bar{V}_0,\bar{U}_0]\in \mathcal{Y}$, we may take $[\bar{n},\bar{v},\bar{u}]\in \mathcal{Y}$ as
\[
\bar{n}(t)=\bar{N}_0(t),~ \bar{v}(t)=\bar{V}_0(t)+\int_0^t (t-\tau)^{\beta-1}E_{\beta,\beta}((t-\tau)^\beta(-(-\Delta)^{\frac{\alpha}{2}}-\gamma))\bar{N}_0(\tau)\,d\tau,~\bar{u}(t)=\bar{U}_0(t),
\]
so that
\[
\mathcal{L}_{[0,0,0]}(\bar{n},\bar{v},\bar{u})=[\bar{N}_0,\bar{V}_0,\bar{U}_0].
\]
Therefore, we prove that $\mathcal{L}_{[0,0,0]}$ is surjective from $\mathcal{Y}$ onto itself.

\textbf{Step 5.}(Existence and uniqueness)
Now, it follows from Banach implicit function theorem that there is a $C^1$-map $g$ defined as
\[
g: \mathcal{X}_M\rightarrow \mathcal{Y}_M,
\]
and there exists some $M>0$ such that
\[
\begin{split}
g(0,0,0,0)&=[0,0,0],\\
F(n_0,v_0,u_0,\phi,g(n_0,v_0,u_0,\phi))&=[0,0,0],~~\forall [n_0,v_0,u_0,\phi]\in \mathcal{X}_M,
\end{split}
\]
where
\[
\begin{split}
\mathcal{X}_M:&=\{[n_0,v_0,u_0,\phi]\in \mathcal{X}:\|[n_0,v_0,u_0,\phi]\|_\mathcal{X}<M\},\\
\mathcal{Y}_M:&=\{[n,v,u]\in \mathcal{Y}:\|[n,v,u]\|_\mathcal{Y}<M\}.
\end{split}
\]

Thus, the existence and uniqueness of mild solutions to system \eqref{chemo-fluid-eq} are the consequences of the existence of the continuously differential map $g$ from $\mathcal{X}_M$ to $\mathcal{Y}_M$.
\end{proof}

\section{Existence of mild solutions in fractional homogeneous Sobolev spaces}\label{sec4}

In this section, we investigate the existence of the solution to \eqref{chemo-fluid-eq} in fractional homogeneous  Sobolev spaces. The methods of proving the local and global existence result are still Banach fixed point theorem and Banach implicit function theorem. The proofs in this section are similar to Section \ref{3.1} and Section \ref{3.2}.

\subsection{Proof of Theorem \ref{thm:localhighregu}}\label{4.1}

\begin{proof}
In this proof, we proceed as the same argument with the proof of Theorem \ref{thm:localmildexis} using Banach fixed point theorem. According to the assumptions in Theorem \ref{thm:localhighregu}, the initial data $n_0\in H^{\mu,q}(\mathbb{R}^d), \nabla v_0\in H^{\mu,r}(\mathbb{R}^d)$ and $u_0\in H^{\mu,p}(\mathbb{R}^d)$, there exists a constant $\mathcal{M}>0$ such that the initial data satisfy
\[
\|n_0\|_{H^{\mu,q}}+\|\nabla v_0\|_{H^{\mu,r}}+\|u_0\|_{H^{\mu,p}}\leq \mathcal{M}.
\]
Define the Banach space
\[
X_{\mu,T}:=\{(n,v,u):n\in C((0,T];H^{\mu,q}(\mathbb{R}^d)),\nabla v \in C((0,T];H^{\mu,r}(\mathbb{R}^d)),u\in C((0,T];H^{\mu,p}(\mathbb{R}^d))\}
\]
with the norm given by
\begin{gather*}
\|(n,v,u)\|_{X_{\mu,T}}:=\sup_{0< t\leq T}\|n(t)\|_{H^{\mu,q}}+\sup_{0< t\leq T}\|\nabla v(t)\|_{H^{\mu,r}}+\sup_{0< t\leq T}\|u(t)\|_{H^{\mu,p}},
\end{gather*}
and consider the closed subset $\tilde{S}$, which is defined by
\[
\tilde{S}:=\{(n,v,u)\in X_{\mu,T}:\|(n,v,u)\|_{X_{\mu,T}}\leq 2\mathcal{M}\}.
\]
The mapping $\mathcal{H}$ is the one given in \eqref{mildsoluH}.

In light of Lemma \ref{highregu}, it is obvious that
\begin{equation}\label{hinvnu}
\begin{split}
\|(-\Delta)^{\frac{\mu}{2}}(un)\|_{\frac{pq}{p+q}}&\leq C\big(\|(-\Delta)^{\frac{\mu}{2}}u\|_{p}\|n\|_q+\|(-\Delta)^{\frac{\mu}{2}}n\|_{q}\|u\|_p\big),\\
\|(-\Delta)^{\frac{\mu}{2}}(n\nabla v)\|_{\frac{rq}{r+q}}&\leq C\big(\|(-\Delta)^{\frac{\mu}{2}}\nabla v\|_{r}\|n\|_q+\|(-\Delta)^{\frac{\mu}{2}}n\|_{q}\|\nabla v\|_r\big).
\end{split}
\end{equation}
 In views of the definition of $\mathcal{H}_1(n,v,u)$ listed in \eqref{mildsoluH}, Proposition \ref{pro:MLopLpqesti}, the above inequality \eqref{hinvnu} and Lemma \ref{equiregu}, we evaluate that for any $0< t\leq T$:
\begin{equation}\label{hirelocalh1}
\begin{split}
\|\mathcal{H}_1(n,v,u)\|_{\dot{H}^{\mu,q}}\leq& \|(-\Delta)^{\frac{\mu}{2}}E_\beta(-t^\beta(-\Delta)^{\frac{\alpha}{2}})n_0\|_{q}\\
&~+\int_0^t(t-\tau)^{\beta-1}\|\nabla \cdot E_{\beta,\beta}(-(t-\tau)^\beta(-\Delta)^{\frac{\alpha}{2}})(-\Delta)^{\frac{\mu}{2}}(u n)\|_{q}\,d\tau\\
&~+\int_0^t(t-\tau)^{\beta-1}\|\nabla\cdot E_{\beta,\beta}(-(t-\tau)^\beta(-\Delta)^{\frac{\alpha}{2}})(-\Delta)^{\frac{\mu}{2}}(n\nabla v)\|_{q}\,d\tau\\
\leq&\|n_0\|_{\dot{H}^{\mu,q}} +CT^{-\frac{d\beta}{\alpha}(\frac{1}{p}-\frac{\alpha-1}{d})}\|u\|_{C((0,T];H^{\mu,p}(\mathbb{R}^d))}\|n\|_{C((0,T];H^{\mu,q}(\mathbb{R}^d))}\\
&+CT^{-\frac{d\beta}{\alpha}(\frac{1}{r}-\frac{\alpha-1}{d})}\|\nabla v\|_{C((0,T];H^{\mu,r}(\mathbb{R}^d))}\|n\|_{C((0,T];H^{\mu,q}(\mathbb{R}^d))},
\end{split}
\end{equation}
here the last inequality holds since the conditions $(1)$-$(4)$ in Theorem \ref{thm:localhighregu} ensure that
\[
-\frac{d\beta}{\alpha}(\frac{1}{p}-\frac{\alpha-1}{d})>0, ~~-\frac{d\beta}{\alpha}(\frac{1}{r}-\frac{\alpha-1}{d})>0.
\]
The inequality \eqref{hirelocalh1} together with \eqref{localmildH1} in Section \ref{3.1} yields that
\begin{equation}\label{hirelocalmildH1}
\begin{split}
\|\mathcal{H}_1(n,v,u)\|_{H^{\mu,q}}\leq&\|n_0\|_{H^{\mu,q}} +CT^{-\frac{d\beta}{\alpha}(\frac{1}{p}-\frac{\alpha-1}{d})}\|u\|_{C((0,T];H^{\mu,p}(\mathbb{R}^d))}\|n\|_{C((0,T];H^{\mu,q}(\mathbb{R}^d))}\\
&+CT^{-\frac{d\beta}{\alpha}(\frac{1}{r}-\frac{\alpha-1}{d})}\|\nabla v\|_{C((0,T];H^{\mu,r}(\mathbb{R}^d))}\|n\|_{C((0,T];H^{\mu,q}(\mathbb{R}^d))}.
\end{split}
\end{equation}

Similarly, by virtue of \eqref{mildsoluH}, Proposition \ref{pro:MLopLpqesti}, Lemma \ref{highregu} and Lemma \ref{equiregu} together with \eqref{localmildH2} in Section \ref{3.1}, we obtain that
\begin{equation}\label{hirelocalmildH2}
\begin{split}
\|\mathcal{H}_2(n,v,u)\|_{H^{\mu,r}}\leq&\|\nabla v_0\|_{H^{\mu,r}}+CT^{-\frac{d\beta}{\alpha}(\frac{1}{q}-\frac{1}{r}-\frac{\alpha-1}{d})}\|n\|_{C((0,T];H^{\mu,q}(\mathbb{R}^d))} \\
&+CT^{-\frac{d\beta}{\alpha}(\frac{1}{p}-\frac{\alpha-1}{d})}\|u\|_{C((0,T];H^{\mu,p}(\mathbb{R}^d))}\|\nabla v\|_{C((0,T];H^{\mu,r}(\mathbb{R}^d))}.
\end{split}
\end{equation}
In terms of the boundedness of the operator $P$, \eqref{MLLpq4} in Proposition \ref{pro:MLopLpqesti} and H\"{o}lder's inequality, we have
\[
\begin{split}
&\int_0^t(t-\tau)^{\beta-1}\|(-\Delta)^{\frac{\mu}{2}}E_{\beta,\beta}(-(t-\tau)^\beta(-\Delta)^{\frac{\alpha}{2}})[P(n\nabla \phi)]\|_{p}\,d\tau\\
\leq&C\int_0^t(t-\tau)^{-\frac{d\beta}{\alpha}(\frac{1}{q}+\frac{1+\mu}{d}-\frac{1}{p})+\beta-1}\|\nabla \phi\|_{d}\|n\|_{q}\,d\tau\\
\leq&CT^{-\frac{d\beta}{\alpha}(\frac{1}{q}+\frac{1+\mu}{d}-\frac{1}{p})+\beta}\|n\|_{C((0,T];L^q(\mathbb{R}^d))},
\end{split}
\]
provided $0\leq\frac{1}{q}-\frac{1}{p}+\frac{1+\mu}{d}<\frac{\alpha-1}{d}$. Therefore, it holds that
\begin{equation}\label{hirelocalmildH3}
\begin{split}
\|\mathcal{H}_3(n,v,u)\|_{\dot{H}^{\mu,p}}\leq& \|(-\Delta)^{\frac{\mu}{2}}E_\beta(-t^\beta(-\Delta)^{\frac{\alpha}{2}})u_0\|_{p}\\
&+\int_0^t(t-\tau)^{\beta-1}\|(-\Delta)^{\frac{\mu}{2}}\nabla\cdot E_{\beta,\beta}(-(t-\tau)^\beta(-\Delta)^{\frac{\alpha}{2}})[P(u\otimes u)]\|_{p}\,d\tau\\
&+\int_0^t(t-\tau)^{\beta-1}\|(-\Delta)^{\frac{\mu}{2}}E_{\beta,\beta}(-(t-\tau)^\beta(-\Delta)^{\frac{\alpha}{2}})[P(n\nabla \phi)]\|_{p}\,d\tau\\
\leq&\|u_0\|_{\dot{H}^{\mu,p}}+CT^{-\frac{d\beta}{\alpha}(\frac{1}{p}-\frac{\alpha-1}{d})}\|u\|^2_{C((0,T];H^{\mu,p}(\mathbb{R}^d))}\\
&+CT^{-\frac{d\beta}{\alpha}(\frac{1}{q}-\frac{1}{p})-\frac{(1+\mu)\beta}{\alpha}+\beta}\|n\|_{C((0,T];L^q(\mathbb{R}^d))}.
\end{split}
\end{equation}
Hence combining the above inequality \eqref{hirelocalmildH3} and \eqref{localmildH3} in Section \ref{3.1}, we find
\begin{equation}\label{highrelocalmildH3}
\begin{split}
\|\mathcal{H}_3(n,v,u)\|_{H^{\mu,p}}
\leq&\|u_0\|_{H^{\mu,p}}+CT^{-\frac{d\beta}{\alpha}(\frac{1}{p}-\frac{\alpha-1}{d})}\|u\|^2_{C((0,T];H^{\mu,p}(\mathbb{R}^d))}\\
&+C\big(T^{-\frac{d\beta}{\alpha}(\frac{1}{q}-\frac{1}{p})-\frac{\beta}{\alpha}
+\beta}+T^{-\frac{d\beta}{\alpha}(\frac{1}{q}-\frac{1}{p})-\frac{(1+\mu)\beta}{\alpha}+\beta}\big)\|n\|_{C((0,T];H^{\mu,q}(\mathbb{R}^d))}.
\end{split}
\end{equation}
In light of \eqref{hirelocalmildH1}, \eqref{hirelocalmildH2} and \eqref{highrelocalmildH3}, if we choose $\tilde{T}_0$ satisfy
\[
\begin{split}
\tilde{T}_0:=\min\{&(8\tilde{C}_1)^{\frac{\alpha q r}{\beta((r-q)d-(\alpha-1)qr)}},
(16\tilde{C}_1)^{\frac{\alpha q p}{\beta((p-q)d-(\alpha-1)qp)}},
(16\tilde{C}_1)^{\frac{\alpha pq}{\beta((p-q)d-(\alpha-1-\mu)pq)}},\\
&~~~~(16\tilde{\mathcal{C}}_2\mathcal{M})^{\frac{\alpha r}{\beta(d-(\alpha-1)r)}},
(16\tilde{\mathcal{C}}_2\mathcal{M})^{\frac{\alpha p}{\beta(d-(\alpha-1)p)}}\},
\end{split}
\]
it is not difficult to find that $\mathcal{H}(n,v,u): \tilde{S}\rightarrow \tilde{S}$, that is
\[
\|\mathcal{H}_1(n,v,u)\|_{H^{\mu,q}}+\|\nabla\mathcal{H}_2(n,v,u)\|_{H^{\mu,r}}+\|\mathcal{H}_3(n,v,u)\|_{H^{\mu,p}}\leq 2\mathcal{M},~~~~0<t\leq \tilde{T}_0.
\]

In the sequel, we shall calculate the distance estimates. For $(n_1,v_1,u_1), (n_2,v_2,u_2)\in X_{\mu,T}$, using similar estimates as  \eqref{hirelocalmildH1}, \eqref{hirelocalmildH2} and \eqref{highrelocalmildH3}, it is clear that
\begin{equation*}
\begin{split}
&\|\mathcal{H}_1(n_1,v_1,u_1)-\mathcal{H}_1(n_2,v_2,u_2)\|_{H^{\mu,q}}\\
\leq &CT^{-\frac{d\beta}{\alpha}(\frac{1}{p}-\frac{\alpha-1}{d})}\|u_1-u_2\|_{C((0,T];H^{\mu,p}\mathbb{R}^d))}\|n_1\|_{C((0,T];H^{\mu,q}(\mathbb{R}^d))}\\
&+CT^{-\frac{d\beta}{\alpha}(\frac{1}{p}-\frac{\alpha-1}{d})}\|n_1-n_2\|_{C((0,T];H^{\mu,q}\mathbb{R}^d))}\|u_2\|_{C((0,T];H^{\mu,p}(\mathbb{R}^d))}\\
&+CT^{-\frac{d\beta}{\alpha}(\frac{1}{r}-\frac{\alpha-1}{d})}\|n_1-n_2\|_{C((0,T];H^{\mu,q}\mathbb{R}^d))}\|\nabla v_1\|_{C((0,T];H^{\mu,r}(\mathbb{R}^d))}\\
&+CT^{-\frac{d\beta}{\alpha}(\frac{1}{r}-\frac{\alpha-1}{d})}\|\nabla v_1-\nabla v_2\|_{C((0,T];H^{\mu,r}\mathbb{R}^d))}\|n_2\|_{C((0,T];H^{\mu,q}(\mathbb{R}^d))}\\
\leq&2\mathcal{C}_4\mathcal{M}\big(T^{-\frac{d\beta}{\alpha}(\frac{1}{p}-\frac{\alpha-1}{d})}+T^{-\frac{d\beta}{\alpha}(\frac{1}{r}-\frac{\alpha-1}{d})}\big) D_T(n_1-n_2,v_1-v_2,u_1-u_2),
\end{split}
\end{equation*}
\begin{equation*}
\begin{split}
&\|\nabla\mathcal{H}_2(n_1,v_1,u_1)-\nabla\mathcal{H}_2(n_2,v_2,u_2)\|_{H^{\mu,r}}\\
\leq& CT^{-\frac{d\beta}{\alpha}(\frac{1}{p}-\frac{\alpha-1}{d})}\|u_1-u_2\|_{C((0,T];H^{\mu,p}\mathbb{R}^d))}(\|\nabla v_1\|_{C((0,T];H^{\mu,r}(\mathbb{R}^d))}\\
&+CT^{-\frac{d\beta}{\alpha}(\frac{1}{p}-\frac{\alpha-1}{d})}\|\nabla v_1-\nabla v_2\|_{C((0,T];H^{\mu,r}\mathbb{R}^d))}(\| u_2\|_{C((0,T];H^{\mu,p}(\mathbb{R}^d))}\\
&+CT^{-\frac{d\beta}{\alpha}(\frac{1}{q}-\frac{1}{r})-\frac{\beta}{\alpha}+\beta}\|n_1-n_2\|_{C((0,T];H^{\mu,q}(\mathbb{R}^d))}\\
\leq&\mathcal{C}_5\big(2\mathcal{M}T^{-\frac{d\beta}{\alpha}(\frac{1}{p}-\frac{\alpha-1}{d})}+T^{-\frac{d\beta}{\alpha}(\frac{1}{q}-\frac{1}{r}-\frac{\alpha-1}{d})}\big)  D_T(n_1-n_2,v_1-v_2,u_1-u_2),
\end{split}
\end{equation*}
and
\begin{equation*}
\begin{split}
&\|\mathcal{H}_3(n_1,v_1,u_1)-\mathcal{H}_3(n_2,v_2,u_2)\|_{H^{\mu,p}}\\
\leq & CT^{-\frac{d\beta}{\alpha}(\frac{1}{p}-\frac{\alpha-1}{d})}\|u_1-u_2\|_{C((0,T];H^{\mu,p}(\mathbb{R}^d))}(\|u_1\|_{C((0,T];H^{\mu,p}(\mathbb{R}^d))}+\|u_2\|_{C((0,T];H^{\mu,p}(\mathbb{R}^d))})\\
&+C\big(T^{-\frac{d\beta}{\alpha}(\frac{1}{q}-\frac{1}{p})-\frac{\beta}{\alpha}+\beta}
+T^{-\frac{d\beta}{\alpha}(\frac{1}{q}-\frac{1}{p})-\frac{(1+\mu)\beta}{\alpha}+\beta}\big)\|n_1-n_2\|_{C((0,T];H^{\mu,q}(\mathbb{R}^d))}\\
\leq&\mathcal{C}_6\big(2\mathcal{M}T^{-\frac{d\beta}{\alpha}(\frac{1}{p}-\frac{\alpha-1}{d})}+T^{-\frac{d\beta}{\alpha}(\frac{1}{q}-\frac{1}{p}-\frac{\alpha-1}{d})}
+T^{-\frac{d\beta}{\alpha}(\frac{1}{q}-\frac{1}{p})-\frac{(1+\mu)\beta}{\alpha}+\beta}\big) D_T(n_1-n_2,v_1-v_2,u_1-u_2).
\end{split}
\end{equation*}
Define
\[
\begin{split}
D_{\mu,T}(n_1-n_2,v_1-v_2,u_1-u_2):=&\sup_{0\leq t\leq T}\|(n_1-n_2)(t)\|_{H^{\mu,q}}\\
&+\sup_{0< t\leq T}\|\nabla (v_1-v_2)(t)\|_{H^{\mu,r}}+\sup_{0< t\leq T}\|(u_1-u_2)(t)\|_{H^{\mu,p}},
\end{split}
\]
therefore,
\[
\begin{split}
&D_{\mu,T}(\mathcal{H}(n_1,v_1,u_1)-\mathcal{H}(n_2,v_2,u_2))\leq \big(\tilde{\mathcal{C}}_3\mathcal{M}T^{-\frac{d\beta}{\alpha}(\frac{1}{p}-\frac{\alpha-1}{d})}
+2\mathcal{C}_4\mathcal{M}T^{-\frac{d\beta}{\alpha}(\frac{1}{r}-\frac{\alpha-1}{d})}\\
&~+\mathcal{C}_5T^{-\frac{d\beta}{\alpha}(\frac{1}{q}-\frac{1}{r}-\frac{\alpha-1}{d})}
+\mathcal{C}_6(T^{-\frac{d\beta}{\alpha}(\frac{1}{q}-\frac{1}{p}-\frac{\alpha-1}{d})}+
T^{-\frac{d\beta}{\alpha}(\frac{1}{q}-\frac{1}{p}-\frac{\alpha-1-\mu}{d})})\big) D_{\mu,T}(n_1-n_2,v_1-v_2,u_1-u_2).
\end{split}
\]
If we choose $\tilde{T}_1$ satisfy
\[
\begin{split}
\tilde{T}_1:=\min\{&\tilde{T}, \big(\frac{4\tilde{\mathcal{C}}_3\mathcal{M}}{\varrho}\big)^{\frac{\alpha p}{\beta(d-(\alpha-1)p)}},
\big(\frac{8\mathcal{C}_4\mathcal{M}}{\varrho}\big)^{\frac{\alpha r}{\beta(d-(\alpha-1)r)}},
\big(\frac{4\mathcal{C}_5}{\varrho}\big)^{\frac{\alpha q r}{\beta((r-q)d-(\alpha-1)qr)}},\\
&\big(\frac{8\mathcal{C}_6}{\varrho}\big)^{\frac{\alpha q p}{\beta((p-q)d-(\alpha-1)qp)}},
\big(\frac{8\mathcal{C}_6}{\varrho}\big)^{\frac{\alpha pq}{\beta((p-q)d-(\alpha-1-\mu)pq)}}\},
\end{split}
\]
then for $0<\varrho<1$, we have
\[
D_{\mu,T}(\mathcal{H}(n_1,v_1,u_1)-\mathcal{H}(n_2,v_2,u_2))\leq \varrho D_{\mu,T}(n_1-n_2,v_1-v_2,u_1-u_2).
\]
The claim that $\mathcal{H}: \tilde{S}\rightarrow \tilde{S}$ is a strict contraction map can be verified as the proof of Theorem \ref{thm:localmildexis}. The existence and uniqueness of mild solution to \eqref{chemo-fluid-eq} in $\tilde{S}$ then follow from Banach fixed point theorem.

Following the approach of the proof of Theorem \ref{thm:localmildexis}, the statement as regards $T_{max}$ can be obtained similarly, and we omit the details.
\end{proof}

\subsection{Proof of Theorem \ref{thm:globalhire}}\label{4.2}

\begin{proof}
The proof of Theorem \ref{thm:globalhire} is similar to that of Theorem \ref{thm:mildexis}. We consider Banach spaces $\mathcal{Y}_\mu$,
\[
\begin{split}
\mathcal{Y}_\mu:=\{[&n,v,u]|t^{\frac{d\beta}{\alpha}(\frac{2\alpha-2}{d}-\frac{1}{q})}n\in C((0,\infty);L^q(\mathbb{R}^d)),~t^{\frac{d\beta}{\alpha}(\frac{\alpha-1}{d}-\frac{1}{r})}\nabla v\in C((0,\infty);L^r(\mathbb{R}^d)),\\
&t^{\frac{d\beta}{\alpha}(\frac{\alpha-1}{d}-\frac{1}{p})}u\in C((0,\infty);L^p(\mathbb{R}^d)),~t^{\frac{d\beta}{\alpha}(\frac{2\alpha-2+\mu}{d}-\frac{1}{q})}n\in C((0,\infty);\dot{H}^{\mu,q}(\mathbb{R}^d)),\\
&t^{\frac{d\beta}{\alpha}(\frac{\alpha-1+\mu}{d}-\frac{1}{r})}\nabla v\in C((0,\infty);\dot{H}^{\mu,r}(\mathbb{R}^d)),
~t^{\frac{d\beta}{\alpha}(\frac{\alpha-1+\mu}{d}-\frac{1}{p})}u\in C((0,\infty);\dot{H}^{\mu,p}(\mathbb{R}^d))\}
\end{split}
\]
endowed with the norm
\begin{gather*}
\begin{split}
\|[n,v,u]\|_{\mathcal{Y}_\mu}:=&\sup_{0<t<\infty}t^{\frac{d\beta}{\alpha}(\frac{2\alpha-2}{d}-\frac{1}{q})}\|n(t)\|_{q}
+\sup_{0<t<\infty}t^{\frac{d\beta}{\alpha}(\frac{2\alpha-2}{d}-\frac{1}{q}+\frac{\mu}{d})}\|n(t)\|_{\dot{H}^{\mu,q}}\\
&+\sup_{0<t<\infty}t^{\frac{d\beta}{\alpha}(\frac{\alpha-1}{d}-\frac{1}{r})}\|\nabla v(t)\|_{r}
+\sup_{0<t<\infty}t^{\frac{d\beta}{\alpha}(\frac{\alpha-1}{d}-\frac{1}{r}+\frac{\mu}{d})}\|\nabla v(t)\|_{\dot{H}^{\mu,r}}\\
&+\sup_{0<t<\infty}t^{\frac{d\beta}{\alpha}(\frac{\alpha-1}{d}-\frac{1}{p})}\|u(t)\|_{p}
+\sup_{0<t<\infty}t^{\frac{d\beta}{\alpha}(\frac{\alpha-1}{d}-\frac{1}{p}+\frac{\mu}{d})}\|u(t)\|_{\dot{H}^{\mu,p}}.
\end{split}
\end{gather*}
The functional space $\mathcal{X}$ and the map $F$ are those introduced in Section \ref{3.2}.

Next we are going to show that $F:\mathcal{X}\times \mathcal{Y}_\mu\rightarrow \mathcal{Y}_\mu$ is a continuous map. By the definition of $\tilde{n}$ in \eqref{mapfmildsolu}, we have
\begin{equation}\label{hinLqesti}
\begin{split}
\|\tilde{n}(t)\|_{\dot{H}^{\mu,q}}&\leq\|n(t)\|_{\dot{H}^{\mu,q}}+\|E_\beta(-t^\beta(-\Delta)^{\frac{\alpha}{2}})n_0\|_{\dot{H}^{\mu,q}}\\
&~~~+\int_0^t(t-\tau)^{\beta-1}
\|\nabla\cdot E_{\beta,\beta}\big(-(t-\tau)^\beta (-\Delta)^{\frac{\alpha}{2}}\big)(-\Delta)^{\frac{\mu}{2}}(u n)\|_{q}\,d\tau\\
&~~~+\int_0^t(t-\tau)^{\beta-1}
\|\nabla \cdot E_{\beta,\beta}\big(-(t-\tau)^\beta (-\Delta)^{\frac{\alpha}{2}}\big)(-\Delta)^{\frac{\mu}{2}}(n\nabla v)\|_{q}\,d\tau.
\end{split}
\end{equation}
By virtue of \eqref{MLLpq1} in Proposition \ref{pro:MLopLpqesti}, it is obvious that if $q>\frac{d}{2\alpha-2+\mu}$, it holds
\[
\|E_\beta(-t^\beta(-\Delta)^{\frac{\alpha}{2}})n_0\|_{\dot{H}^{\mu,q}}\leq Ct^{-\frac{d\beta}{\alpha}(\frac{2\alpha-2+\mu}{d}-\frac{1}{q})}\|n_0\|_{\frac{d}{2\alpha-2}}.
\]
Proposition \ref{pro:MLopLpqesti} and Lemma \ref{highregu} yield that
\[
\begin{split}
&~~~\int_0^t(t-\tau)^{\beta-1}
\|\nabla\cdot E_{\beta,\beta}\big(-(t-\tau)^\beta (-\Delta)^{\frac{\alpha}{2}}\big)(-\Delta)^{\frac{\mu}{2}}(u n)\|_{q}\,d\tau\\
&\leq C\int_0^t(t-\tau)^{-\frac{d\beta}{\alpha p}-\frac{\beta}{\alpha}+\beta-1}\big(\|(-\Delta)^{\frac{\mu}{2}}u\|_{p}\|n\|_{q}+\|(-\Delta)^{\frac{\mu}{2}} n\|_{q}\|u\|_{p}\big)\,d\tau \\
&\leq C t^{-\frac{d\beta}{\alpha}(\frac{2\alpha-2+\mu}{d}-\frac{1}{q})}\mathbf{B}(-\frac{d\beta}{\alpha p}-\frac{\beta}{\alpha}+\beta, \frac{d\beta}{\alpha }(\frac{1}{q}+\frac{1}{ p})-\frac{(3\alpha-3+\mu)\beta}{\alpha}+1)\\
&~~~\times\big(\sup_{\tau>0}\tau^{\frac{d\beta}{\alpha}(\frac{\alpha-1+\mu}{d}-\frac{1}{p})}\|u\|_{\dot{H}^{\mu,p}}
\sup_{\tau>0}\tau^{\frac{d\beta}{\alpha}(\frac{2\alpha-2}{d}-\frac{1}{q})}\|n\|_{q}\\
&~~~~~~~~+\sup_{\tau>0}\tau^{\frac{d\beta}{\alpha}(\frac{2\alpha-2+\mu}{d}-\frac{1}{q})}\|n\|_{\dot{H}^{\mu,q}}
\sup_{\tau>0}\tau^{\frac{d\beta}{\alpha}(\frac{\alpha-1}{d}-\frac{1}{p})}\|u\|_{p}\big)\\
&\leq Ct^{-\frac{d\beta}{\alpha}(\frac{2\alpha-2+\mu}{d}-\frac{1}{q})}\big(\sup_{\tau>0}\tau^{\frac{d\beta}{\alpha}(\frac{\alpha-1+\mu}{d}-\frac{1}{p})}\|u\|_{\dot{H}^{\mu,p}}
\sup_{\tau>0}\tau^{\frac{d\beta}{\alpha}(\frac{2\alpha-2}{d}-\frac{1}{q})}\|n\|_{q}\\
&~~~~~~~~~~~~~~~~~~~~~~~~~~+\sup_{\tau>0}\tau^{\frac{d\beta}{\alpha}(\frac{2\alpha-2+\mu}{d}-\frac{1}{q})}\|n\|_{\dot{H}^{\mu,q}}
\sup_{\tau>0}\tau^{\frac{d\beta}{\alpha}(\frac{\alpha-1}{d}-\frac{1}{p})}\|u\|_{p}\big),
\end{split}
\]
where the exponents $\alpha,\beta,\mu,p,q,r$ satisfy the conditions in Assumption \ref{assum2}, which imply that
\[
-\frac{d\beta}{\alpha p}-\frac{\beta}{\alpha}+\beta>0, ~~\frac{d\beta}{\alpha q}+\frac{d\beta}{\alpha p}-\frac{(3\alpha-3+\mu)\beta}{\alpha}+1>0.
\]
Similarly, Proposition \ref{pro:MLopLpqesti} and Lemma \ref{highregu} lead to
\[
\begin{split}
&~~~\int_0^t(t-\tau)^{\beta-1}
\|\nabla \cdot E_{\beta,\beta}\big(-(t-\tau)^\beta (-\Delta)^{\frac{\alpha}{2}}\big)(-\Delta)^{\frac{\mu}{2}}(n\nabla v)\|_{q}\,d\tau\\
&\leq C t^{-\frac{d\beta}{\alpha}(\frac{2\alpha-2+\mu}{d}-\frac{1}{q})}\mathbf{B}(-\frac{d\beta}{\alpha r}-\frac{\beta}{\alpha}+\beta, \frac{d\beta}{\alpha}(\frac{1}{q}+\frac{1}{r})-\frac{(3\alpha-3+\mu)\beta}{\alpha}+1)\\
&~~~\times\big(\sup_{\tau>0}\tau^{\frac{d\beta}{\alpha}(\frac{\alpha-1+\mu}{d}-\frac{1}{r})}\|\nabla v\|_{\dot{H}^{\mu,r}}
\sup_{\tau>0}\tau^{\frac{d\beta}{\alpha}(\frac{2\alpha-2}{d}-\frac{1}{q})}\|n\|_{q}\\
&~~~~~~~~+\sup_{\tau>0}\tau^{\frac{d\beta}{\alpha}(\frac{2\alpha-2+\mu}{d}-\frac{1}{q})}\|n\|_{\dot{H}^{\mu,q}}
\sup_{\tau>0}\tau^{\frac{d\beta}{\alpha}(\frac{\alpha-1}{d}-\frac{1}{r})}\|\nabla v\|_{r}\big),
\end{split}
\]
where Assumption \ref{assum2} ensure that
\[
-\frac{d\beta}{\alpha r}-\frac{\beta}{\alpha}+\beta>0,~ \frac{d\beta}{\alpha q}+\frac{d\beta}{\alpha r}-\frac{(3\alpha-3+\mu)\beta}{\alpha}+1>0.
\]
Hence we obtain that
\begin{equation}\label{hinestbdd}
\begin{split}
\sup_{t>0}t^{\frac{d\beta}{\alpha}(\frac{2\alpha-2+\mu}{d}-\frac{1}{q})}\|\tilde{n}(t)\|_{\dot{H}^{\mu,q}}&\leq C\| n_0\|_{\frac{d}{2\alpha-2}}
+\sup_{t>0}t^{\frac{d\beta}{\alpha}(\frac{2\alpha-2+\mu}{d}-\frac{1}{q})}\|n(t)\|_{\dot{H}^{\mu,q}}\\
&~~~+
C\sup_{t>0}t^{\frac{d\beta}{\alpha}(\frac{\alpha-1+\mu}{d}-\frac{1}{p})}\|u\|_{\dot{H}^{\mu,p}}
\sup_{t>0}t^{\frac{d\beta}{\alpha}(\frac{2\alpha-2}{d}-\frac{1}{q})}\|n\|_{q}\\
&~~~+C\sup_{t>0}t^{\frac{d\beta}{\alpha}(\frac{2\alpha-2+\mu}{d}-\frac{1}{q})}\|n\|_{\dot{H}^{\mu,q}}
\sup_{t>0}t^{\frac{d\beta}{\alpha}(\frac{\alpha-1}{d}-\frac{1}{p})}\|u\|_{p}\\
&~~~+C\sup_{t>0}t^{\frac{d\beta}{\alpha}(\frac{\alpha-1+\mu}{d}-\frac{1}{r})}\|\nabla v\|_{\dot{H}^{\mu,r}}
\sup_{t>0}t^{\frac{d\beta}{\alpha}(\frac{2\alpha-2}{d}-\frac{1}{q})}\|n\|_{q}\\
&~~~+C\sup_{t>0}t^{\frac{d\beta}{\alpha}(\frac{2\alpha-2+\mu}{d}-\frac{1}{q})}\|n\|_{\dot{H}^{\mu,q}}
\sup_{t>0}t^{\frac{d\beta}{\alpha}(\frac{\alpha-1}{d}-\frac{1}{r})}\|\nabla v\|_{r}.
\end{split}
\end{equation}
Based on \eqref{mapfmildsolu}, we utilize Proposition \ref{pro:MLopLpqesti} and Lemma \ref{highregu} again to calculate that
\begin{equation}\label{hivestbdd}
\begin{split}
&\sup_{t>0}t^{\frac{d\beta}{\alpha}(\frac{\alpha-1+\mu}{d}-\frac{1}{r})}\|\nabla\tilde{ v}(t)\|_{\dot{H}^{\mu,r}}
\leq C\| \nabla v_0\|_{\frac{d}{\alpha-1}}
+\sup_{t>0}t^{\frac{d\beta}{\alpha}(\frac{\alpha-1+\mu}{d}-\frac{1}{r})}\|\nabla v\|_{\dot{H}^{\mu,r}}\\
&~~~~+C\mathbf{B}(\beta-\frac{d\beta}{\alpha p}-\frac{\beta}{\alpha}, \frac{d\beta}{\alpha }(\frac{1}{p}+\frac{1}{r})-\frac{(2\alpha-2+\mu)\beta}{\alpha}+1)
\big(\sup_{t>0}t^{\frac{d\beta}{\alpha}(\frac{\alpha-1+\mu}{d}-\frac{1}{p})}\|u\|_{\dot{H}^{\mu,p}}\\
&~~~~~~~\times\sup_{t>0}t^{\frac{d\beta}{\alpha}(\frac{\alpha-1}{d}-\frac{1}{r})}\|\nabla v\|_{r}+\sup_{t>0}t^{\frac{d\beta}{\alpha}(\frac{\alpha-1+\mu}{d}-\frac{1}{r})}\|\nabla v\|_{\dot{H}^{\mu,r}}
\sup_{t>0}t^{\frac{d\beta}{\alpha}(\frac{\alpha-1}{d}-\frac{1}{p})}\|u\|_{p}\big)\\
&~~~~+C\mathbf{B}(\beta-\frac{d\beta}{\alpha }(\frac{1}{q}-\frac{1}{r})-\frac{\beta}{\alpha}, \frac{d\beta}{\alpha q}-\frac{(2\alpha-2+\mu)\beta}{\alpha}+1)\sup_{t>0}t^{\frac{d\beta}{\alpha}(\frac{2\alpha-2+\mu}{d}-\frac{1}{q})}\|n\|_{\dot{H}^{\mu,q}}\\
&~\leq C\| \nabla v_0\|_{\frac{d}{\alpha-1}}
+\sup_{t>0}t^{\frac{d\beta}{\alpha}(\frac{\alpha-1+\mu}{d}-\frac{1}{r})}\|\nabla v\|_{\dot{H}^{\mu,r}}+C\sup_{t>0}t^{\frac{d\beta}{\alpha}(\frac{2\alpha-2+\mu}{d}-\frac{1}{q})}\|n\|_{\dot{H}^{\mu,q}}\\
&~~~~+C\sup_{t>0}t^{\frac{d\beta}{\alpha}(\frac{\alpha-1+\mu}{d}-\frac{1}{p})}\|u\|_{\dot{H}^{\mu,p}}\sup_{t>0}t^{\frac{d\beta}{\alpha}(\frac{\alpha-1}{d}-\frac{1}{r})}\|\nabla v\|_{r}\\
&~~~~+C\sup_{t>0}t^{\frac{d\beta}{\alpha}(\frac{\alpha-1+\mu}{d}-\frac{1}{r})}\|\nabla v\|_{\dot{H}^{\mu,r}}
\sup_{t>0}t^{\frac{d\beta}{\alpha}(\frac{\alpha-1}{d}-\frac{1}{p})}\|u\|_{p},
\end{split}
\end{equation}
and
\begin{equation}\label{hiuestbdd}
\begin{split}
&\sup_{t>0}t^{\frac{d\beta}{\alpha}(\frac{\alpha-1+\mu}{d}-\frac{1}{p})}\|\tilde{u}(t)\|_{\dot{H}^{\mu,p}}\leq C\| u_0\|_{\frac{d}{\alpha-1}}
+\sup_{t>0}t^{\frac{d\beta}{\alpha}(\frac{\alpha-1+\mu}{d}-\frac{1}{p})}\|u(t)\|_{\dot{H}^{\mu,p}}\\
&~~+
C\mathbf{B}(\beta-\frac{d\beta}{\alpha p}-\frac{\beta}{\alpha}, \frac{2d\beta}{\alpha p}-\frac{(2\alpha-2+\mu)\beta}{\alpha}+1)\sup_{t>0}t^{\frac{d\beta}{\alpha}(\frac{\alpha-1+\mu}{d}-\frac{1}{p})}\|u\|_{\dot{H}^{\mu,p}}
\sup_{t>0}t^{\frac{d\beta}{\alpha}(\frac{\alpha-1}{d}-\frac{1}{p})}\|u\|_{p}\\
&~~+C\mathbf{B}(\beta-\frac{d\beta}{\alpha}(\frac{1}{q}-\frac{1}{p})-\frac{(1+\mu)\beta}{\alpha}, \frac{d\beta}{\alpha q}-\frac{(2\alpha-2)\beta}{\alpha}+1)\sup_{t>0}t^{\frac{d\beta}{\alpha}(\frac{2\alpha-2}{d}-\frac{1}{q})}\|n\|_{q}\|\nabla\phi\|_{d}\\
&\leq C\| u_0\|_{\frac{d}{\alpha-1}}+\sup_{t>0}t^{\frac{d\beta}{\alpha}(\frac{\alpha-1+\mu}{d}-\frac{1}{p})}\|u(t)\|_{\dot{H}^{\mu,p}}
+C\sup_{t>0}t^{\frac{d\beta}{\alpha}(\frac{\alpha-1+\mu}{d}-\frac{1}{p})}\|u\|_{\dot{H}^{\mu,p}}
\sup_{t>0}t^{\frac{d\beta}{\alpha}(\frac{\alpha-1}{d}-\frac{1}{p})}\|u\|_{p}\\
&~~+C\sup_{t>0}t^{\frac{d\beta}{\alpha}(\frac{2\alpha-2}{d}-\frac{1}{q})}\|n\|_{q}\|\nabla\phi\|_{d},
\end{split}
\end{equation}
where $p,q,r,\alpha,\beta,\mu$ in Assumption \ref{assum2} satisfy
\begin{align*}
&\beta-\frac{d\beta}{\alpha p}-\frac{\beta}{\alpha}>0,~~~\frac{d\beta}{\alpha q}-\frac{(2\alpha-2)\beta}{\alpha}+1>0,~~~\beta-\frac{d\beta}{\alpha }(\frac{1}{q}-\frac{1}{r})-\frac{\beta}{\alpha}>0,\\
 &\frac{d\beta}{\alpha }(\frac{1}{p}+\frac{1}{r})-\frac{(2\alpha-2+\mu)\beta}{\alpha}+1>0, ~~~\frac{d\beta}{\alpha q}-\frac{(2\alpha-2+\mu)\beta}{\alpha}+1>0, \\
&\frac{2d\beta}{\alpha p}-\frac{(2\alpha-2+\mu)\beta}{\alpha}+1>0,~~~ \beta-\frac{d\beta}{\alpha}(\frac{1}{q}-\frac{1}{p})-\frac{(1+\mu)\beta}{\alpha}>0.
\end{align*}

The time continuity follows in a similar approach as the argument in the proof of Theorem \ref{thm:localmildexis}, then we have
\[
\begin{split}
&t^{\frac{d\beta}{\alpha}(\frac{2\alpha-2+\mu}{d}-\frac{1}{q})}\tilde{n}\in C((0,\infty);\dot{H}^{\mu,q}(\mathbb{R}^d)),\\
&t^{\frac{d\beta}{\alpha}(\frac{\alpha-1+\mu}{d}-\frac{1}{r})}\nabla\tilde{v}\in C((0,\infty);\dot{H}^{\mu,r}(\mathbb{R}^d)),\\
&t^{\frac{d\beta}{\alpha}(\frac{\alpha-1+\mu}{d}-\frac{1}{p})}\tilde{u}\in C((0,\infty);\dot{H}^{\mu,p}(\mathbb{R}^d)).
\end{split}
\]
From \eqref{nestbdd},\eqref{vtildeestbdd}, \eqref{uestbdd} and \eqref{hinestbdd}-\eqref{hiuestbdd}, we conclude that $F(n_0,v_0,u_0,\phi,n,v,u)\in \mathcal{Y}_\mu$ with
\[
\|F(n_0,v_0,u_0,\phi,n,v,u)\|_{\mathcal{Y}_\mu}\leq C\|[n_0,v_0,u_0,\phi]\|_{\mathcal{X}}+C\|[n,v,u]\|_{\mathcal{Y}_\mu}(1+\|[n,v,u]\|_{\mathcal{Y}_\mu}+\|\nabla\phi\|_{d}).
\]
Following what we have done before in the proof of Step 3 of Theorem \ref{thm:mildexis}, recall the definition of $N(t)$, $\nabla V(t)$ and $U(t)$ listed in \eqref{Nt}, \eqref{Vt} and \eqref{Ut}, respectively, we arrive at
\begin{gather}\label{hilNmildsolu}
\begin{split}
\|N(t)\|_{\dot{H}^{\mu,q}}\leq&Ct^{-\frac{d\beta}{\alpha}(\frac{2\alpha-2+\mu}{d}-\frac{1}{q})}\mathbf{B}(-\frac{d\beta}{\alpha p}-\frac{\beta}{\alpha}+\beta, \frac{d\beta}{\alpha }(\frac{1}{q}+\frac{1}{ p})-\frac{(3\alpha-3+\mu)\beta}{\alpha}+1)\\
&~~\times\big(\sup_{t>0}t^{\frac{d\beta}{\alpha}(\frac{\alpha-1+\mu}{d}-\frac{1}{p})}\|\bar{u}\|_{\dot{H}^{\mu,p}}
\sup_{t>0}t^{\frac{d\beta}{\alpha}(\frac{2\alpha-2}{d}-\frac{1}{q})}\|\bar{n}\|_{q}\\
&~~~~~+\sup_{t>0}t^{\frac{d\beta}{\alpha}(\frac{2\alpha-2+\mu}{d}-\frac{1}{q})}\|\bar{n}\|_{\dot{H}^{\mu,q}}
\sup_{t>0}t^{\frac{d\beta}{\alpha}(\frac{\alpha-1}{d}-\frac{1}{p})}\|\bar{u}\|_{p}\big)\\
&+Ct^{-\frac{d\beta}{\alpha}(\frac{2\alpha-2+\mu}{d}-\frac{1}{q})}\mathbf{B}(-\frac{d\beta}{\alpha r}-\frac{\beta}{\alpha}+\beta, \frac{d\beta}{\alpha}(\frac{1}{q}+\frac{1}{r})-\frac{(3\alpha-3+\mu)\beta}{\alpha}+1)\\
&~~\times\big(\sup_{t>0}t^{\frac{d\beta}{\alpha}(\frac{\alpha-1+\mu}{d}-\frac{1}{r})}\|\nabla \bar{v}\|_{\dot{H}^{\mu,r}}
\sup_{t>0}t^{\frac{d\beta}{\alpha}(\frac{2\alpha-2}{d}-\frac{1}{q})}\|\bar{n}\|_{q}\\
&~~~~~+\sup_{t>0}t^{\frac{d\beta}{\alpha}(\frac{2\alpha-2+\mu}{d}-\frac{1}{q})}\|\bar{n}\|_{\dot{H}^{\mu,q}}
\sup_{t>0}t^{\frac{d\beta}{\alpha}(\frac{\alpha-1}{d}-\frac{1}{r})}\|\nabla \bar{v}\|_{r}\big),
\end{split}
\end{gather}
\begin{gather}\label{hilVmildsolu}
\begin{split}
\|\nabla V(t)\|_{\dot{H}^{\mu,r}}\leq&Ct^{-\frac{d\beta}{\alpha}(\frac{\alpha-1+\mu}{d}-\frac{1}{r})}\mathbf{B}(\beta-\frac{d\beta}{\alpha p}-\frac{\beta}{\alpha}, \frac{d\beta}{\alpha }(\frac{1}{p}+\frac{1}{r})-\frac{(2\alpha-2+\mu)\beta}{\alpha}+1)\\
&\times\big(\sup_{t>0}t^{\frac{d\beta}{\alpha}(\frac{\alpha-1+\mu}{d}-\frac{1}{p})}\|\bar{u}\|_{\dot{H}^{\mu,p}}
\sup_{t>0}t^{\frac{d\beta}{\alpha}(\frac{ \alpha-1}{d}-\frac{1}{r})}\|\nabla\bar{v}\|_{r}\\
&~~~~+C\sup_{t>0}t^{\frac{d\beta}{\alpha}(\frac{ \alpha-1+\mu}{d}-\frac{1}{r})}\|\nabla\bar{v}\|_{\dot{H}^{\mu,r}}
\sup_{t>0}t^{\frac{d\beta}{\alpha}(\frac{\alpha-1}{d}-\frac{1}{p})}\|\bar{u}\|_{p}\big),
\end{split}
\end{gather}
and
\begin{gather}\label{hilUmildsolu}
\begin{split}
\| U(t)\|_{\dot{H}^{\mu,p}}\leq&Ct^{-\frac{d\beta}{\alpha}(\frac{\alpha-1+\mu}{d}-\frac{1}{p})}\mathbf{B}(\beta-\frac{d\beta}{\alpha p}-\frac{\beta}{\alpha}, \frac{2d\beta}{\alpha p}-\frac{(2\alpha-2+\mu)\beta}{\alpha}+1)\\
&\times\big(\sup_{t>0}t^{\frac{d\beta}{\alpha}(\frac{\alpha-1+\mu}{d}-\frac{1}{p})}\|\bar{u}\|_{\dot{H}^{\mu,p}}
\sup_{t>0}t^{\frac{d\beta}{\alpha}(\frac{ \alpha-1}{d}-\frac{1}{p})}\| \bar{u}\|_{p}\big).
\end{split}
\end{gather}
Therefore, \eqref{lNmildsolu}, \eqref{lVmildsolu}, \eqref{lUmildsolu} and \eqref{hilNmildsolu}-\eqref{hilUmildsolu} indicate that for each $[n_0,v_0,u_0,\phi]\in \mathcal{X}$ and each $[n,v,u]\in \mathcal{Y}_\mu$, it holds
\begin{gather*}
\lim_{\|[\bar{n},\bar{v},\bar{u}]\|_{\mathcal{Y}_\mu}\rightarrow 0}\frac{\|[N,V,U]\|_{\mathcal{Y}_\mu}}{\|[\bar{n},\bar{v},\bar{u}]\|_{\mathcal{Y}_\mu}}=0,
\end{gather*}
which implies that the Fr\'{e}chet derivative of $F$ at point $[n_0,v_0,u_0,\phi,n+\bar{n},v+\bar{v},u+\bar{u}]\in \mathcal{X}\times \mathcal{Y}_\mu$ in the direction to $[n,v,u]$ is equal to $\mathcal{L}_{[n,v,u]}(\bar{n},\bar{v},\bar{u})$.

Applying Banach implicit function theorem, we observe that, for some $\bar{M}>0$, there is a $C^1$-map $\bar{g}:\mathcal{X}_{\bar{M}}\rightarrow \mathcal{Y}_{\mu,\bar{M}}$ such that
\[
\begin{split}
\bar{g}(0,0,0,0)&=[0,0,0],\\
F(n_0,v_0,u_0,\phi,\bar{g}(n_0,v_0,u_0,\phi))&=[0,0,0],~~\forall [n_0,v_0,u_0,\phi]\in \mathcal{X}_{\bar{M}},
\end{split}
\]
where
\[
\begin{split}
\mathcal{X}_{\bar{M}}:&=\{[n_0,v_0,u_0,\phi]\in \mathcal{X}:\|[n_0,v_0,u_0,\phi]\|_{\mathcal{X}}<\bar{M}\},\\
\mathcal{Y}_{\mu,\bar{M}}:&=\{[n,v,u]\in \mathcal{Y}_\mu:\|[n,v,u]\|_{\mathcal{Y}_\mu}<\bar{M}\}.
\end{split}
\]
Thus, the existence and uniqueness of mild solution to system \eqref{chemo-fluid-eq} are the consequence of the existence of the continuously differential map $\bar{g}$ from $\mathcal{X}_{\bar{M}}$ to $\mathcal{Y}_{\mu,\bar{M}}$.
\end{proof}

\section{Properties of mild solutions}\label{sec5}
In this section, we explore the integrability of the local mild solution if $n_0\in L^1(\mathbb{R}^d)\cap L^{q}(\mathbb{R}^d),~v_0\in L^{1}(\mathbb{R}^d),~\nabla v_0\in L^{r}(\mathbb{R}^d)$ and $u_0\in L^{1}(\mathbb{R}^d)\cap L^{p}(\mathbb{R}^d)$, and then obtain the mass conservation. In addition, we prove the decay estimates and stability of global mild solution as a byproduct of Theorem \ref{thm:mildexis}. At the end of this paper, we investigate the self-similar solution whenever taking $\gamma=0$ in system \eqref{chemo-fluid-eq}.

\begin{lem}\label{PQinte}\cite{LLW}
Assume $P(x,t)$, $Q(x,t)$ to be those defined in \eqref{PQ}, then $P(x,t)$, $Q(x,t)$ are nonnegative and integrable. In particular, it holds
\[
\int_{\mathbb{R}^d}P(x,t)\,dx=1,~~~~ \int_{\mathbb{R}^d}Q(x,t)\,dx=\frac{1}{\Gamma(\beta)}.
\]
\end{lem}

\subsection{Proof of Theorem \ref{thm:localL1mass}}\label{5.1}

\begin{proof}
In this proof, we consider the Banach space
\[
\begin{split}
X^1_{T}:=\{(n,v,u): &~n\in C((0,T];L^1(\mathbb{R}^d)\cap L^q(\mathbb{R}^d) ), v \in C((0,T];L^1(\mathbb{R}^d)),\\
&\nabla v \in C((0,T];L^r(\mathbb{R}^d)),u\in C((0,T];L^p(\mathbb{R}^d))\}
\end{split}\]
with the norm given by
\begin{align*}
\|(n,v,u)\|_{X^1_{T}}:=&\sup_{0< t\leq T}\|n(t)\|_{1}+\sup_{0< t\leq T}\|n(t)\|_{q}+\sup_{0< t\leq T}\|v(t)\|_{1}+\sup_{0< t\leq T}\|\nabla v(t)\|_{r}+\sup_{0< t\leq T}\|u(t)\|_{p},
\end{align*}
and the closed subset $S_1:=\{(n,v,u)\in X^1_T:\|(n,v,u)\|_{X^1_T}\leq 2M_1\}.$ The mapping $\mathcal{H}$ is the one defined in \eqref{mildsoluH}.

Next for $0< t\leq T$, we will estimate $\|\mathcal{H}_1(n,v,u)\|_{1}$, $\|\mathcal{H}_2(n,v,u)\|_{1}$ and $\|\mathcal{H}_3(n,v,u)\|_{1}$, respectively. Proposition \ref{pro:MLopLpqesti}, H\"{o}lder's inequality and interpolation inequality imply that
\begin{equation}\label{L1H1}
\begin{split}
\|\mathcal{H}_1(n,v,u)\|_{1}\leq& \|E_\beta(-t^\beta(-\Delta)^{\frac{\alpha}{2}})n_0\|_{1}+\int_0^t(t-\tau)^{\beta-1}\|\nabla \cdot E_{\beta,\beta}(-(t-\tau)^\beta(-\Delta)^{\frac{\alpha}{2}})(u n)\|_{1}\,d\tau\\
&~+\int_0^t(t-\tau)^{\beta-1}\|\nabla\cdot E_{\beta,\beta}(-(t-\tau)^\beta(-\Delta)^{\frac{\alpha}{2}})(n\nabla v)\|_{1}\,d\tau\\
\leq&\|n_0\|_{1}+ C\int_0^t(t-\tau)^{-\frac{\beta}{\alpha}+\beta-1}\big(\|u\|_{p}\|n\|_{\frac{p}{p-1}}+\|\nabla v\|_{r}\|n\|_{\frac{r}{r-1}}\big)\,d\tau\\
\leq&\|n_0\|_{1}+ C\int_0^t(t-\tau)^{-\frac{\beta}{\alpha}+\beta-1}\big(\|u\|_{p}\|n\|^{1-\theta_1}_{1}\|n\|^{\theta_1}_{q}+\|\nabla v\|_{r}\|n\|^{1-\theta_2}_{1}\|n\|^{\theta_2}_{q}\big)\,d\tau\\
\leq& \|n_0\|_{1}+ CT^{-\frac{\beta}{\alpha}+\beta}\big(\|u\|_{C((0,T];L^p(\mathbb{R}^d))}\|n\|^{1-\theta_1}_{C((0,T];L^1(\mathbb{R}^d))}\|n\|^{\theta_1}_{C((0,T];L^q(\mathbb{R}^d))}\\
&~+ \|\nabla v\|_{C((0,T];L^r(\mathbb{R}^d))}\|n\|^{1-\theta_2}_{C((0,T];L^1(\mathbb{R}^d))}\|n\|^{\theta_2}_{C((0,T];L^q(\mathbb{R}^d))}\big),
\end{split}
\end{equation}
where $0<\theta_1,\theta_2<1$ satisfy $1-\frac{1}{p}=1-\theta_1+\frac{\theta_1}{q}$ and $1-\frac{1}{r}=1-\theta_2+\frac{\theta_2}{q}$.

With a similar fashion as \eqref{L1H1}, we have the following estimates
\begin{equation}\label{L1H2}
\begin{split}
\|\mathcal{H}_2(n,v,u)\|_{1}\leq&\|v_0\|_{1}+CT^{\beta}\|n\|_{C((0,T];L^1(\mathbb{R}^d))}\\
&~+CT^{-\frac{\beta}{\alpha}+\beta}\big(\|u\|_{C((0,T];L^p(\mathbb{R}^d))}
\|v\|^{1-\theta_3}_{C((0,T];L^1(\mathbb{R}^d))}\|\nabla v\|^{\theta_3}_{C((0,T];L^r(\mathbb{R}^d))}\big),
\end{split}
\end{equation}
where $0<\theta_3=\frac{rd}{p(rd-d+r)}<1$.

Combining \eqref{localmildH1}, \eqref{localmildH2}, \eqref{localmildH3} and \eqref{L1H1}, \eqref{L1H2}, the claim that $\mathcal{H}:S_1\rightarrow S_1$ is a strict contract mapping can be proved as the same argunent with the proof of Theorem \ref{thm:localmildexis}. Then the existence
\[
\begin{split}
&n\in C((0,T];L^1(\mathbb{R}^d)\cap L^q(\mathbb{R}^d) ), v \in C((0,T];L^1(\mathbb{R}^d)),\\
&\nabla v \in C((0,T];L^r(\mathbb{R}^d)),u\in C((0,T];L^p(\mathbb{R}^d)),
\end{split}
\]
can be performed similarly as we did in the proof of Theorem \ref{thm:localmildexis} in Section \ref{3.1}.

As soon as we have the integrability, that is
\[
\begin{split}
\int_{\mathbb{R}^d}n(x,t)\,dx=&\int_{\mathbb{R}^d}E_\beta(-t^\beta(-\Delta)^{\frac{\alpha}{2}})n_0(x)\,dx\\
&-\int_{\mathbb{R}^d}\int_0^t(t-\tau)^{\beta-1}E_{\beta,\beta}(-(t-\tau)^\beta(-\Delta)^{\frac{\alpha}{2}})\big(\nabla\cdot(un+n\nabla v)\big)\,d\tau\,dx,
\end{split}
\]
by Lemma \ref{PQinte}, we then get
\[
\int_{\mathbb{R}^d}E_\beta(-t^\beta(-\Delta)^{\frac{\alpha}{2}})n_0(x)\,dx=\int_{\mathbb{R}^d}n_0(x)\,dx.
\]
For any $t>0$, we have $un+n\nabla v\in C((0,t];L^1(\mathbb{R}^d))$ since $(n,v,u)\in X^1_T$. By approximating $un+n\nabla v$ with $C_c^\infty((0,t]\times\mathbb{R}^d)$, we find
\[
\int_{\mathbb{R}^d}\int_0^t(t-\tau)^{\beta-1}E_{\beta,\beta}(-(t-\tau)^\beta(-\Delta)^{\frac{\alpha}{2}})\big(\nabla\cdot(un+n\nabla v)\big)\,d\tau\,dx=0.
\]
Thus, we obtain the following mass conservation
\[
\int_{\mathbb{R}^d}n(x,t)\,dx=\int_{\mathbb{R}^d}n_0(x)\,dx.
\]

In addition, on the one hand, whenever $\gamma=0$, we derive that
\begin{align}\label{vL10}
\begin{split}
\int_{\mathbb{R}^d}v(x,t)\,dx=&\int_{\mathbb{R}^d}E_\beta(-t^\beta(-\Delta)^{\frac{\alpha}{2}})v_0(x)\,dx\\
&-\int_{\mathbb{R}^d}\int_0^t(t-\tau)^{\beta-1}E_{\beta,\beta}(-(t-\tau)^\beta(-\Delta)^{\frac{\alpha}{2}})\big(\nabla\cdot(uv)\big)\,d\tau\,dx\\
&+\int_{\mathbb{R}^d}\int_0^t(t-\tau)^{\beta-1}E_{\beta,\beta}(-(t-\tau)^\beta(-\Delta)^{\frac{\alpha}{2}})n\,d\tau\,dx.
\end{split}
\end{align}
In light of Lemma \ref{PQinte}, we get
\[
\begin{split}
&\int_{\mathbb{R}^d}\int_0^t(t-\tau)^{\beta-1}E_{\beta,\beta}(-(t-\tau)^\beta(-\Delta)^{\frac{\alpha}{2}})n(x,\tau)\,d\tau\,dx\\
=&\frac{1}{\Gamma(\beta)}\int_0^t(t-\tau)^{\beta-1}\int_{\mathbb{R}^d}n(x,\tau)\,dx\,d\tau\\
=&\frac{t^\beta}{\beta\Gamma(\beta)}\int_{\mathbb{R}^d}n_0(x)\,dx.
\end{split}
\]
Then \eqref{vL10} becomes
\[
\int_{\mathbb{R}^d}v(x,t)\,dx=\int_{\mathbb{R}^d}v_0(x)\,dx+\frac{t^\beta}{\beta\Gamma(\beta)}\int_{\mathbb{R}^d}n_0(x)\,dx.
\]
On the other hand, whenever $\gamma>0$, we obtain
\begin{align}\label{vL11}
\begin{split}
\int_{\mathbb{R}^d}v(x,t)\,dx=&\int_{\mathbb{R}^d}E_\beta(t^\beta(-(-\Delta)^{\frac{\alpha}{2}}-\gamma))v_0(x)\,dx\\
&-\int_{\mathbb{R}^d}\int_0^t(t-\tau)^{\beta-1}E_{\beta,\beta}((t-\tau)^\beta(-(-\Delta)^{\frac{\alpha}{2}}-\gamma))\big(\nabla\cdot(uv)\big)\,d\tau\,dx\\
&+\int_{\mathbb{R}^d}\int_0^t(t-\tau)^{\beta-1}E_{\beta,\beta}((t-\tau)^\beta(-(-\Delta)^{\frac{\alpha}{2}}-\gamma))n\,d\tau\,dx.
\end{split}
\end{align}
Since the semigroup property of Mittag-Leffler functions and Lemma \ref{PQinte}, one has
\[
\begin{split}
\int_{\mathbb{R}^d}E_\beta(t^\beta(-(-\Delta)^{\frac{\alpha}{2}}-\gamma))v_0(x)\,dx=&\int_{\mathbb{R}^d}E_{\beta}(-\gamma t^\beta)E_{\beta}(-t^\beta(-\Delta)^{\frac{\alpha}{2}})v_0(x)\,dx\\
=&E_{\beta}(-\gamma t^\beta)\int_{\mathbb{R}^d}v_0(x)\,dx.
\end{split}
\]
Notice that $E_{\beta,\beta}(z)=\beta E'_\beta(z)$ and $E_\beta(0)=1$, then the following identities hold to be true.
\[
\begin{split}
&\int_{\mathbb{R}^d}\int_0^t(t-\tau)^{\beta-1}E_{\beta,\beta}((t-\tau)^\beta(-(-\Delta)^{\frac{\alpha}{2}}-\gamma))n\,d\tau\,dx\\
=&\int_{\mathbb{R}^d}\int_0^t(t-\tau)^{\beta-1}E_{\beta,\beta}(-\gamma(t-\tau)^\beta)E_{\beta,\beta}(-(t-\tau)^\beta(-\Delta)^{\frac{\alpha}{2}})n(x,\tau)\,d\tau\,dx\\
=&\frac{1}{\gamma\Gamma(\beta)}\int_0^t-\beta\gamma(t-\tau)^{\beta-1}E'_\beta(-\gamma(t-\tau)^\beta)\,d\tau\int_{\mathbb{R}^d}n_0(x)\,dx\\
=&\frac{1-E_\beta(-\gamma t^\beta)}{\gamma\Gamma(\beta)}\int_{\mathbb{R}^d}n_0(x)\,dx.
\end{split}
\]
Thus, \eqref{vL11} becomes
\[
\int_{\mathbb{R}^d}v(x,t)\,dx=E_\beta(-\gamma t^\beta)\int_{\mathbb{R}^d}v_0(x)\,dx+\frac{1-E_\beta(-\gamma t^\beta)}{\gamma\Gamma(\beta)}\int_{\mathbb{R}^d}n_0(x)\,dx.
\]
\end{proof}

\begin{rem}
Due to technical reasons that the limitation of the relationship between $p,q$ in $L^p-L^q$ estimates of Mittag-Leffler operators and the nonlocal effect of fractional Laplacian $(-\Delta)^{\frac{\alpha}{2}}$, we cannot obtain the integrability and mass conservation of global mild solution to system \eqref{chemo-fluid-eq} in this paper.
\end{rem}

\subsection{Some properties to the global mild solutions}\label{5.2}

On the basis of Theorem \ref{thm:mildexis}, we can easily have the following corollaries.

\begin{cor} \label{thm:decayofmildsolu}
Suppose that $(n,v,u)$ is the mild solution to system \eqref{chemo-fluid-eq} given by Theorem \ref{thm:mildexis}, which exhibits the following decay behaviors, that is, there exists $C>0$ such that for any $t>0$,
\begin{equation}\label{decay1}
t^{\frac{d\beta}{\alpha}(\frac{2\alpha-2}{d}-\frac{1}{q})}\|n(t)-E_\beta(-t^\beta(-\Delta)^{\frac{\alpha}{2}})n_0\|_{q}\leq C,
\end{equation}
\begin{equation}\label{decay2}
t^{\frac{d\beta}{\alpha}(\frac{\alpha-1}{d}-\frac{1}{r})}\|\nabla v(t)-\nabla E_\beta(t^\beta(-(-\Delta)^{\frac{\alpha}{2}})-\gamma)v_0\|_{r}\leq C,
\end{equation}
\begin{equation}\label{decay3}
t^{\frac{d\beta}{\alpha}(\frac{\alpha-1}{d}-\frac{1}{p})}\|u(t)-E_\beta(-t^\beta(-\Delta)^{\frac{\alpha}{2}})u_0\|_{p}\leq C.
\end{equation}
\end{cor}

\begin{proof}
Obviously, on the basis of \eqref{mildsolu0}, the decay estimate \eqref{decay1}, \eqref{decay2} and \eqref{decay3} can be obtained strictly by \eqref{A2esti}-\eqref{A3esti}, \eqref{Q2esti1}-\eqref{Q3esti} and \eqref{B2esti}-\eqref{0B3esti} in the proof of Theorem \ref{thm:mildexis}, respectively.
\end{proof}

Next, we are going to prove the global stability of mild solution in Theorem \ref{thm:mildexis} under the initial disturbance and the perturbation of external forces.

\begin{cor}\label{thm:stability}
Let the exponents $p,q,r$ be as in Assumption \ref{assum1} and $M$ be as in Theorem \ref{thm:mildexis}. Assume that two initial data $(n_0,v_0,u_0)$ and $(n'_0,v'_0,u'_0)$ and two gravitational potential $\phi$ and $\phi'$ satisfy that
\begin{equation}\label{stainitial1}
\|n_0\|_{\frac{d}{2\alpha-2}}+\|\nabla v_0\|_{\frac{d}{\alpha-1}}+\|u_0\|_{\frac{d}{\alpha-1}}+\|\nabla\phi\|_{d}<M.
\end{equation}
\begin{equation}\label{stainitial2}
\|n'_0\|_{\frac{d}{2\alpha-2}}+\|\nabla v'_0\|_{\frac{d}{\alpha-1}}+\|u'_0\|_{\frac{d}{\alpha-1}}+\|\nabla\phi'\|_{d}<M.
\end{equation}
Suppose that $(n,v,u)$ and $(n',v',u')$ are mild solutions of \eqref{chemo-fluid-eq} given by Theorem \ref{thm:mildexis} with $(n,v,u)|_{t=0}=(n_0,v_0,u_0)$ and $(n',v',u')|_{t=0}=(n'_0,v'_0,u'_0)$. For any $\eta>0$, there is a constant $\delta=\delta(d,p,q,r,\eta)$ such that
\begin{equation}\label{minussta1}
\|n_0-n'_0\|_{\frac{d}{2\alpha-2}}+\|\nabla v_0-\nabla v'_0\|_{\frac{d}{\alpha-1}}+\|u_0-u'_0\|_{\frac{d}{\alpha-1}}+\|\nabla\phi-\nabla\phi'\|_{d}<\delta,
\end{equation}
then we have
\begin{equation}\label{minussta2}
\begin{split}
&\sup_{0<t<\infty}t^{\frac{d\beta}{\alpha}(\frac{2\alpha-2}{d}-\frac{1}{q})}\|n(t)-n'(t)\|_{q}
+\sup_{0<t<\infty}t^{\frac{d\beta}{\alpha}(\frac{\alpha-1}{d}-\frac{1}{r})}\|\nabla v(t)-\nabla v'(t)\|_{r}\\
&~~~+\sup_{0<t<\infty}t^{\frac{d\beta}{\alpha}(\frac{ \alpha-1}{d}-\frac{1}{p})}\|u(t)-u'(t)\|_{p}<\eta.
\end{split}
\end{equation}
\end{cor}

\begin{proof}
Due to the proof of Theorem \ref{thm:mildexis}, we observe that $g$ is a $C^1$-map. Suppose that the initial data and gravitational potential $n_0,v_0,u_0,\phi$, $n'_0,v'_0,u'_0,\phi'$ satisfy \eqref{stainitial1} and \eqref{stainitial2}, respectively.\\ By the continuity of $g$, it holds that for any $\eta>0$, there exists $\delta=\delta(d,p,q,r,\eta)$ such that if
\[
\|[n_0,v_0,u_0,\phi]-[n'_0,v'_0,u'_0,\phi']\|_{\mathcal{X}_M}\leq \delta,
\]
then
\[
\|[n,v,u]-[n',v',u']\|_{\mathcal{Y}_M}\leq \eta.
\]
That completes the proof.
\end{proof}

At the end of this section, for $\gamma=0$, we prove the result in Theorem \ref{self-similarsolu} with respect to the self-similar solution.

\begin{proof}
Based on \eqref{mildsolu0}, we derive the mild solution of \eqref{chemo-fluid-eq} with $\gamma=0$ as follows.
\begin{equation}\label{mildsolu0alpha=2}
\left\{
  \begin{aligned}
  &n(x,t)=E_\beta(-t^\beta (-\Delta)^{\frac{\alpha}{2}})n_0-\int_0^t(t-\tau)^{\beta-1}E_{\beta,\beta}\big(-(t-\tau)^\beta (-\Delta)^{\frac{\alpha}{2}}\big)(u\cdot\nabla n+\nabla\cdot (n\nabla v))\,d\tau,\\
  &v(x,t)=E_\beta(-t^\beta (-\Delta)^{\frac{\alpha}{2}})v_0-\int_0^t(t-\tau)^{\beta-1}E_{\beta,\beta}\big(-(t-\tau)^\beta (-\Delta)^{\frac{\alpha}{2}}\big)(u\cdot\nabla v-n)\,d\tau,\\
  &u(x,t)=E_\beta(-t^\beta (-\Delta)^{\frac{\alpha}{2}})u_0-\int_0^t(t-\tau)^{\beta-1}E_{\beta,\beta}\big(-(t-\tau)^\beta (-\Delta)^{\frac{\alpha}{2}}\big)[P((u\cdot\nabla) u+n\nabla\phi)]\,d\tau.
  \end{aligned}
     \right.
\end{equation}
For initial data $(n_0,v_0,u_0)$ given in Theorem \ref{thm:mildexis}, $j\in \mathbb{N}$, we construct the following iterative sequence
\begin{align*}
\begin{split}
&n_1=E_\beta(-t^\beta (-\Delta)^{\frac{\alpha}{2}})n_0,~v_1=E_\beta(-t^\beta (-\Delta)^{\frac{\alpha}{2}})v_0,~u_1=E_\beta(-t^\beta (-\Delta)^{\frac{\alpha}{2}})u_0,\\
&n_{j+1}=n_1-\int_0^t(t-\tau)^{\beta-1}E_{\beta,\beta}\big(-(t-\tau)^\beta (-\Delta)^{\frac{\alpha}{2}}\big)(u_j\cdot\nabla n_j+\nabla\cdot (n_j\nabla v_j))\,d\tau,\\
&v_{j+1}=\int_0^t(t-\tau)^{\beta-1}E_{\beta,\beta}\big(-(t-\tau)^\beta (-\Delta)^{\frac{\alpha}{2}}\big)(u_j\cdot\nabla v_j-n_j)\,d\tau,\\
&u_{j+1}=\int_0^t(t-\tau)^{\beta-1}E_{\beta,\beta}\big(-(t-\tau)^\beta (-\Delta)^{\frac{\alpha}{2}}\big)[P((u_j\cdot\nabla) u_j+n_j\nabla\phi)]\,d\tau.
\end{split}
\end{align*}
Recall that $K^\alpha_tf(x)=(K_t\ast f)(x)=t^{-\frac{d}{\alpha}}\big(K(t^{-\frac{1}{\alpha}}\cdot)\ast f(\cdot)\big)(x)$ in \eqref{Kalphat}, by the assumption on homogeneity of initial data \eqref{inida}, it is easy to check that the following identities holds,
\begin{equation}\label{semiinit1}
\begin{split}
\lambda^{2\alpha-2}\big(K^\alpha(\lambda^\alpha t)n_0\big)(\lambda x)&=\big(K^\alpha(t)n_0\big)(x),\\
\big(K^\alpha(\lambda^\alpha t)v_0\big)(\lambda x)&=\big(K^\alpha(t)v_0\big)(x),\\
\lambda^{\alpha-1}\big(K^\alpha(\lambda^\alpha t)u_0\big)(\lambda x)&=\big(K^\alpha(t)u_0\big)(x).
\end{split}
\end{equation}
We shall infer that $n_j(x,t)$ are self-similar. First, for $j=1$, with the help of \eqref{MLM} and \eqref{semiinit1}, one has
\begin{gather}\label{semiN0}
\begin{split}
\lambda^{2\alpha-2}n_1(\lambda x,\lambda^{\frac{\alpha}{\beta}}t)&=\lambda^{2\alpha-2}\int_0^\infty M_{\beta}(s)\big(K^\alpha(\lambda^\alpha st^\beta)n_0\big)(\lambda x)\,ds\\
&=\int_0^\infty M_{\beta}(s)\big(K^\alpha(st^\beta)n_0\big)(x)\,ds\\
&=n_1( x,t).
\end{split}
\end{gather}
Proceeding by induction, for $j=2,3,\cdots$, we verify that $\lambda^{2\alpha-2}n_j(\lambda x,\lambda^{\frac{\alpha}{\beta}}t)=n_j(x,t)$. Indeed,
\[
\begin{split}
\lambda^{2\alpha-2}n_j(\lambda x,\lambda^{\frac{2}{\beta}}t)=&\lambda^{3\alpha-2}\int_0^t(t-\tau)^{\beta-1}\big(E_{\beta,\beta}\big(-\lambda^\alpha(t-\tau)^\beta (-\Delta)^{\frac{\alpha}{2}}\big)(u_{j-1}\cdot\nabla n_{j-1})(\lambda^{\frac{\alpha}{\beta}}\tau)\big)(\lambda x)\,d\tau\\
&~+\lambda^{3\alpha-2}\int_0^t(t-\tau)^{\beta-1}\big(E_{\beta,\beta}\big(-\lambda^\alpha(t-\tau)^\beta (-\Delta)^{\frac{\alpha}{2}}\big)\\
&~~~~~~~~~~~~~\times(\nabla\cdot (n_{j-1}\nabla v_{j-1}))(\lambda^{\frac{\alpha}{\beta}}\tau)\big)(\lambda x)\,d\tau\\
=&\int_0^t(t-\tau)^{\beta-1}\big(E_{\beta,\beta}\big(-(t-\tau)^\beta (-\Delta)^{\frac{\alpha}{2}}\big)(u_{j-1}\cdot\nabla n_{j-1})(\tau)\big)(x)\,d\tau\\
&~+\int_0^t(t-\tau)^{\beta-1}\big(E_{\beta,\beta}\big(-(t-\tau)^\beta (-\Delta)^{\frac{\alpha}{2}}\big)(\nabla\cdot (n_{j-1}\nabla v_{j-1}))(\tau)\big)(x)\,d\tau\\
=&n_j( x,t).
\end{split}
\]

In a similar way, we obtain that for $j\in \mathbb{N}$,
\[
v_j(\lambda x,\lambda^{\frac{\alpha}{\beta}}t)=v_j( x,t),~~~~
\lambda^{\alpha-1} u_j(\lambda x,\lambda^{\frac{\alpha}{\beta}}t)=u_j( x,t).
\]
Since $[n,v,u]$ is the limit in $\mathcal{Y}$ of the sequence $[n_j,v_j,u_j]$, it follows that $[n,v,u]$ is the self-similar solution to system (1.1) with $\gamma=0$.

\end{proof}

\textbf{Acknowledgments}
The author thank the anonymous reviewers for careful reading of this manuscript and insightful suggestions. This work is supported by the National Natural Science Foundation of China (Grant No. 12271433).

\section*{References}

\bibliography{frac-KSNS-equ-b-2}

\begin{thebibliography}{10}

\bibitem{MALCAV}
M.~Allen, L.~Caffarelli, and A.~Vasseur.
\newblock Porous medium flow with both a fractional potential pressure and
  fractional time derivative.
\newblock {\em Chin. Ann. Math. Ser. B}, 38(1):45--82, 2017.

\bibitem{Azevedo2022}
J.~Azevedo, C.~Cuevas, J.~Dantas, and C.~Silva.
\newblock On the fractional chemotaxis navier-stokes system in the critical
  spaces.
\newblock {\em Discrete Contin. Dyn. Syst. Ser. B}, 28(1):538--559, 2023.

\bibitem{Azevedo}
J.~Azevedo, C.~Cuevas, and E.~Henriquez.
\newblock Existence and asymptotic behaviour for the time-fractional
  keller--segel model for chemotaxis.
\newblock {\em Math. Nachr.}, 292(3):462--480, 2019.

\bibitem{BnBaCn}
N.~Bellomo, A.~Bellouquid, and N.~Chouhad.
\newblock From a multiscale derivation of nonlinear cross-diffusion models to
  keller--segel models in a navier--stokes fluid.
\newblock {\em Math. Models Methods Appl. Sci.}, 26(11):2041--2069, 2016.

\bibitem{MarioB}
M.~Bezerra, C.~Cuevas, C.~Silva, and H.~Soto.
\newblock On the fractional doubly parabolic keller-segel system modelling
  chemotaxis.
\newblock {\em Sci. China Math.}, 65(9):1827--1874, 2022.

\bibitem{Biler1998}
P.~Biler.
\newblock Local and global solvability of some parabolic systems modelling
  chemotaxis.
\newblock {\em Adv. Math. Sci. Appl.}, 8:715--743, 1998.

\bibitem{BW}
P.~Biler and G.~Wu.
\newblock Two-dimensional chemotaxis models with fractional diffusion.
\newblock {\em Math. Method. Appl. Sci.}, 32(1):112--126, 2009.

\bibitem{bonforte2016}
M.~Bonforte, Y.~Sire, and J.~L. V{\'a}zquez.
\newblock Optimal existence and uniqueness theory for the fractional heat
  equation.
\newblock {\em Nonlinear Anal.}, 153:142--168, 2017.

\bibitem{NBVCal}
N.~Bournaveas and V.~Calvez.
\newblock The one-dimensional keller--segel model with fractional diffusion of
  cells.
\newblock {\em Nonlinearity}, 23(4):923--935, 2010.

\bibitem{JBRGBelin}
J.~Burczak and R.~Granero-Belinch{\'o}n.
\newblock Critical keller--segel meets burgers on: large-time smooth solutions.
\newblock {\em Nonlinearity}, 29(12):3810--3836, 2016.

\bibitem{Caffarelli2007}
L.~Caffarelli and L.~Silvestre.
\newblock An extension problem related to the fractional laplacian.
\newblock {\em Commun. Partial Differ. Equ.}, 32(8):1245--1260, 2007.

\bibitem{Caffarelli2008}
L.~A. Caffarelli and P.~E. Souganidis.
\newblock Convergence of nonlocal threshold dynamics approximations to front
  propagation.
\newblock {\em Arch. Ration. Mech. Anal.}, 195(1):1--23, 2008.

\bibitem{Caffarelli2010}
L.~A. Caffarelli and A.~Vasseur.
\newblock Drift diffusion equations with fractional diffusion and the
  quasi-geostrophic equation.
\newblock {\em Ann. Math.}, pages 1903--1930, 2010.

\bibitem{CalvezLC2008}
V.~Calvez and L.~Corrias.
\newblock The parabolic-parabolic keller-segel model in r2.
\newblock {\em Commun. Math. Sci.}, 6(2):417--447, 2008.

\bibitem{ChaeKangLee}
M.~Chae, K.~Kang, and J.~Lee.
\newblock Existence of smooth solutions to coupled chemotaxis-fluid equations.
\newblock {\em Discrete Contin. Dyn. Syst. Ser. A}, 33(6):2271--2297, 2013.

\bibitem{MCHAEKLEE}
M.~Chae, K.~Kang, and J.~Lee.
\newblock Global existence and temporal decay in keller-segel models coupled to
  fluid equations.
\newblock {\em Commun. Partial Differ. Equ.}, 39(7):1205--1235, 2014.

\bibitem{CHENLIOU}
W.~Chen, C.~Li, and B.~Ou.
\newblock Classification of solutions for an integral equation.
\newblock {\em Commun. Pure Appl. Math.}, 59(3):330--343, 2006.

\bibitem{LCBP2006}
L.~Corrias and B.~Perthame.
\newblock Critical space for the parabolic-parabolic keller--segel model in rd.
\newblock {\em C. R. Math.}, 342(10):745--750, 2006.

\bibitem{de}
P.~M. de~Carvalho-Neto and G.~Planas.
\newblock Mild solutions to the time fractional navier--stokes equations in rn.
\newblock {\em J. Differential Equations}, 259(7):2948--2980, 2015.

\bibitem{DCL}
D.~del Castillo-Negrete, B.~Carreras, and V.~Lynch.
\newblock Fractional diffusion in plasma turbulence.
\newblock {\em Phys. Plasmas}, 11(8):3854--3864, 2004.

\bibitem{KDiethelm}
K.~Diethelm.
\newblock {\em The analysis of fractional differential equations: An
  application-oriented exposition using operators of Caputo type}.
\newblock Springer Berlin, 2010.

\bibitem{DLMarkowich}
R.~Duan, A.~Lorz, and P.~Markowich.
\newblock Global solutions to the coupled chemotaxis-fluid equations.
\newblock {\em Commun. Partial Differ. Equ.}, 35(9):1635--1673, 2010.

\bibitem{E}
C.~Escudero.
\newblock The fractional keller segel model.
\newblock {\em Nonlinearity}, 19(12):2909--2918, 2006.

\bibitem{Fengliliuxu2018}
Y.~Feng, L.~Li, J.-G. Liu, and X.~Xu.
\newblock A note on one-dimensional time fractional odes.
\newblock {\em Appl. Math. Lett.}, 83:87--94, 2018.

\bibitem{LucasCFFJCP}
L.~C. Ferreira and J.~C. Precioso.
\newblock Existence and asymptotic behaviour for the parabolic--parabolic
  keller--segel system with singular data.
\newblock {\em Nonlinearity}, 24(5):1433--1449, 2011.

\bibitem{GuoBL}
B.~Guo, X.~Pu, and F.~Huang.
\newblock {\em Fractional partial differential equations and their numerical
  solutions}.
\newblock World Scientific, 2015.

\bibitem{HuaZhang}
Q.~Hua and Q.~Zhang.
\newblock On the global well-posedness for the 3d axisymmetric incompressible
  keller--segel--navier--stokes equations.
\newblock {\em Z. Angew. Math. Phys.}, 72(5):179, 2021.

\bibitem{HL}
H.~Huang and J.-G. Liu.
\newblock Well-posedness for the keller-segel equation with fractional
  laplacian and the theory of propagation of chaos.
\newblock {\em Kinet. Relat. Models}, 9(4):715--748, 2016.

\bibitem{JLLZhou}
K.~Jiang, Z.~Ling, Z.~Liu, and L.~Zhou.
\newblock Existence, uniqueness and decay estimates on mild solutions to
  fractional chemotaxis-fluid systems.
\newblock {\em Topol. Method. Nonl. An.}, 57(1):25--56, 2021.

\bibitem{ZWJiangLW}
Z.-w. Jiang and L.-z. Wang.
\newblock Weak solutions to the cauchy problem of fractional time-space
  keller--segel equation.
\newblock {\em Math. Method. Appl. Sci.}, 44(18):14094--14113, 2021.

\bibitem{KangKim}
K.~Kang and D.~Kim.
\newblock Existence of generalized solutions for keller-segel-navier-stokes
  equations with degradation in dimension three.
\newblock {\em Math. Eng.}, 4(5):1--25, 2022.

\bibitem{EFKellerSe1}
E.~F. Keller and L.~A. Segel.
\newblock Initiation of slime mold aggregation viewed as an instability.
\newblock {\em J. Theoret. Biol.}, 26(3):399--415, 1970.

\bibitem{EFKellerSe2}
E.~F. Keller and L.~A. Segel.
\newblock Traveling bands of chemotactic bacteria: a theoretical analysis.
\newblock {\em J. Theoret. Biol.}, 30(2):235--248, 1971.

\bibitem{SGSAAKOIM}
A.~A. Kilbas, O.~Marichev, and S.~Samko.
\newblock {\em Fractional integrals and derivatives (theory and applications)}.
\newblock Gordon and Breach, Switzerland, 1993.

\bibitem{Fracdiffeq2006}
A.~A. Kilbas, H.~M. Srivastava, and J.~J. Trujillo.
\newblock {\em Theory and Applications of Fractional Differential Equations}.
\newblock Elsevier, 2006.

\bibitem{Kozono2016}
H.~Kozono, M.~Miura, and Y.~Sugiyama.
\newblock Existence and uniqueness theorem on mild solutions to the
  keller--segel system coupled with the navier--stokes fluid.
\newblock {\em J. Funct. Anal.}, 270(5):1663--1683, 2016.

\bibitem{KOZONOYS2009}
H.~Kozono and Y.~Sugiyama.
\newblock Global strong solution to the semi-linear keller--segel system of
  parabolic--parabolic type with small data in scale invariant spaces.
\newblock {\em J. Differential Equations}, 247(1):1--32, 2009.

\bibitem{LMZH}
B.~Lai, C.~Miao, and X.~Zheng.
\newblock Forward self-similar solutions of the fractional navier-stokes
  equations.
\newblock {\em Adv. Math.}, 352:981--1043, 2019.

\bibitem{TALBIH}
T.~Langlands and B.~Henry.
\newblock Fractional chemotaxis diffusion equations.
\newblock {\em Phys. Rev. E}, 81(5):051102, 2010.

\bibitem{LRZ}
D.~Li, J.~L. Rodrigo, and X.~Zhang.
\newblock Exploding solutions for a nonlocal quadratic evolution problem.
\newblock {\em Rev. Mat. Iberoam.}, 26(1):295--332, 2010.

\bibitem{LJLiu2}
L.~Li and J.-G. Liu.
\newblock A generalized definition of caputo derivatives and its application to
  fractional odes.
\newblock {\em SIAM J. Math. Anal.}, 50(3):2867--2900, 2018.

\bibitem{LJLiu1}
L.~Li and J.-G. Liu.
\newblock Some compactness criteria for weak solutions of time fractional pdes.
\newblock {\em SIAM J. Math. Anal.}, 50(4):3963--3995, 2018.

\bibitem{LLW}
L.~Li, J.-G. Liu, and L.~Wang.
\newblock Cauchy problems for keller--segel type time--space fractional
  diffusion equation.
\newblock {\em J. Differential Equations}, 265(3):1044--1096, 2018.

\bibitem{LYLY}
Y.~Li and Y.~Li.
\newblock Global boundedness of solutions for the chemotaxis-navier--stokes
  system in r2.
\newblock {\em J. Differential Equations}, 261(11):6570--6613, 2016.

\bibitem{LIULorz}
J.-G. Liu and A.~Lorz.
\newblock A coupled chemotaxis-fluid model: Global existence.
\newblock {\em Ann. Inst. H. Poincar{\'e}, Anal. Non Lin{\'e}aire},
  28(5):643--652, 2011.

\bibitem{Lorz2010}
A.~Lorz.
\newblock Coupled chemotaxis fluid model.
\newblock {\em Math. Models Methods Appl. Sci.}, 20(06):987--1004, 2010.

\bibitem{MBSB}
M.~M. Meerschaert, D.~A. Benson, H.-P. Scheffler, and B.~Baeumer.
\newblock Stochastic solution of space-time fractional diffusion equations.
\newblock {\em Phys. Rev. E}, 65(4):041103, 2002.

\bibitem{MiaoYuanZhang}
C.~Miao, B.~Yuan, and B.~Zhang.
\newblock Well-posedness of the cauchy problem for the fractional power
  dissipative equations.
\newblock {\em Nonlinear Anal.}, 68(3):461--484, 2008.

\bibitem{TNRSMU}
T.~NAGAI, R.~Syukuinn, and M.~Umesako.
\newblock Decay properties and asymptotic profiles of bounded solutions to a
  parabolic system of chemotaxis in rn.
\newblock {\em Funkcial. Ekvac.}, 46(3):383--407, 2003.

\bibitem{SakaYama}
K.~Sakamoto and M.~Yamamoto.
\newblock Initial value/boundary value problems for fractional diffusion-wave
  equations and applications to some inverse problems.
\newblock {\em J. Math. Anal. Appl.}, 382(1):426--447, 2011.

\bibitem{Stanvazqez}
D.~Stan and J.~L. V{\'a}zquez.
\newblock The fisher-kpp equation with nonlinear fractional diffusion.
\newblock {\em SIAM J. Math. Anal.}, 46(5):3241--3276, 2014.

\bibitem{Stein}
E.~M. Stein.
\newblock {\em Singular integrals and differentiability properties of
  functions}, volume~2.
\newblock Princeton university press, 1970.

\bibitem{TANZhou2018}
Z.~Tan and J.~Zhou.
\newblock Decay estimate of solutions to the coupled chemotaxis--fluid
  equations in r3.
\newblock {\em Nonlinear Anal. Real World Appl.}, 43:323--347, 2018.

\bibitem{TanZhou}
Z.~Tan and J.~Zhou.
\newblock Global existence and time decay estimate of solutions to the
  keller--segel system.
\newblock {\em Math. Method Appl. Sci.}, 42(1):375--402, 2019.

\bibitem{Taylor1}
M.~Taylor.
\newblock Remarks on fractional diffusion equations.
\newblock \url{https://mtaylor.web.unc.edu/files/2018/04/fdif.pdf}.

\bibitem{TuvalCDW}
I.~Tuval, L.~Cisneros, C.~Dombrowski, C.~W. Wolgemuth, J.~O. Kessler, and R.~E.
  Goldstein.
\newblock Bacterial swimming and oxygen transport near contact lines.
\newblock {\em Proc. Natl. Acad. Sci. USA}, 102(7):2277--2282, 2005.

\bibitem{V1}
J.-L. V{\'a}zquez.
\newblock {\em Nonlinear diffusion with fractional Laplacian operators},
  volume~29.
\newblock Springer, 2012.

\bibitem{V2}
J.-L. V{\'a}zquez.
\newblock Recent progress in the theory of nonlinear diffusion with fractional
  laplacian operators.
\newblock {\em Discrete Contin. Dyn. Syst. Ser. S}, 7(4):857--885, 2014.

\bibitem{WinklerM2010}
M.~Winkler.
\newblock Aggregation vs. global diffusive behavior in the higher-dimensional
  keller--segel model.
\newblock {\em J. Differential Equations}, 248(12):2889--2905, 2010.

\bibitem{Winkler2012}
M.~Winkler.
\newblock Global large-data solutions in a chemotaxis-(navier--) stokes system
  modeling cellular swimming in fluid drops.
\newblock {\em Commun. Partial Differ. Equ.}, 37(2):319--351, 2012.

\bibitem{Winkler2014}
M.~Winkler.
\newblock Stabilization in a two-dimensional chemotaxis-navier--stokes system.
\newblock {\em Arch. Ration. Mech. Anal.}, 211(2):455--487, 2014.

\bibitem{MWinkler2019}
M.~Winkler.
\newblock A three-dimensional keller--segel--navier--stokes system with
  logistic source: global weak solutions and asymptotic stabilization.
\newblock {\em J. Funct. Anal.}, 276(5):1339--1401, 2019.

\bibitem{W}
L.~Wolfersdorf.
\newblock On identification of memory kernels in linear theory of heat
  conduction.
\newblock {\em Math. Method. Appl. Sci.}, 17(12):919--932, 1994.

\bibitem{ZhangLI}
Q.~Zhang and Y.~Li.
\newblock Convergence rates of solutions for a two-dimensional
  chemotaxis-navier-stokes system.
\newblock {\em Discrete Contin. Dyn. Syst. Ser. B}, 20(8):2751--2759, 2015.

\bibitem{ZMSD}
Y.~Zhou, J.~Manimaran, L.~Shangerganesh, and A.~Debbouche.
\newblock Weakness and mittag--leffler stability of solutions for
  time-fractional keller--segel models.
\newblock {\em Int. J. Nonlin. Sci. Num.}, 19(7-8):753--761, 2018.

\bibitem{ZLZ}
S.~Zhu, Z.~Liu, and L.~Zhou.
\newblock Decay estimates for the classical solution of keller--segel system
  with fractional laplacian in higher dimensions.
\newblock {\em Appl. Anal.}, 99(3):447--461, 2020.

\end{thebibliography}
\bibliographystyle{abbrv}

\end{document}